\documentclass{amsart}
\usepackage{enumerate}
\usepackage[dvips]{graphicx}
\usepackage{color}
\usepackage{pstricks}
\usepackage{pst-node}
\newtheorem{theo}{Theorem}[section]
\newtheorem{cor}[theo]{Corollary}
\newtheorem{lem}[theo]{Lemma}

\newtheorem{prop}[theo]{Proposition}

\newtheorem{defn}[theo]{Definition}

\newtheorem{Notation}[theo]{Notation} 
\theoremstyle{definition}
\newtheorem{rem}[theo]{Remark}

\def\a{{\a}}
\def\a{{\mathfrak a}}

\def\C{\mathbb C}

\def\CC{\mathcal C}
\def\F{\mathbb{F}}

\def\Card{\operatorname{Card}}

\def\E{\mathbb E}

\def\H{\mathbb H}

\def\N{\mathbb N}
\def\P{\mathbb{P}}

\def\PP{\mathcal P}

\def\R{\mathbb R}

\def\Re{\operatorname{Re}}

\def\Z{\mathbb{Z}}
\begin{document} 

\title[Orbit measures and intelaced processes]{Orbit measures, random matrix theory and interlaced determinantal processes}

\author{Manon Defosseux}
\address{Laboratoire de Probabilit\'es et Mod\`eles al\'eatoires, Universit\'e Paris 6, 175 rue du Chevaleret, Paris 75013}
\email{manon.defosseux@gmail.com}

\begin{abstract} A connection between representation of compact groups and some invariant ensembles of Hermitian matrices is described. 
We focus on two types of invariant ensembles which extend the Gaussian and the Laguerre Unitary ensembles. We study them using projections and convolutions of invariant probability measures on adjoint orbits of a compact Lie group. These measures are described by semiclassical approximation involving tensor and restriction mulltiplicities. We  show that a large class of them are determinantal. 
\vspace{0.3cm}\\
{\sc R\'esum\'e.} Nous d\'ecrivons les liens unissant les repr\'esentations de groupes compacts et certains ensembles invariants de matrices al\'eatoires. Cet articles porte plus particuli\`erement sur deux types d'ensembles invariants qui g\'en\'eralisent les ensembles gaussiens ou de Laguerre. Nous les \'etudions en consid\'erant des convolutions ou des projections de probabilit\'es invariantes sur des orbites adjointes de groupes de Lie compacts. Par approximation semi-classique, ces mesures sont d\'ecrites par des produits tensoriels ou des restrictions de repr\'esentations. Nous montrons qu'une large classe d'entre elles sont d\'eterminantales. 

\end{abstract}
\maketitle
\section{Introduction}

 One of the first aims of random matrix theory (RMT) is computation of eigenvalue distributions. Its first appearance is in statistics in 1928, when Wishart \cite{Wishart} has determined the maximum likelihood estimator of the covariance matrix of a Gaussian vector. In 1951, Wigner \cite{Wigner} introduced random Hermitian matrices in physics, with the idea that their eigenvalues behave as the high energy levels of hard nucleus. Up to now, in this very active field of research, the detailed analysis of these eigenvalues most often rests on the explicit determinantal expression of their distribution, see, e.g.\ Mehta \cite{Mehta}. Although these distributions are usually obtained by more or less easy applications of the change of variable formula, it has been noticed that they contain expressions familiar to the theory of group representations. Actually, many tools from this theory occur in RMT: for instance Young tableaux, Harish-Chandra-Itzkinson-Zuber formula, symmetric spaces, and so on.

 The purpose of this paper is to establish a direct link between  classical compact groups and RMT and to use it to compute the distributions of some new ensembles. On the one hand it gives  expressions which are maybe not so easy to obtain directly. On the other hand, and  more 
 importantly, it explains  the frequent occurrence of  concepts from  representation theory  in some aspects of random matrix theory.

 The main idea is simple. Roughly speaking, the ensembles we will consider are invariant under the action of a unitary group by conjugacy. Computations will use ultimately a detailed description of the images of the Haar measure on orbits under the adjoint action. They are called orbit measures. In the spirit of  Kirillov's orbit method,  these measures are obtained by semi-classical approximation as limit of the empirical distribution of the weights of irreducible representations of high dimension. For RMT, the quantities of interest will be expressed either by sums or by projections of orbit measures. We will compute them using tensor products or restrictions of representations. This latter computation will be made in a combinatorial-geometric manner, by using Kashiwara crystal theory, which can be viewed as a recent and deep generalization of Young tableaux.  
 
The paper is divided into two parts.  We will describe the theoretical approach only in the second part, because it uses  algebraic machinery which can scare some readers. 

The first part is devoted to its application to concrete problems in RMT. 
To illustrate our approach,  we study some non classical ensembles of Hermitian complex matrices. They will be either the set of all $n \times n$ Hermitian matrices denoted below ${\mathcal P}_n(\C)$, or the set of skew-symmetric Hermitian matrices denoted ${\mathcal P}_n(\R)$, or the set of Hamiltonian Hermitian matrices  denoted  ${\mathcal P}_n(\H)$. The reason of these maybe strange notations is the following. 
 Let $\F=\R, \C$ or $\H$ be either the field of reals or complex  or quaternions numbers. The so called classical compact groups are the neutral components $U_n(\F)$ of the unitary groups. If ${\mathfrak U}_n(\F)$ is the Lie algebra of $U_n(\F)$, then ${\mathcal P}_n(\F)=i {\mathfrak U}_n(\F)$ is a subset of the set of Hermitian matrices with complex entries. 
 They correspond to the so-called classical flat symmetric spaces associated to the complex semisimple groups. In RMT they occur in the  Atland-Zirnbauer classification \cite{Altland}, but among them only ${\mathcal P}_n(\C)$ occurs in the Dyson threefold way \cite{Dyson}.
As usual in physics, we are interested in ensembles invariant under an appropriate group of symmetry. So we will look at random Hermitian matrices in  ${\mathcal P}_n(\F)$ whose laws are invariant under conjugation by the elements of the compact group $U_n(\F)$. Recall that in RMT an ensemble is a random matrix.
\begin{defn}A random matrix, or an ensemble, $M$ with values in ${\mathcal P}_n(\F)$ is called invariant  if its distribution is invariant under conjugation by $U_n(\F)$.    \end{defn}
  It is for these ensembles that representation theory plays a role. Among them, a pre-eminent role is played by the family of ensembles which form a projective series as $n$ increases. Indeed, in physical applications the finite dimension $n$ is only an approximation.  
It is interesting to notice that these series admit a complete description, in the spirit of De Finetti's theorem. We will give it in section 2, by applying  a remarkable result of Pickrell \cite{Pickrell}. They are obtained as a "double mixture" of two simple types of ensembles that we call  GUE($\F$) and LUE($\F$). The classical  GUE and LUE (Gaussian and Laguerre Unitary ensemble) are obtained for $\F=\C$. 
Notice that when $\F=\R$ and $\F=\H$ they are not linked with the GOE and the GSE nor the LOE and LSE. Actually GUE($\R$) is in the class D of Altland and Zirnbauer \cite{Altland}, and GUE($\H$) is in their class C. In the spirit of random matrix theory one can say that they are all in the $\beta=2$ family. Some of their applications will be recalled in \ref{subappli}. 

 As a first application of the introduced method we compute in section 3 the distribution of the main minors of an invariant random matrix. We show that the eigenvalues of the successive main minors of an invariant random matrix in ${\mathcal P}_n(\F)$ with given eigenvalues  have the uniform distribution, or a projection of it, on a conveniently defined  Gelfand--Tsetlin polytope, which describes their interlacing. This was first proved for ${\mathcal P}_n(\C)$ by Baryshnikov \cite{Baryshnikov}, by a different method, motivated by queuing theory. We use the approximation of projections of  orbits detailed in 7.3.
 Notice that the role of Gelfand Tsetlin patterns in the study of shape process already appeared in Cohn, Larsen and Propp \cite{Cohn}.
  
 As a second application, we study the LUE($\F$). These ensembles can be written as $M\Omega M^*$ where $M$ is a standard Gaussian matrix with entries in $\F$ and $\Omega$ is a simple fixed appropriate matrix.  For instance $\Omega=I$ when $\F=\C$. The result of Pickrell mentioned above show that they are a building block in the harmonic analysis of infinite Hermitian matrices in the spirit of Olshanski and Vershik \cite{OlshanskiVershik}, Olshanski \cite{Olshanski3}, Borodin and Olshanski \cite{BorodinOlshanski} for instance. Moreover, we will see that radically new phenomenon occur when $\F$ is not equal to $\C$, so that this study is interesting in itself.
In sections 4 and 5, we determine the distribution of the eigenvalues.  In section 4, one considers the case where $\Omega$ is of rank one and analyse the perturbation of any matrix in $\mathcal{P}_n(\mathbb{F})$ by such a random matrix. This rests on the theoretical results of the second part on tensor products of representations. The general case is considered in section 5.
   
In section 6, considering the minor process associated to some invariant ensembles and successive rank one perturbations, we obtain two types of interlaced point processes, called ``triangular'' and ``rectangular''.  We  deduce from the description given in sections 3, 4 and 5 that a large class of them are determinantal.  This shows that these interlacings exhibit repulsion. In the GUE case, this was also proved recently by Johansson and Nordenstam \cite{JohanssonNordenstam}, and Okounkov and Reshetikhin \cite{Okounkov}.  

After a  first part devoted to applications we develop in the second part  of this paper the tools coming from representation theory used to establish them. In section 7, we present a variant of a theorem of Heckman which allows us to describe in a precise way convolutions  and projections of adjoint orbit measures, once we know the so called tensor or branching rules. For our applications these rules are described in section 9. They are  classical and simple in the case  when $\F=\C$, but more involved in the other ones. Actually we only need a geometrical description of these rules and not their combinatorics as usual. This is quite remarkable. These geometric descriptions are easily and directly obtained using Kashiwara crystal theory. As explained in section 9, crystal theory gives us a description in terms of non-intersecting paths, or interlaced points, which is exactly what we need.  Finally in section 10, we apply the results obtained in this second part to the context of RMT described in the first part.

One can find in the litterature different versions of the theorem of Heckman. For instance, Collins and Sniady gave one recently in \cite{Collins}, in the framework of noncommutative probabilities. Their approach consists in considering a random matrix as a limit  of   random matrices with non-commutative entries. While finishing to write this paper announced in \cite{Defosseux}, Forrester and Nordenstam \cite{ForresterNordenstam} posted an article in arxiv  dealing with the GUE($\R$) case. 
 \begin{Notation} In this paper, for an integer $n$ we will write $$\tilde n=  \left\{\begin{array}{ll } n & \textrm{ when } \F=\C \textrm{ and }\H \\  \lfloor n/2\rfloor & \textrm{ when } \F=\R. \end{array} \right. $$
We let $c=1$ if $\mathbb{F}=\C,\R$, $c=2$ if $\F=\H$ and  $\epsilon=1$ if $n$ is odd and $0$ otherwise.
\end{Notation}

\textit{Acknowlegments:} This research has partly been carried out during a visit at the Boole Centre for Research in Informatics, University College Cork. The author would like to thank Ton Dieker,  Anthony Metcalfe, Neil O'Connell, Jon Warren and her advisor Philippe Bougerol for many helpful and illuminating discussions.
  
\part{\sc Random matrices}
    \section {Ensembles of Hermitian matrices}
   \subsection {Some invariant set of Hermitian matrices}
  The set ${\mathcal P}_n(\C)$  of  $n\times n$ Hermitian matrices is the real vector space of complex matrices $M$ such that $M^*=M$, where $M^*$ is the adjoint of $M$. 
Many classical ensembles considered in physics occur on subsets of ${\mathcal P}_n(\C)$. Let us distinguish three important classes which occur as flat symmetric spaces associated with compact groups, or equivalently complex semi-simple groups, and are thus of the so called $\beta=2$ type. They have been introduced in the literature under various names. Our choice is due to the fact that we want to have a common setting for all of them.

  The first set we consider is ${\mathcal P}_n(\C)$ itself.
   The second set is the set ${\mathcal P}_n(\R)$ of Hermitian complex matrices $M$ which can be written as $M=iX$ where $X$ is a real matrix. In this case $X$ is skewsymmetric (i.e. $X+X^t=0$). Thus ${\mathcal P}_n(\R)$ is just a convenient parametrization of the set of skewsymmetric real matrices, studied for instance by Mehta \cite{Mehta}.

In order to introduce the third one, we first define the $C$-symmetry class of Atland and Zirnbauer \cite{Altland}. It is the set of complex Hermitian matrices ${\mathcal H}$ which can be written as \begin{equation}\label{classC}{\mathcal H}=\begin{pmatrix}H& S  \\
S^* & -\bar H
\end{pmatrix}\end{equation}
where $H$ and $S$ are two $n\times n$ complex matrices, with $H$ Hermitian and $S$ symmetric. In other words it is the set of Hermitian matrices of the Lie algebra of the complex symplectic group.
One recognizes the form of the Bogoliubov--de Gennes Hamiltonian  in condensed matter physics (see below). 
Actually we will use a more convenient representation by using quaternions. For us, the set $\H$ of quaternion is just the set of $2 \times 2$ 
  matrix $Z$ with complex entries which can be written as 
  $$ Z=
\begin{pmatrix}
  a& b   \\
 -\bar b &  \bar a   
\end{pmatrix},
  $$
  where $a,b\in \C$. Its conjugate  $Z^*$  is the usual adjoint of the complex matrix $Z$. We define ${\mathcal P}_n(\H)$ as the set of $2n\times 2n$ complex Hermitian matrices $M$ which can be written as $M=iX$ where $X$ is a 
$n\times n$  matrix with quaternionic entries.
 Let $W$ be the matrix of the permutation of $\C^{2n}$:  $$(x_1,x_2,\cdots) \mapsto (x_1,x_{n+1},x_2,x_{n+2},x_3,\cdots).$$   Then ${\mathcal H}$ is an Hamiltonian given by (\ref{classC}) if and only if \begin{equation}\label{eqW}\tilde  {\mathcal H}=W{\mathcal H}W^{-1}\end{equation} is in $\PP_n(\H)$.  Therefore $\PP_n(\H)$ is just a parametrization of the class $C$ of Altland and Zirnbauer. Notice that the matrices of the GSE are not of this type since they are self dual matrices with entries in $\H$. We can thus define:
 \begin{defn} For $\F=\R,\C,\H$,  $\mathcal{P}_n(\F)$ is the set of $n\times n$ Hermitian matrices with entries in $i\F$.
 \end{defn}
  One recognizes in $\mathcal{P}_n(\F)$ the three infinite families of Cartan motion groups associated with compact (or complex) groups. Indeed, let $U_n(\F)$ be the neutral component of the group of unitary matrices with entries in  $\F$. Its Lie algebra  ${\mathfrak U}_n(\F)$ is the set of matrices $M$ with entries in $\F$ such that $M+M^*=0$. Then ${\mathcal P}_n(\F)=i {\mathfrak U}_n(\F)$, and  the Cartan motion group associated with $U_n(\F)$ is 
 $$G= U_n(\F) \times_\sigma {\mathcal P}_n(\F)$$
 where  $U_n(\F)$ acts on $ {\mathcal P}_n(\F)$ through $\sigma$ by conjugation (i.e. by adjoint action). 
 In the classification of symmetric spaces, $\PP_n(\C)$ is said to be of type A  and $U_n(\C)$ is the unitary group. When  $n=2r$, $\mathcal{P}_n(\R)$ is of type D and  when  $n=2r+1$, $\PP_n(\R)$ is of type B, and $U_n(\R)$ is the special orthogonal group $SO(n)$ in both cases. At last, $\PP_n(\H)$ is of type C and $U_n(\H)$ is the symplectic unitary group $Sp(n)$.

   \subsection{Eigenvalues and radial part}
   
   Consider a matrix $M$ in ${\mathcal P}_n(\F)$. Since $M$ is an Hermitian complex matrix, it has real eigenvalues $\lambda_1 \geq \lambda_2 \geq \cdots \geq \lambda_n$ when $\F=\R$ and $\C$, and $\lambda_1 \geq \lambda_2 \geq \cdots \geq \lambda_{2n}$ when $\F=\H$.  When $\F=\C$ there is no further restriction, but when $\F=\R$, then  $\lambda_{n-k+1}=-\lambda_{k}$, for $k=1,\cdots, \tilde n+1$, which implies $\lambda_{\tilde n+1}=0$ when $n$ is odd (Recall that  $\tilde{n}=[n/2]$ when $\F=\R$). When $\F=\H$ then $\lambda_{2n-k+1}= -\lambda_{k}$, for $k=1,\cdots,n$. We define the Weyl chambers $\CC_n$ in the different cases by :
    when  $\F=\C$,
   $${\CC}_n=\{\lambda\in \R^n; \lambda_1  \geq \lambda_2 \geq \cdots \geq \lambda_n\},$$
   
 \noindent   when $\F=\R$, and $n$ is odd, 
   $${\CC}_n=\{\lambda\in \R^{\tilde n}; \lambda_1  \geq \lambda_2 \geq \cdots \geq \lambda_{\tilde n} \geq 0\}, $$
   
 \noindent    when $\F=\R$, and $n$ is even (see Remark \ref{remstrange}),
   $${\CC}_n=\{\lambda\in \R^{\tilde n}; \lambda_1  \geq \lambda_2 \geq \cdots \geq \lambda_{\tilde n-1}  \geq |\lambda_{\tilde n}| \geq 0\}, $$
   
 \noindent   when $\F=\H$,
     $${\CC}_n=\{\lambda\in \R^n; \lambda_1  \geq \lambda_2 \geq \cdots \geq \lambda_n \geq 0\}.$$

The Weyl chamber is a fundamental domain for the adjoint action of $U_n(\F)$ on $\mathcal{P}_n(\F)$. More precisely, let us introduce the following matrices.  
   For $\F=\mathbb{C}$, and $\lambda=(\lambda_1,\cdots,\lambda_n)$ in $\mathbb{R}^n$, we denote  by $\Omega_n(\lambda) $ the $n \times n$ diagonal matrix 
$$\Omega_n(\lambda)=\begin{pmatrix}
   \lambda_1  &   &   \\
  & \ddots  &   \\
  &   &   \lambda_{ n} 
\end{pmatrix}.
$$
When $\F=\mathbb{R}$,  we let $\omega(\alpha)=\begin{pmatrix} 0 &  i\, \alpha  \\
-i\, \alpha &  0       
\end{pmatrix}$ where $\alpha \in \R$, 
  and for $\lambda \in \mathbb{R}^{\tilde n}$,  we write $\Omega_n(\lambda)$ for the $n \times n$ block-diagonal matrix  given by, when  $n$ is even,  $$
\Omega_n(\lambda)=\begin{pmatrix}
   \omega(\lambda_1)  &   &   \\
  & \ddots  &   \\
  &   &   \omega(\lambda_{\tilde n}) 
\end{pmatrix},
$$  and when $n$ is odd,
 $$ \Omega_n(\lambda)= \begin{pmatrix}
 \omega(\lambda_1) &   &    &\\
  & \ddots  &   & \\
  &   &  \omega(\lambda_{\tilde n}) & \\
  & & & 0
\end{pmatrix}.$$
When $\F=\H$ and $\lambda=(\lambda_1,\cdots,\lambda_n)$ in $\mathbb{R}^n$, we let 
$$\Omega_n(\lambda)= \begin{pmatrix}
   Z(\lambda_1)  &   &   \\
  & \ddots  &   \\
  &   &  Z( \lambda_{ n} )
\end{pmatrix}
$$
where, for $\alpha \in \R$, $Z(\alpha)$ is the $2\times 2$ matrix  $Z(\alpha)=\begin{pmatrix}
  \alpha &  0  \\
 0 &  - \alpha   
\end{pmatrix}$. Then it is well known and not difficult to prove that:
   \begin{lem} \label{KAK}
   Let $M$ be a matrix in ${\mathcal P}_n(\F)$. Then there exists a unique $\lambda\in \CC_n$ and a matrix $U\in U_n(\F)$ such that
   $$M=U\Omega_n(\lambda)U^*.$$
We call $\lambda$  the radial part of $M$ and will denote it by $\lambda=X^{(n)}(M)$.
   \end{lem}
     This is the so called radial decomposition of $M$ in the flat symmetric space ${\mathcal P}_n(\F)$. We  see that in each case 
    $$\{kMk^*, k \in U_n(\mathbb{F})\} \cap \{\Omega_n(\mu), \mu \in \CC_n\}=\{\Omega_n(\lambda)\}.$$
    \begin{rem}\label{remstrange} The definition of $\CC_n$ when $\F=\R$ and $n$ is even may look strange. Actually, in this case both $\lambda_{\tilde n}$ and $-\lambda_{\tilde n}$ are eigenvalues, so $\lambda_1,\cdots,\lambda_{{\tilde n}-1}, |\lambda_{\tilde n}|$ is the set of positive eigenvalues. But one has to take this  $\CC_n$ to have the lemma above.    \end{rem}

   \subsection{Infinite invariant ensembles}
 We have defined  an invariant  random matrix (or invariant ensemble) in ${\mathcal P}_n(\F)$ as a random matrix with values in $\mathcal{P}_n(\F)$, whose distribution is invariant under conjugation by $U_n(\F)$.
   There are of course many such matrices. Actually it is well known that one has the following lemma.
   
\begin{lem} \label{invmeasure} A random matrix  $M$  with value in $\mathcal{P}_n(\F)$ is invariant if and only if it can be written as $M=U\Omega_n(\Lambda)U^*$, where $U\in U_n(\F)$ and $\Lambda\in {\CC_n}$ are independent random variables, $U$ having the Haar distribution.
\end{lem} 
\proof The lemma \ref{KAK} allows us to write $M=U\Omega_n(\Lambda)U^*$, with $U\in U_n(\F)$ and $\Lambda\in \CC_n$. Let $V\in U_n(\F)$ be a Haar distributed random variable independent of $M$. Then $M$ as the same distribution as $(VU)\Omega_n(\Lambda)(VU)^*$. The Haar measure being invariant by multiplication, this has the same law as $V\Omega_n(\Lambda)V^*$.   \qed\medskip
  
    Two important classes of invariant random matrices in ${\mathcal P}_n(\F)$ are to be distinguished. The first one is the class of ergodic measures. An invariant probability is called ergodic if it cannot be written as a barycenter of other invariant probabilities.
    On ${\mathcal P}_n(\F)$ the ergodic invariant measures are the  orbit measures, that is the law of $U\Omega_n(\lambda)U^*$ when $U$ has the Haar distribution and $\lambda$ is fixed in $\CC_n$. The second class is linked with Random Matrix Theory. Actually, in that case one is interested in a family $\nu_n$ of probability measures on ${\mathcal P}_n(\F)$ which forms a projective system as $n$ groes, and thus defines a probability measure $\nu$ on the set ${\mathcal P}_\infty(\F)$ of infinite Hermitian matrices. More precisely, for $\F=\R,\C$ or $\H$, let ${\mathcal P}_\infty(\F)$ be the set of matrices $\{M_{k,l}, 1\leq k, l <Ê+\infty\}$, with entries in $i\F$ such that $M_{l,k}=M_{k,l}^*$. For each $n\in \N$, $U_n(\F)$ acts on ${\mathcal P}_\infty(\F)$. 
A
probability measure on ${\mathcal P}_\infty(\F)$ is called invariant if it is invariant under the action of each $U_n(\F)$.

It is remarkable that, following Pickrell \cite{Pickrell} and Olshanski and Vershik \cite{OlshanskiVershik}, one can describe explicitly the set of invariant measures. As in De Finetti's description, each of this measure is obtained as a mixture of ergodic ones, and each ergodic one has a product structure : the diagonal elements form an i.i.d. sequence (see below). In order to describe them let us introduce the two basic ensembles on ${\mathcal P}_\infty(\F)$. 

Let us denote ${\mathcal M}_{n,m}(\F)$ the set of $n \times m$  matrices with entries in $\F$. It is a real vector space. 
We put on it the Euclidean structure defined by the scalar product, 
$$\langle M, N \rangle =a\Re \mbox{tr}(MN^*), \quad M,N\in {{\mathcal M}_{n,m}(\F)},$$
where $a=1$ for $\mathbb{F}=\mathbb{R}$, and  $a=2$  for $\mathbb{F}=\mathbb{C},\mathbb{H}$.
Recall that a standard Gaussian variable on a real Euclidean space with finite dimension $d$ is  a random variable with density $$x\mapsto (2\pi)^{-d/2} e^{-\langle x, x\rangle/2}.$$
Our choice of the Euclidean structure above defines a notion of standard Gaussian variable on  ${{\mathcal M}_{n,m}(\F)}$.   Taking $m=n=1$ this defines standard Gaussian variables in $\F$ itself. We equip   the real vector space ${\mathcal P}_n(\F)$ with the scalar product $$\langle M,N\rangle=b \, \mbox{tr}(MN), \quad M,N\in \mathcal{P}_n(\F),$$ where $b=1$ when $\F=\C$ and $b=1/2$ when $\F=\R,\H$, and thus  define a standard Gaussian variable on $\mathcal{P}_n(\F)$.

We have defined above, for each choice of $\F$, the matrix $\Omega_n(\lambda)$ for $\lambda \in \CC_n$. For $k \leqÊ\tilde n$, we let
\begin{align} \label{omegank}\Omega_n^k=\Omega_n(1,\cdots,1,0,\cdots 0)\end{align}
where $1$ appears $k$ times, and, when $1$ appears $\tilde n$ times,  we let \begin{align}\label{omegan} \Omega_n=\Omega_n^{\tilde{n}}=\Omega_n(1,...,1).\end{align}
\begin{defn}\label{def_LUE}For $\F=\R,\C$ or $\H$, and $k,n \in \N$, we define 

1. The ensemble GUE$_n(\F)$ as the set of matrices in ${\mathcal P}_n(\F)$ with the standard Gaussian distribution. 

2. The ensemble LUE$_{n,k}(\F)$ as the set of matrices  
$M\Omega_k M^*$
when $M$ is a standard Gaussian random variable in ${\mathcal M}_{n,k}(\F)$.  
\end{defn}
Notice that if the  matrices of the LUE$_{n,k}(\F)$ may look strange, their Fourier transform does not (recall that $\tilde k =k$ when $\F=\C,\H$ and $\tilde{k}=[k/2]$ when  $\F=\R$, and that $c=1$ when $\F=\C,\R$ and $c=2$ when $\F=\H$):
\begin{lem}\label{FourierLUE} Let $M$ be a standard Gaussian random variable in ${\mathcal M}_{n,k}(\F)$. Then the Fourier transform of $M\Omega_kM^*$ is given by
$$\E(e^{-i\langle N, M\Omega_kM^* \rangle})=  \det(I+\frac{i}{c}N)^{-{\tilde k}}, \quad N\in \PP_n(\F).$$
\end{lem}
\proof As $M\Omega_kM^*$ is invariant, it is enough to prove the identity for $N=\Omega_n(\lambda)$, with $\lambda\in \CC_n$. When $\mathbb{F}=\mathbb{C}$, $$\langle N, M\Omega_k M^*\rangle= \sum_{i=1}^n\sum_{j=1}^k \lambda_i\vert M_{i,j}\vert^2,$$ where $M_{i,j}$ are independent standard Gaussian complex r.v.. We have, for all $\alpha\in \R$, $$\mathbb{E}( e^{-i\alpha\vert M_{1,1}\vert^2})=\frac{1}{1+i\alpha},$$
which gives the complex case. When $\mathbb{F}=\mathbb{H}$, $$\langle N, M\Omega_k M^*\rangle=\sum_{i=1}^n\sum_{j=1}^j \lambda_i(\vert a_{i,j}\vert^2-\vert b_{i,j}\vert^2),$$ where the matrices $\begin{pmatrix}
a_{i,j} & -\bar{b}_{i,j}\\
b_{i,j}&\bar{a}_{i,j}
\end{pmatrix}$ are independent standard Gaussian variables in $\mathbb{H}$. We have,  $$\mathbb{E}( e^{-i\alpha(\vert a_{1,1}\vert^2-\vert b_{1,1}\vert^2)})=\frac{1}{1+\big(\frac{\alpha}{2}\big)^2},$$
which gives the quaternionic case. When $\mathbb{F}=\mathbb{R}$, $$\langle N, M\Omega_k M^*\rangle= \sum_{i=1}^{\tilde n}\sum_{j=1}^{\tilde k}  \lambda_i(M_{2i,2j-1}M_{2i-1,2j} -M_{2i-1,2j-1}M_{2i,2j}),$$ where the $M_{i,j}$'s are independent standard real  Gaussian variables. We have $$\mathbb{E}( e^{-i\alpha(M_{2,1}M_{1,2} -M_{1,1}M_{2,2})})=\frac{1}{1+\alpha^2},$$
which gives the real case.  \qed\medskip

When $\F=\C$ we obtain the classical LUE, called the Laguerre Unitary or Complex Wishart ensemble, which is carried by  the cone of positive definite matrices. The situation is completely different for the fields $\R$ and $\H$: in these cases the Fourier transform
$$\det(I+\frac{i}{c}N)^{-{\tilde k}}=\det(I+\frac{N^2}{c^2})^{-{\tilde k/2}}$$
is real, and therefore the distribution of a random matrix of the LUE($\F$) is symmetric. Actually the support of $M\Omega_nM^*$ is the whole of $\mathcal{P}_n(\mathbb{F})$. Observe that in the cases when $\F=\H$ and $\F=\R$ with $n$ odd, all the invariant measures on $\mathcal{P}_n(\F)$ are symmetric.

Let us give a justification  for the introduction of these invariant ensembles.
We define the set LUE$_{\infty}^1(\F)$ as the set of matrices $M\Omega_\infty^1M^*$ with $M\in\mathcal{M}_{\infty}(\F)$ such that  the submatrices $\{M_{i,j}, \, i,j=1,\cdots, n\}$ are standard Gaussian variables in $\mathcal{M}_{n}(\F)$  and the set GUE$_{\infty}(\F)$ as the set of matrices $M\in\mathcal{P}_{\infty}(\F)$ such that the submatrices $\{M_{i,j}, \, i,j=1,\cdots, n\}$ are   standard Gaussian variables in $\mathcal{P}_{n}(\F)$. 
 A random matrix in ${\mathcal P}_\infty(\F)$ is called invariant if its law is invariant under the action of each $U_n(\F)$. As will be clear from the proof, the following theorem is essentially contained in Pickrell \cite{Pickrell}. It can be useful to notice that the intuition behind this result is the fact that limit of orbit measures are of this type, by Borel's theorem \ref{HaartoGauss} recalled below.

\begin{theo} \label{ergodicinfinity}
Each ergodic invariant random matrix $M$ in ${\mathcal P}_\infty(\F)$ is  sum of elements of  GUE$_\infty(\F)$ and LUE$_{\infty}^1(\F)$: it can be written as $$M=aI+bG+\sum_{k=1}^{+\infty}d_kL_k$$
where $G$ belongs to GUE$_\infty(\F)$, $L_k$ belongs to LUE$_{\infty}^1(\F)$, the random variables $G,L_1,L_2\cdots$ are independent, and $a,b, d_k$ are constants such that $\sum d_k^2 < +\infty$, $I$ is the identity matrix. Moreover $a=0$ when $\F=\R$ and $\F=\H$.
\end{theo}

\proof The proof will use Olshanski spherical pairs, see Olshanski \cite{Olshanski1,Olshanski2} or Faraut \cite{Faraut}. Given a topological group $G$ and a closed subgroup $K$, one says that $(G,K)$ is an Olshanski spherical pair if for each irreducible unitary representation $\pi$ of $G$ in an Hilbert space $H$, the space $\{u\in H; \pi(k)u=u, \mbox{ for all } k\in K\}$ is zero or one dimensional. 
For instance, an inductive limit of Gelfand pairs is an Olshanski pair.

The dual of the vector space ${\mathcal P}_\infty(\F)$ is the inductive limit of the ${\mathcal P}_n(\F)$'s, which is the set ${\mathcal P}^{(\infty)}(\F)$ of matrices $M$ in ${\mathcal P}_\infty(\F)$ such that $M_{i,j}=0$ for $i+j $ large enough. Each $U_n(\F)$ acts on ${\mathcal P}^{(\infty)}(\F)$ as on ${\mathcal P}_\infty(\F)$.
Let $U^{(\infty)}(\F)$ be the group inductive limit of the $U_n(\F)$'s. Recall the radial decomposition $\PP_n(\F)=\{U\Omega_n(\lambda)U^*, U\in U_n(\F), \lambda \in \CC_n\}$. Let $\lambda=(\lambda_1,\lambda_2,\cdots) $ be  an infinite  sequence of real numbers, with $\lambda_k=0$ for $k$ large enough. In this case, we write $\Omega(\lambda)$ instead of $\Omega_\infty(\lambda)$. Notice that each matrix of ${\mathcal P}^{(\infty)}(\F)$ can be written as $U\Omega(\lambda)U^*$ for an $U\in U^{(\infty)}(\F)$ and such a $\lambda$. 

As inductive limit of Gelfand pairs  $({U}^{(\infty)} (\F)\times_{\sigma}{\mathcal P}^{(\infty)}(\F) ,  {U}^{(\infty)}(\F))$ is an Olshanski pair. 
Therefore, by the so-called  multiplicative property of Voiculescu  and Olshanski (see Olshanski \cite{Olshanski2}, Pickrell \cite{Pickrell}) an invariant probability measure $\nu$ on ${\mathcal P}_\infty(\F)$ is ergodic if and only if its Fourier transform $\psi$ on ${\mathcal P}^{(\infty)}(\F)$ is a positive definite invariant function such that, for some function $\phi:\R\to \C$,
$$\psi(\Omega(\lambda_1,\lambda_2,\dots))=\phi(\lambda_1)\phi(\lambda_2)\cdots$$
for all $\lambda$ as above. 
When $\F=\C$ it is proved in Pickrell \cite{Pickrell} (see also  Olshanski and Vershik \cite{OlshanskiVershik}) that there exist unique real numbers $a,b \geq 0$ and $d_k,k \geq1$, such that for all $t \in \R$,
\begin{equation}\label{psidef}
\phi(t)=e^{ia t}e^{-bt^2}\prod_{k=1}^\infty [(1+id_kt)e^{id_kt}]^{-1}.\end{equation}
 Therefore the theorem holds when $\F=\C$.
We now consider the case where $\F=\R$, 
following  an idea in Pickrell \cite{Pickrell}. To any complex matrix $M\in \PP^{(\infty)}(\C)$ we associate a  matrix $f(M)\in \PP^{(\infty)}(\R)$ by replacing each entry $m=x+iy$,  $x, y\in \R$ of $M$ by the $2\times 2$ matrix 
$$ \tilde m=\begin{pmatrix}
iy& ix  \\
-ix  & iy
\end{pmatrix}.$$
For all $\lambda=(\lambda_1,\lambda_2,\cdots)$, 
$f(\Omega(\lambda))=\Omega(\lambda)$
where $\Omega$ is, on  the left hand side, the one defined for $\F=\C$ and  on the right hand side the one defined for $\F=\R$.

Consider an ergodic invariant probability measure on ${\mathcal P}_\infty(\R)$ and let $\psi$ be its Fourier transform defined on ${\mathcal P}^{(\infty)}(\R)$. Then $\psi$ is invariant and  positive definite  and by the multiplicativity theorem there exists, as above,  a function $ \phi:\R\to \C$ such that 
$$\psi(\Omega(\lambda_1,\lambda_2,\dots))=\phi(\lambda_1)\phi(\lambda_2)\cdots$$
The function $\psi \circ f$ on ${\mathcal P}^{(\infty)}(\C)$ is obviously positive definite and invariant. Moreover, since $f\circ \Omega=\Omega$ one has 
$$(\psi \circ f)(\Omega(\lambda_1,\lambda_2,\cdots))=\phi(\lambda_1)\phi(\lambda_2)\cdots$$
Therefore by the sufficient condition of the multiplicativity theorem, $ \psi \circ f$ is the Fourier transform of an ergodic invariant probability measure on ${\mathcal P}_\infty(\C)$. Thus $\phi $ can be written as (\ref{psidef}) above. Moreover, the function $\psi$ is invariant under the groups $U_n(\R)=SO(n)$.  Using the adequate reflection in $SO(3)$, we see that 
$$\psi(\Omega(\lambda_1,\lambda_2,\cdots))=\psi(\Omega(-\lambda_1,\lambda_2,\cdots)).$$
Therefore, for all $t\in \R$, 
$\phi(t)=\phi(-t)$ which implies by uniqueness that, 
$$\phi(t)=e^{-bt^2}\prod_{k=1}^\infty [(1+(d_kt)^2)]^{-1}.$$
Using the expression of the Fourier transform given in Lemma \ref{FourierLUE}, we obtain the theorem in the case $\F=\R$.
When $\F=\H$ the proof is similar: one uses the map $\tilde f:\PP^{(\infty)}(\C)\to \PP^{(\infty)}(\H)$ defined in the following way. First, we define $\tilde f_n:\PP_n(\C)\to \PP_n(\H)$, by when $M\in \PP_n(\C)$ and  $W$ is given by (\ref{eqW}),
$$\tilde  f_n(M)=
   W
\begin{pmatrix}M& 0  \\
0 & - \bar M 
\end{pmatrix} 
W^* \in \PP_n(\H).$$   For $M\in \mathcal{P}^{(\infty)}(\H)$, let $\pi_n(M)$ be its main minor of order $n$. The fact that $\pi_n(\tilde{f}_{n+1}(\pi_{n+1}(M)))=\tilde{f}_n(\pi_n(M))$ allows us to define $\tilde f: M\in\PP^{(\infty)}(\C)\to\tilde f (M)\in \PP^{(\infty)}(\H)$ by $\pi_n(\tilde f(M))=\tilde f_n(\pi_n(M))$. We also have $\tilde  f \circ \Omega= \Omega$. The symmetry $\lambda\mapsto -\lambda$ given by the action of  $U_n(\H)$ allows us to conclude as  when $\mathbb{F}=\mathbb{R}$.

  \subsection{Symmetry classes and some applications}\label{subappli}

Let us recall  the three main historical steps in the description of the ten symmetry classes, i.e.\ series of classical symmetric spaces, in physical applications of RMT, see   Atland and Zirnbauer \cite{Altland}, Caselle and Magnea \cite{Caselle}, Forrester \cite{ForresterBook}, Heinzner, Huckleberry, and  Zirnbauer \cite{Heinzner}. We refer to the
symmetry classes by Cartan's symbol for the symmetric
space corresponding to their Hamiltonians. The first step is the introduction of the "threefold way" by Dyson \cite{Dyson} in 1962 where are defined the GUE (class A), 
the GOE (class AI) and the GSE (class AII), often called the Wigner--Dyson classes. They describe for instance single particle excitations in the presence of a random potential. In the 90's, Altland and Zirnbauer \cite{Altland} have defined the classes  BD, C, DIII and CI to describe mesoscopic normal-superconducting hybrid structures: for instance a normalconducting quantum dots in contact with two superconducting regions. They are sometimes called the Altland--Zirnbauer,  or Bogoliubov--de Gennes, or Superconductor classes. At least chiral classes AIII (LUE), BDI (LOE) and CII (LSE) were introduced  by Verbaarschot  \cite{Verbaarschot} to describe Dirac fermions or systems with purely off-diagonal disorder, as in random flux models. The explicit description of the distribution of the eigenvalues in all these classes is given for instance in Forrester \cite{ForresterBook} or in Eichelsbacher and Stolz \cite{Eichelsbacher}.

In our ensembles $\beta=2$, and thus only the classes A,B,C,D,AIII occur. 
Let us for instance recall rapidly how the new classes C and D appear in quantum mechanics. Dynamics of the systems of the Wigner Dyson class is given in term of second quantization. For the superconductor classes, one convert this set up into  first quantization by using the Bogoliubov--de Gennes Hamiltonian. As explained in Atland and Zirnbauer \cite{Altland}, this Hamiltonian acts on a $2n$-dimensional Hilbert space by a complex Hermitian matrix ${\mathcal H}$ which can be written as 
$${\mathcal H}=\begin{pmatrix}H& \Delta  \\
-\bar \Delta & -H^t
\end{pmatrix}$$
where $H$ and $\Delta$ are $n\times n$ matrices. 
Let $U_0$ be the $2n\times 2n$ unitary matrix, block diagonal with each diagonal block equal to
$$u_0=\frac{1}{\sqrt{2}}\begin{pmatrix}1&1  \\
i  & -i
\end{pmatrix},$$
in other words $U_0=u_0\otimes I_n$. Then $X=U_0{\mathcal H}U_0^{-1}$ is in $\PP_{2n}(\R)$ and each matrix in $\PP_{2n}(\R)$ is of this form. This shows that $\PP_{2n}(\R)$ is a parametrization of the class $D$ of Altland and Zirnbauer.

If we add spin rotation invariance the BdG Hamiltonian can be written (see \cite{Altland}) as two commuting subblocks of the form
$${\mathcal H}=\begin{pmatrix}H_1& H_2  \\
H_2^* & -\bar H_1
\end{pmatrix}.$$
This is the class $C$ of Altland Zirnbauer. As seen above,  $\PP_n(\H)$  describes this set.
Notice also that the GUE($\R$), or equivalently the antisymmetric case,  was already studied  as soon as 1968 by Rozenbaum and Mehta in \cite{Mehta_Roz}. Recently it also occurs for instance in Cardy \cite{Cardy} and Brezin et al. \cite{Brezin, Brezin2} for instance.

When $\F=\R$ and $\H$, the eigenvalues of the matrices in $\PP_n(\F)$ come in pairs symmetric with respect to the origin (this is sometimes linked with Kramers degeneracy). So in a sense there is a presence of a wall at $0$. This often explains their occurences in applications, see for instance  Krattenthaler et al.\ \cite{Krattenthaler}, Katori et al.\ \cite{Katori1,Katori2}, Gillet \cite{Gillet}, Forrester and Nordenstam \cite{ForresterNordenstam}.
The LUE($\C$) is in a chiral class, but not the LUE($\R$) nor the LUE($\H$) which appear to be new and for which, we are not aware of any physical application.

\section{Minors and Gelfand--Tsetlin polytopes}
In this section, we compute the joint distribution of the main minors of invariant random matrices in $\mathcal{P}_n(\mathbb{F})$.  For $M=\{M_{i,j}, \, 1\le i,j\le n\}$ in $\mathcal{P}_n(\mathbb{F})$ and $k \leq n$, the main minor of order $k$ of $M$, is the submatrix
$$\pi_k(M)=\{M_{ij}, \, 1\le i,j\le k\}$$
(this is not the standard definition of a minor: usually it is the determinant of a submatrix of $M$, and not the submatrix itself). The main minor of order $k$ of $M$ belongs to $\mathcal{P}_k(\mathbb{F})$, so we can consider its radial part denoted $X^{(k)}(M)$.  Considering the radial parts of all the main minors of an invariant random matrix in $\mathcal{P}_n(\mathbb{F})$, we get  a random variable, 
$$X(M)=(X^{(1)}(M),\cdots,X^{(n)}(M)),$$ which is, when $\mathbb{F}$ is equal to $\mathbb{C}$, and $M\in \PP_n(\C)$ is a matrix from the GUE, the one introduced by Baryshnikov in relation with queuing theory in \cite{Baryshnikov}, and called the minor process by Johansson and Nordenstam in \cite{JohanssonNordenstam}. The main result of this section is stated at theorem \ref{theoGT}. It claims that  for any $\mathbb{F}$, the minor process associated to an invariant random matrix with a fixed radial part, is distributed according to the uniform law, or a projection of it when $\mathbb{F}=\mathbb{H}$, on a so called Gelfand--Tsetlin polytope. 

Our proofs rest on results given from sections 7 to 10, which involve elements of representation theory of compact Lie groups. In this section, our statements are made without any reference to this theory and  most of the proofs are postponed up to the section 10. 

When $M$ is a complex Hermitian matrix, Rayleigh's theorem claims that if $\lambda\in \mathbb{R}^n$  is the vector of the ordered eigenvalues of $M$ and if  $\beta\in \mathbb{R}^{n-1}$ is the one of its main minor $\pi_{n-1}(M)$ of order $n-1$, then $\lambda$ and $\beta$ satisfy interlacing conditions $\lambda_i\ge \beta_i\ge \lambda_{i+1}$, $i=1,\cdots,n-1$. Obviously this result also holds when $M$ belongs to $\mathcal{P}_n(\R)$ and $\mathcal{P}_n(\H)$, these sets being  subsets of complex Hermitian matrices. Thus for $\mathbb{F}=\mathbb{C},\mathbb{R}$, one obtain easily that $X(M)$ belongs to the so called Gelfand--Tsetlin polytopes, that we define below. We will see after these definitions what happens for $\mathbb{F}=\mathbb{H}$. For $x, y\in \mathbb{R}^{n}$ we write $x \succeq y$ if $x$ and $y$ are interlaced, i.e.\@
$$x_{1}\ge y_1\ge x_2 \ge \cdots\ge x_n \ge y_n  $$ and we write $x \succ y$ when
 $$ x_{1}> y_1> x_2 > \cdots> x_n> y_n. $$
When $x\in \mathbb{R}^{n+1}$ and $y\in\mathbb{R}^n$ we add the relation $y_n \geq x_{n+1}$ (resp. $y_n> x_{n+1}$).  We  denote $\vert x\vert$ the vector  whose components are  the absolute values of those of $x$. 
\begin{defn} Let $\lambda$ be in the Weyl chamber $\CC_n$. The Gelfand--Tsetlin polytope $GT_n(\lambda)$ is defined by  :
\begin{itemize}
\item
when $\mathbb{F}=\mathbb{C}$,
$$GT_n(\lambda)=\{(x^{(1)}, \cdots,x^{(n)}) : x^{(n)}=\lambda, x^{(k)}\in \mathbb{R}^{k}, \, x^{(k)}\succeq x^{(k-1)}, 1\leq k \leq n\},$$
\item
when $\mathbb{F}=\mathbb{H}$,
\begin{eqnarray*}GT_n(\lambda)&=&\{(x^{(\frac{1}{2})}, x^{(1)}, x^{(\frac{3}{2})},\cdots,x^{(n-\frac{1}{2})},x^{(n)}): x^{(n)}=\lambda, \\ &&\hskip 30pt  x^{(k)},\, x^{(k-\frac{1}{2})}\in \mathbb{R}_+^{k}, \, x^{(k)}\succeq x^{(k-\frac{1}{2})}\succeq x^{(k-1)}, 1\leq k \leq n\},\end{eqnarray*}
\item
when $\mathbb{F}=\mathbb{R}$,
\begin{eqnarray*}GT_n(\lambda)&=&\{(x^{(1)},\cdots, x^{(n)}): x^{(n)}=\lambda, x^{(k)}\in \mathbb{R}_+^{i-1}\times \mathbb{R}\mbox{ when } k=2i,\\ &&\hskip -20pt  x^{(1)}= 0,  \, x^{(k)}\in \mathbb{R}_+^{i} \mbox{ when } k=2i+1,\, \vert x^{(k)} \vert \succeq \vert  x^{(k-1)} \vert, 1\leq k \leq n\}.\end{eqnarray*}
 \end{itemize}
\end{defn} 
If $M$ is a matrix in $\mathcal{P}_n(\mathbb{H})$ such that $X^{(n)}(M)=\lambda$, then $X(M)$ belongs to the image of  $GT_n(\lambda)$ by the map $(x^{(\frac{1}{2})},x^{(1)},\cdots,x^{(n)})\in GT_n(\lambda) \mapsto (x^{(1)},x^{(2)},\cdots,x^{(n)})$. To prove it, we can consider for instance, for $r=1,\cdots, n$, the vector $X^{(r-\frac{1}{2})}(M)\in \mathbb{R}^r$ whose components are the ordered absolute values of the $r$ largest eigenvalues of the main minor of order $2r-1$ of $M$ considered  as a matrix from $\mathcal{P}_{2n}(\mathbb{C})$. Then Rayleigh's theorem implies that $$(X^{(\frac{1}{2})}(M),X^{(1)}(M),\cdots,X^{(n-\frac{1}{2})}(M),X^{(n)}(M))$$ belongs to the Gelfand--Tsetlin polytope $GT_n(\lambda)$ of type $\mathbb{H}$, which gives the announced  property.

Usually, an element $x$ of a Gelfand--Tsetlin polytope, is represented by a triangular array,  called Gelfand--Tsetlin array, as indicated from figures \ref{coneA} to \ref{coneD}.

\begin{figure}[h!]
\begin{pspicture}(2,3.5)(9,0)
 \put(0.75,0){$x^{(n)}_{1}$}  \put(2.25,0){$x^{(n)}_{2}$}  \put(3.75,0){$\cdots$}  \put(6.75,0){$\cdots$}  \put(8.25,0){$x^{(n)}_{n-1}$}  \put(9.75,0){$x^{(n)}_{n}$}
 
 \put(1.5,0.5){$x^{(n-1)}_{1}$}  \put(3,0.5){$x^{(n-1)}_{2}$}  \put(4.5,0.5){$\cdots$}  \put(6,0.5){$\cdots$}  \put(7.5,0.5){$x^{(n-1)}_{n-2}$}  \put(9,0.5){$x^{(n-1)}_{n-1}$}
 
  \put(5.25,1.25){$\cdots$}  
       \put(3.75,2){$x^{(3)}_{1}$}       \put(5.25,2){$x^{(3)}_{2}$}   \put(6.75,2){$x^{(3)}_{3}$}    
      \put(4.5,2.5){$x^{(2)}_{1}$}       \put(6,2.5){$x^{(2)}_{2}$} 
    \put(5.25,3){$x^{(1)}_{1}$}    

\end{pspicture}
  \caption{A Gelfand--Tsetlin array  for $\mathbb{F}=\mathbb{C}$}
  \label{coneA}
\end{figure} 

\begin{figure}[h!]
\begin{pspicture}(2,3.2)(9,-0.3)

 \put(0.25,-0.5){$x^{(n)}_{1}$}  \put(2,-0.5){$\cdots$}  \put(4,-0.5){$x^{(n)}_{n}$} \put(5.35,-0.5){$0$}    \put(6.25,-0.5){$-x^{(n)}_{n}$}  \put(8.5,-0.5){$\cdots $}  \put(10.25,-0.5){$-x^{(n)}_{1}$}
 
 \put(1,0){$x^{(n-\frac{1}{2})}_{1}$}  \put(2.75,0){$\cdots$}  \put(4.5,0){$x^{(n-\frac{1}{2})}_{n}$}  \put(5.75,0){$-x^{(n-\frac{1}{2})}_{n}$}  \put(7.75,0){$\cdots $}  \put(9.5,0){$-x^{(n-\frac{1}{2})}_{1}$}
 \put(4.25,0.5){$\cdots$}  \put(6,0.5){$\cdots$} 
   \put(2.25,1){$x^{(2)}_{1}$}     \put(3.75,1){$x^{(2)}_{2}$}     \put(5.35,1){$0$}   \put(6.5,1){$-x^{(2)}_{2}$}   \put(8,1){$-x^{(2)}_{1}$}  
     \put(3,1.5){$x_{1}^{(\frac{3}{2})}$}  \put(4.5,1.5){$x_{2}^{(\frac{3}{2})}$} \put(5.75,1.5){$-x_{2}^{(\frac{3}{2})}$} \put(7.125,1.5){$-x_{1}^{(\frac{3}{2})}$} 
       \put(3.75,2){$x^{(1)}_{1}$}       \put(5.35,2){$0$}   \put(6.5,2){$-x^{(1)}_{1}$}    
      \put(4.5,2.5){$x^{(\frac{1}{2})}_{1}$}       \put(5.75,2.5){$-x^{(\frac{1}{2})}_{1}$} 
    \put(5.35,3){$0$}    

\end{pspicture}
  \caption{A Gelfand--Tsetlin array  for $\mathbb{F}=\mathbb{H}$}
  \label{coneC}
\end{figure}

\begin{figure}[h!]
\begin{pspicture}(2,4)(9,0.25)
 \put(1,0){$x^{(n)}_{1}$}  \put(2.5,0){$\cdots $}\put(4.5,0){$x^{(n)}_{\tilde n}$}   \put(5.75,0){$-x^{(n)}_{\tilde n}$}  \put(8,0){$\cdots $}  \put(9.5,0){$-x^{(n)}_{1}$}

 \put(1.75,0.5){$x^{(n-1)}_{1}$}  \put(2.75,0.5){$\cdots $}\put(3.75,0.5){$x^{(n-1)}_{\tilde n-1}$}  \put(5.25,0.5){$x^{(n-1)}_{\tilde n}$}   \put(6.5,0.5){$-x^{(n-1)}_{\tilde n-1}$}  \put(7.75,0.5){$\cdots $}  \put(8.75,0.5){$-x^{(n-1)}_{1}$}
 
 \put(4.5,1){$\cdots$}  \put(6,1){$\cdots$} 
     \put(3,1.5){$x_{1}^{(5)}$}  \put(4.5,1.5){$x_{2}^{(5)}$} \put(5.75,1.5){$-x_{2}^{(5)}$} \put(7.125,1.5){$-x_{1}^{(5)}$} 
       \put(3.75,2){$x^{(4)}_{1}$}       \put(5.25,2){$x^{(4)}_{2}$}   \put(6.5,2){$-x^{(4)}_{1}$}    
      \put(4.5,2.5){$x^{(3)}_{1}$}       \put(5.75,2.5){$-x^{(3)}_{1}$} 
    \put(5.25,3){$x^{(2)}_{1}$}    

\end{pspicture}
  \caption{A Gelfand--Tsetlin array  for $\mathbb{F}=\mathbb{R}$, $n$ odd}
  \label{coneB}
\end{figure} 

\begin{figure}[h!]
\begin{pspicture}(2,4)(9,0.25)

 \put(1,0){$x^{(n)}_{1}$}  \put(2.5,0){$\cdots $}       \put(3.75,0){$x^{(n)}_{\tilde n-1}$}       \put(5.25,0){$x^{(n)}_{\tilde  n}$}   \put(6.5,0){$-x^{(n)}_{\tilde  n-1}$}      \put(8.25,0){$\cdots $}  \put(9.5,0){$-x^{(n)}_{1}$}
 
 \put(1.75,0.5){$x^{(n-1)}_{1}$}  \put(3,0.5){$\cdots $}       \put(4.5,0.5){$x^{(n-1)}_{\tilde  n-1}$}       \put(5.75,0.5){$-x^{(n-1)}_{\tilde n-1}$}      \put(7.5,0.5){$\cdots $}  \put(8.75,0.5){$-x^{(n-1)}_{1}$}
 \put(4.5,1){$\cdots$}  \put(6,1){$\cdots$} 
     \put(3,1.5){$x_{1}^{(5)}$}  \put(4.5,1.5){$x_{2}^{(5)}$} \put(5.75,1.5){$-x_{2}^{(5)}$} \put(7.125,1.5){$-x_{1}^{(5)}$} 
       \put(3.75,2){$x^{(4)}_{1}$}       \put(5.25,2){$x^{(4)}_{2}$}   \put(6.5,2){$-x^{(4)}_{1}$}    
      \put(4.5,2.5){$x^{(3)}_{1}$}       \put(5.75,2.5){$-x^{(3)}_{1}$} 
    \put(5.25,3){$x^{(2)}_{1}$}    

\end{pspicture}
  \caption{A Gelfand--Tsetlin array  for $\mathbb{F}=\mathbb{R}$, $n$ even}
  \label{coneD}
\end{figure} 

Let us say what is meant by the uniform measure on a Gelfand--Tsetlin polytope. It is a  bounded convex set of a real vector space. As usual, we define the volume of a bounded convex set $C$  as its measure according to the Lebesgue measure on the real affine subspace that it spans. We denote it $\text{vol}(C)$. We define the Lebesgue measure on $C$ as this Lebesgue measure  restricted to $C$ and the uniform probability measure on $C$ as the normalized Lebesgue measure on $C$.  

Let $M\in \mathcal{P}_{n}(\mathbb{F})$ be an invariant random matrix. The vector $X(M)$ is a random variable with values in $GT_n=\cup_{\lambda\in \CC_n} GT_n(\lambda)$. We will show that the law of $X(M)$ involves uniform probability measures on Gelfand--Tsetlin polytopes. 

\begin{defn} \label{mulambda} For $\lambda$ in the Weyl chamber $\CC_n$, we let $\mu_\lambda $ be the image of the uniform probability measure on $GT_n(\lambda)$ by the map $p_{n-1}: x\in GT_n(\lambda) \mapsto x^{(n-1)}\in \mathcal{C}_{n-1}$. 
\end{defn}

We observe from figures \ref{coneA} to \ref{coneD} that Gelfand--Tsetlin polytopes can be defined recursively. Thus  the uniform measure on $GT_n(\lambda)$, denoted $m_{GT_n(\lambda)}$, satisfies the remarkable identity 
\begin{align} \label{desintegration} m_{GT_n(\lambda)}=\int m_{GT_{n-1}(\beta)}\, \mu_\lambda(d\beta),\end{align}
which explains why we first focus on the measures $\mu_\lambda$, $\lambda\in \mathcal{C}_n$. The following lemma is proved at paragraph  \ref{prooflemGT}. The matrix  $\Omega_n(\lambda)$ considered in this lemma is defined in $2.2$.
 \begin{lem} \label{lemGT} Let  $\lambda$ be in the Weyl chamber $ \mathcal{C}_n$ and $U\in U_n(\mathbb{F})$ be a  Haar distributed random variable. Then the distribution of the radial part of the main minor of order $n-1$ of $U\Omega_n(\lambda)U^*$ is $\mu_\lambda$.
\end{lem}
We will now describe the law of $X(M)$ for every invariant random matrix $M$ in $\mathcal{P}_n(\F)$. It follows from  lemma \ref{invmeasure} that it is enough to describe it for $M=U\Omega_n(\lambda)U^*$, with $U$ a Haar distributed random variable in $U_n(\F)$ and $\lambda$ fixed in $\CC_n$.

\begin{theo} \label{theoGT} Let  $M=U\Omega_n(\lambda)U^*$, with $U$ Haar distributed in $U_n(\F)$ and $\lambda\in \CC_n$. 
Then $X(M)$  is uniformly distributed on $GT_n(\lambda)$ for $\mathbb{F}=\mathbb{R},\mathbb{C}$ and  is distributed according to the image of the uniform measure on $GT_n(\lambda)$ by the map $(x^{(\frac{1}{2})},\cdots,x^{(n-\frac{1}{2})},x^{(n)})\in GT_n(\lambda)\mapsto (x^{(1)},x^{(2)},\cdots,x^{(n)})$ for $\mathbb{F}=\mathbb{H}$.
\end{theo}
\begin{proof} Identity (\ref{desintegration}) implies that it is enough to prove that for every integer $k\in\{1,\cdots,n-1\}$ and every bounded  measurable function $f:\CC_k \to\mathbb{R}$,  the conditional expectations satisfy  $$\mathbb{E}\big[f \big(X^{(k)}(M) \big)\vert \sigma\{X^{(k+1)}(M),\cdots,X^{(n)}(M)\}\big]=\mathbb{E}\big[f \big(X^{(k)}(M) \big)\vert \sigma \{X^{(k+1)}(M)\}\big].$$
For $V\in U_{k+1}(\mathbb{F})$ we write $VMV^*$ instead of $$\begin{pmatrix}  V&   0   \\
 0 &   I 
\end{pmatrix}
M\begin{pmatrix}  V^*&   0   \\
 0 &   I 
\end{pmatrix},$$ where $I$ is the identity matrix with appropriate dimension. Let us write the radial decomposition $\pi_{k+1}(M)=V\Omega_{k+1}(X^{(k+1)}(M))V^*$, with $V\in U_{k+1}(\mathbb{F})$. Let $W$ be a random variable independent of $M$, Haar distributed in $U_{k+1}(\mathbb{F})$. We have   $W\pi_{k+1}(M)W^*=\pi_{k+1}(WMW^*)$ and $X^{(r)}(WMW^*)=X^{(r)}(M)$, $r=k+1,\cdots,  n$, so $$ \big(\pi_{k+1}(M),X^{(k+1)}(M),\cdots,X^{(n)}(M) \big)$$ has the same distribution as $$\big(W\Omega_{k+1}(X^{(k+1)}(M))W^*,X^{(k+1)}(M), \cdots,X^{(n)}(M)\big).$$ Then we have
\begin{align*}
\mathbb{E} \big[f(X^{(k)}&(M))\vert \sigma\{ X^{(k+1)}(M), \cdots,X^{(n)}(M)\} \big]\\
&=\mathbb{E} \big[f\big(X^{(k)}(\pi_{k+1}(M)) \big)\vert \sigma \{X^{(k+1)}(M), \cdots,X^{(n)}(M)\} \big]\\
&=\mathbb{E} \big[f \big(X^{(k)}(W\Omega_{k+1}(X^{(k+1)}(M))W^*) \big)\vert \sigma \{X^{(k+1)}(M), \cdots,X^{(n)}(M)\} \big]\\
&=\mathbb{E} \big[f \big(X^{(k)}(M) \big)\vert \sigma \{X^{(k+1)}(M)\} \big]. \qedhere
\end{align*} 
\end{proof}
Let us now give an explicit description of the measures $\mu_\lambda$, $\lambda\in \CC_n$.  We first introduce a function $d_n$, that we call asymptotic dimension. Recall that $\epsilon$ is equal to $1$ if $n\notin 2\mathbb{N}$ and $0$ otherwise.
\begin{defn} \label{asymptoticdim}  We define the function  $d_n$ on $\mathcal{C}_n$ by $$d_n(\lambda)=c_n(\lambda)^{-1}V_n(\lambda), \quad \lambda\in \mathcal{C}_n,$$ where the functions $V_n$ and $c_n$ are given by :
 \begin{itemize}
 \item when $\mathbb{F}=\mathbb{C}$,
 \begin{align*}
 V_{n}(\lambda)&= \prod_{\substack{1\le i<j\le n\\ \lambda_i\ne \lambda_j}}(\lambda_i-\lambda_j),\\  c_n(\lambda)&=\prod_{\substack{1\le i<j\le n\\ \lambda_i\ne \lambda_j}}(j-i),
  \end{align*}
  \item when $\mathbb{F}=\mathbb{H}$,
 \begin{align*}
 V_{n}(\lambda)&=\prod_{\substack{1\le i<j\le n\\ \lambda_i\ne \lambda_j}}(\lambda_i-\lambda_j)\prod_{\substack{1\le i<j\le n\\ \lambda_i\ne  -\lambda_j}}(\lambda_i+\lambda_j)\prod_{ \substack{1\le i\le n\\ \lambda_i\ne 0}} \lambda_i,\\  c_n(\lambda)&=\prod_{\substack{1\le i<j\le n\\ \lambda_i\ne \lambda_j}}(j-i)\prod_{\substack{1\le i<j\le n\\ \lambda_i\ne - \lambda_j}}(2n+2-j-i)\prod_{ \substack{1\le i\le n\\ \lambda_i\ne 0}}(n+1-i),
\end{align*}
  \item when $\mathbb{F}=\mathbb{R}$,
 \begin{align*}
 V_{n}(\lambda)&= \prod_{\substack{1\le i<j\le \tilde{n}\\ \lambda_i\ne \lambda_j}}(\lambda_i-\lambda_j) \prod_{\substack{1\le i<j\le \tilde{n}\\ \lambda_i\ne -\lambda_j}}(\lambda_i+\lambda_j)\prod_{ \substack{1\le i\le \tilde{n}\\ \lambda_i\ne 0}} \lambda_i^{\epsilon} ,\\  c_n(\lambda)&=\prod_{\substack{1\le i<j\le \tilde{n}\\ \lambda_i\ne \lambda_j}}(j-i) \prod_{\substack{1\le i<j\le \tilde{n}\\ \lambda_i\ne - \lambda_j}}(n-j-i)\prod_{ \substack{1\le i\le \tilde{n}\\ \lambda_i\ne 0}}(\tilde{n}+\frac{1}{2}-i)^{\epsilon} .
\end{align*}
\end{itemize}
\end{defn}
When $\lambda  $ is in the interior of the Weyl chamber, then $d_n(\lambda)$ is just, up to a constant, the product of the positive roots of $U_n(\mathbb{F})$.  We let $c_n=c_n(\lambda)$   in this case. We have  the following lemma.
\begin{lem} \label{rouge} Let $\lambda$ be in the interior of $\CC_n$.  Then
 \begin{itemize}
 \item when $\mathbb{F}=\mathbb{C}$,
 \begin{align*}
 d_{n}(\lambda)&= c_n \,det(\lambda_i^{j-1})_{n\times n},
  \end{align*}
  \item when $\mathbb{F}=\mathbb{H}$,
 \begin{align*}
d_{n}(\lambda)&=c_n \,det(\lambda_i^{2j-1})_{n\times n},\end{align*}
  \item when $\mathbb{F}=\mathbb{R}$,
 \begin{align*}
d_{n}(\lambda)&= c_n \,det(\lambda_i^{2j-2+\epsilon})_{\tilde n\times \tilde n}.
\end{align*}
\end{itemize}
\end{lem}

In particular, when $\F=\C$ and $\lambda$ is in the interior of the Weyl chamber, $d_n(\lambda)$ is just the Vandermonde polynomial. The following  lemma shows the importance of these fonctions for us.
\begin{lem} \label{volGT}
For any $\lambda$ in the Weyl chamber $\CC_n$,  the volume of  $GT_n(\lambda)$ is $d_n(\lambda)$. 
\end{lem}
\proof It is an immediate consequence of lemma \ref{dimGT}. \qed \medskip

 For $\lambda$   in the Weyl chamber, we let $l_\lambda$  be the Lebesgue measure on the convex set $p_{n-1}(GT_n(\lambda))$, where $p_{n-1}$ is the projection introduced at definition \ref{mulambda}.

\begin{lem}\label{restrictiongeneral} Let $\lambda$ be in the Weyl chamber. Then,
\begin{itemize}
\item  when $\mathbb{F}=\mathbb{R},\mathbb{C}$,
\begin{align*}
\mu_\lambda(d\beta) = \frac{d_{n-1}(\beta)}{d_n(\lambda)}\, l_\lambda(d\beta), 
\end{align*}
\item  when $\mathbb{F}=\mathbb{H}$,
\begin{align*}
\mu_\lambda(d\beta) = \frac{d_{n-1}(\beta)}{d_n(\lambda)}\, \text{vol}(\{z\in \mathbb{R}^n : \lambda\succeq z\succeq  \beta\}) l_\lambda(d\beta).
\end{align*}
\end{itemize}
\end{lem}
\proof For $x\in GT_n(\lambda)$, the vector $(x^{(1)},\cdots,x^{(n-1)})$  when $\mathbb{F}=\mathbb{C},\mathbb{R}$  and   the vector $(x^{(\frac{1}{2})},x^{(1)},x^{(\frac{3}{2})},\cdots,x^{(n-1)})$  when $\mathbb{F}=\mathbb{H}$, belong  the Gelfand--Tsetlin polytope $GT_{n-1}(x^{(n-1)})$, whose volume is equal to $d_{n-1}(x^{(n-1)})$ by lemma \ref{volGT}. This implies easily the lemma. \qed \medskip

We now give the density of $\mu_\lambda$ for two particular  cases: when $\lambda$ is in the interior of the Weyl chamber and when $\lambda$ has only one strictly positive component. For the computations, we recall a generalised Cauchy-Binet identity (see for instance  \cite{Johansson}). Let $(E,\mathcal{B},m)$ be a measure space, and let  $\phi_{i}$ and $\psi_{j}$, $1\le i,j\le n$, be measurable functions such that the $\phi_{i}\psi_{j}$'s are  integrable. The generalised Cauchy Binet identity is
\begin{align} \label{CauchyBinet}
\det\Big(\int_{E}\phi_{i}(x)\psi_{j}(x)dm(x)\Big)=\frac{1}{n!}\int_{E^{n}}\det\big(\phi_{i}(x_{j}) \big)\det \big(\psi_{i}(x_{j}) \big)\prod_{k=1}^n dm(x_{k}). \end{align}
Let us also recall  the identity which gives interlacing conditions with the help of a determinant (see Warren \cite{Warren}). Let $x$ and $y$ be two vectors in $\mathbb{R}^{n}$ such that $x_{1}> \cdots >x_{n}$ and $y_{1}> \cdots >y_{n}$. Then \begin{align} \label{formuleentrelac} 1_{\{x\succ y\}}=\det(1_{\{x_{i}>y_{j}\}})_{n\times n}.\end{align}

\begin{prop} \label{tworestrictions}   Let  $\lambda$ be in the Weyl chamber $\CC_n$. If $\lambda$  is in the interior of $\CC_n$, then the measure $\mu_ \lambda $ has a density $f_\lambda$ with respect to the Lebesgue measure on $\mathcal{C}_{n-1}$ defined by :
\begin{itemize}
\item when $\mathbb{F}=\mathbb{C}$, 
$
f_\lambda(\beta)=   \frac{d_{n-1}(\beta)}{d_n(\lambda)}1_{\{\lambda\, \succeq \beta\}},$
\item when $\mathbb{F}=\mathbb{R}$,
$
f_\lambda(\beta) =\frac{d_{n-1}(\beta)}{d_n(\lambda)}1_{\{\vert \lambda\vert \,\succeq \vert \beta\vert \}} ,$
\item when $\mathbb{F}=\mathbb{H}$,  
$
f_\lambda(\beta)=  \frac{d_{n-1}(\beta)}{d_n(\lambda)}\,\det( (\lambda_i- \beta_j)1_{\{\lambda_i\ge \beta_j\}})_{n\times n},
$ with the convention $\beta_n=0$.
 \end{itemize}
 If  $\lambda =(\theta,0, \cdots,0)$, $\theta\in \mathbb{R}_+$, then the measure $\mu_ \lambda $ is equal to $\tilde\mu_\lambda\otimes \delta_0^{\tilde n-1}$, $\tilde \mu_\lambda$ having a density $g_\theta$ with respect to the Lebesgue measure on $\mathbb{R}_+$ defined  by :\
 \begin{itemize}
 
\item when $\mathbb{F}=\mathbb{C}$, 
$
g_\theta(\beta)=(n-1) \frac{\beta^{n-2}}{\theta^{n-1}}1_{[0, \theta]}(\beta), 
$
\item when $\mathbb{F}=\mathbb{R}$, 
$
g_ \theta(\beta)=(n-2) \frac{\beta^{n-3}}{\theta^{n-2}}1_{[0, \theta]}(\beta),
$
 \item when $\mathbb{F}=\mathbb{H}$, 
$
g_ \theta(\beta)=(2n-2)(2n-1)\frac{\beta^{2n-3}}{\theta ^{2n-1}}(\theta-\beta)1_{[0, \theta]}(\beta).
$
\end{itemize}
 \end{prop}
\begin{proof}  Lemma \ref{restrictiongeneral} gives immediately the   densities $f_\lambda$ and $g_\theta$ for $\mathbb{F}=\mathbb{C},\mathbb{R}$.  For $\mathbb{F}=\mathbb{H}$ and $\lambda$ in the interior of the Weyl chamber, this lemma implies that, for $\beta\in \R^n_+$,$$f_\lambda(\beta)= \frac{d_{n-1}(\beta)}{d_n(\lambda)} \int_{\mathbb{R}^{n}} 1_{\{\lambda\succeq z\}}1_{\{ z\succeq \beta\}} \,dz.$$ We get the announced formula using the identity (\ref{formuleentrelac}) and the generalised Cauchy-Binet identity (\ref{CauchyBinet}). For $\mathbb{F}=\mathbb{H}$ and $\lambda=(\theta,0,\cdots,0)$ we get from lemma \ref{restrictiongeneral}  that, for $\beta\in \R_+$, $$g_ \theta(\beta)= \frac{d_{n-1}(\beta)}{d_n(\lambda)} \int_{\mathbb{R}} 1_{\{\theta\ge z\ge \beta\}} \, dz=\frac{d_{n-1}(\beta)}{d_n(\lambda)} (\theta-\beta)1_{[0,\theta]}(\beta). \qedhere
$$
\end{proof}

\section{Rank one perturbation on $\mathcal{P}_n(\mathbb{F})$}
The next two sections are devoted to the LUE introduced in Definition \ref{def_LUE}. We will focus on the distribution of the eigenvalues. A random matrix of the LUE$_{n,k}(\mathbb{F})$ can be written as $\sum_{i=1,\cdots,\tilde k}M_i\Omega_n^1M_i$, where the $M_i$'s are independent standard Gaussian variables in $\mathcal{M}_n(\F)$ (cf. Lemma \ref{FourierLUE}). We will compute the distribution of its eigenvalues  recursively on $k$. 
Thus, the study of an additive perturbation by the simplest Laguerre ensemble, i.e.\ the LUE$^1_{n}(\mathbb{F})$, is the first step, and we give in Theorem \ref{theopertuLUE} the distribution of $\Omega_n(\lambda)+M\Omega_n^1M^*$, for $\lambda\in \CC_n$. 

 For $\lambda $ an element of  $\mathcal{C}_n$ having only $k$ nonzero components, the others being equal to zero, we will write  $\Omega_n(\lambda_1, \cdots, \lambda_k)$ instead of $\Omega_n(\lambda)$. Let $\theta\in \mathbb{R}_+$ and $U$ be a Haar distributed random variable in $U_n(\F)$. We first describe the distribution of the radial part of $\Omega_n(\lambda)+ U\Omega_n(\theta) U^*$. We introduce the following sets.
\begin{defn} For $\lambda\in \mathcal{C}_n$, $\theta\in \mathbb{R}_+^*$, we define the set $\mathcal{E}(\lambda,\theta)$ by :
\begin{itemize}
\item when $\mathbb{F}=\mathbb{C}$, \begin{align*}\mathcal{E}(\lambda,\theta)=\{ (\beta,x)\in \mathbb{R}^n\times GT_n : \beta\succeq \lambda, \, \sum_{i=1}^n  (\beta_i-\lambda_i)= \theta, \, x\in GT_n(\beta)\}, 
\end{align*}
\item when $\mathbb{F}=\mathbb{H}$, \begin{align*} \mathcal{E}(\lambda,\theta)=\{(\beta,z,x)\in \mathbb{R}^n\times  \mathbb{R}^n & \times GT_n :   \lambda,\beta\succeq z , \\ & \sum_{i=1}^{n}(\lambda_i+ \beta_i-2z_i)=\theta, \, x\in GT_n(\beta)\},\end{align*}

\item when $ \mathbb{F}=\mathbb{R}$, $n=2r+1$, \begin{align*}\mathcal{E}(\lambda,\theta)=\{(\beta,z,x,&s)  \in\mathbb{R}_+^r\times  \mathbb{R}_+^r\times GT_n\times \{0,1\}: \lambda,\beta\succeq z , \\ & \sum_{i=1}^{r}(\lambda_i + \beta_i-2z_i)=\theta,\, x\in GT_n(\beta), \, s=0 \textrm{ if } \lambda_r=0\},\end{align*}

\item when $\mathbb{F}=\mathbb{R}$, $n=2r$,
  \begin{align*} \mathcal{E}(\lambda,\theta)=
 \{(\beta,z,x) \in &\mathbb{R}^{r}\times \mathbb{R}_+^{r-1}\times GT_n : \lambda,\beta \succeq z, \, \max(\vert \lambda_r \vert, \vert \beta_r \vert)\le z_{r-1} ,\\
&   \sum_{k=1}^{r-1}(\lambda_k+\beta_k-2z_k)+ \vert \lambda_r-\beta_r\vert =\theta, \, x\in GT_n(\beta) \}.
\end{align*} 
\end{itemize}
\end{defn}
Each set $\mathcal{E}(\lambda,\theta)$ is either a convex set or an union of two convex sets. Thus we can define the Lebesgue measure on it.
\begin{defn} \label{defnulambdatheta} For $\lambda\in \CC_n$, $\theta\in \mathbb{R}_+^*$, we define $\nu_{\lambda,\theta}$ as the image of the uniform probability measure on $\mathcal{E}(\lambda,\theta)$ by the projection on the component $\beta$, denoted by $p$.
\end{defn}
The following proposition is proved at paragraph \ref{proofpertutheta}. 
\begin{prop} \label{pertutheta}
Let $\theta\in\mathbb{R}_+^*$,  $\lambda\in\mathcal{C}_n$ and $U$ be a Haar distributed  random variable in $U_n(\mathbb{F})$. Then the radial part of the random matrix $\Omega_n(\lambda)+ U\Omega_n(\theta) U^*$ is distributed according to the measure $\nu_{\lambda,\theta}$. \end{prop}

Recall that a real random variable has a Gamma distribution with parameters $(\alpha,n)\in \R_+^*\times \mathbb{N}$, if its density is equal to
$$\frac{\alpha}{(n-1)!}e^{-\alpha x}(\alpha x)^{n-1}1_{\R_+}(x).$$   Recall also that we use the notation  $c=1$ when $\mathbb{F}=\C,\R$ and $c=2$ when $\F=\H$.
\begin{lem} \label{chisimple} 
Let $M$ be a standard Gaussian variable in $\mathcal{M}_n(\mathbb{F})$. Then the radial part of $M\Omega_n^1M^*$ has only one nonzero component $\Theta$. It has a Gamma distribution with parameters $(1,n)$ when $\F=\C$, $(1,n-1)$ when $\F=\R$, and $(2,2n)$ when $\F=\H$. Its density $f_\Theta$ can be written, for some $k>0$ as 
$$f_\Theta(\theta)= k\, d_n(\theta)e^{-c\theta}$$
\end{lem}
\proof We make the simple remark  that  $M\Omega_{n}^1M^*$ has the same eigenvalues as $M^{*}M\Omega_n^1$. If $\mathbb{F}=\mathbb{C}$ or $\mathbb{H}$,  the matrix $M^{*}M\Omega_n^1$ is equal to $\vert V\vert^2\Omega_n^1$, where $V$ is a standard Gaussian variable of $\mathbb{F}^n$. This shows that $M\Omega_n^1M^{*}$ has only one strictly positive eigenvalue, which has a Gamma  distribution with parameters   $(1,n)$ for $\mathbb{F}=\mathbb{C}$ and $(2,2n)$  for $\mathbb{H}$.  In the case when $\mathbb{F}=\mathbb{R}$, $M\Omega_n^1M^{*}$ has only one strictly positive eigenvalue.  The proof that it has a Gamma distribution with paramaters $(1,n-1)$ along the same line is not so immediate but remains quite elementary. Anyway, it will also follow from proposition \ref{Wishartgeneral} below. The last statement follows from the fact that $d_n(\theta)$ is equal to \begin{align} \label{dntheta} \frac{\theta^{n-1}}{(n-1)!}\,\textrm{ for }\F=\C,\,\frac{2\theta^{n-2}}{(n-2)!} \,\textrm{ for } \F=\R,\,\frac{\theta^{2n-1}}{(2n-1)!}\, \textrm{ for } \F=\H.\hskip1cm \qed \end{align} 
\medskip

 This fact, which will be useful for computations, is not a coincidence. Actually it is a particular case of a more general result, proposition \ref{Wishartgeneral}, the proof of which provides an interesting way to understand why the asymptotic dimension $d_n$ appears. 

\begin{theo} \label{theopertuLUE}
Let $M$ be a standard Gaussian variable in $\mathcal{M}_n(\mathbb{F})$ and $\lambda$ be an element of the Weyl chamber $\mathcal{C}_n$.  Then the distribution of the radial part of   $\Omega_n(\lambda)+M\Omega_n^1M^*,$ that we denote $\nu_\lambda$,  is the probability measure proportional to $$\int_{\mathbb{R}_+}\nu_{\lambda,\theta}\, d_n(\theta) e^{-c\theta}d\theta.$$  
\end{theo}
\proof The matrix $M\Omega_n^1M^*$ is an invariant random matrix in $\mathcal{P}_n(\F)$. Thus lemmas \ref{invmeasure} and \ref{chisimple} ensure that $M$ can be written $U\Omega_n(\Theta)U^*$, where $U$ and $\Theta$ are independent  random variables with a Haar distribution on  $U_n(\F)$ and the density $f_\Theta$. It suffices to apply  proposition \ref{pertutheta} to see that $\nu_\lambda=\int_{\R}\nu_{\lambda,\theta}f_{\Theta}(\theta)\, d\theta$.  \qed \medskip

In the following section we will need an explicit formula for the density of the measure $\nu_\lambda$. We first deal with the measure $\nu_{\lambda,\theta}$. 
 
\begin{lem} \label{volE} For $\lambda\in \mathcal{C}_n$, $\theta\in \mathbb{R}_+$, the volume of $\mathcal{E}(\lambda,\theta)$ is equal to $d_n(\lambda)d_n(\theta)$. 
\end{lem}
\proof The lemma is immediately  deduced from lemma \ref{dimE}. \qed \medskip

We denote by $l_{\lambda,\theta}$ the Lebesgue measure on $p(\mathcal{E}(\lambda,\theta))$, where $p$ is the projection introduced at definition \ref{defnulambdatheta}.

\begin{lem} \label{convolution} Let $\lambda$ be in the Weyl chamber and $\theta$ in $\mathbb{R}_+$. Then
\begin{itemize}
\item when $\mathbb{F}=\mathbb{C}$,
$$\nu_{\lambda,\theta}(d\beta)=\frac{d_n(\beta)}{d_n(\lambda)d_n(\theta)} l_{\lambda,\theta}(d\beta),$$
\item when $\mathbb{F}=\mathbb{H},\mathbb{R},$
$$\nu_{\lambda,\theta}(d\beta)=\frac{d_n(\beta)}{d_n(\lambda)d_n(\theta)}  \text{vol}(\mathcal{M}_{\lambda,\theta}(\beta)) l_{\lambda,\theta}(d\beta),$$\end{itemize}
where $\mathcal{M}_{\lambda,\theta}(\beta)$ is the projection, for $\beta$ fixed, of $\mathcal{E}(\lambda,\theta)$ on the component $z$ when $\mathbb{F}=\mathbb{H}$ or $\F=\R$ with n $even$, and on the component $(z,s)$ when $\F=\R$ with $n$ odd.
\end{lem}
\proof By definition $\nu_{\lambda,\theta}$ is the image of the uniform measure on  $\mathcal{E}(\lambda,\theta)$ by the projection $p$. Thus, the normalisation follows from the lemma \ref{volE} and the factor $d_n(\beta)$ appears when one integrates this uniform measure with respect to the component $x\in GT_n(\beta)$. \qed\medskip

We will see that $\text{vol}(\mathcal{M}_{\lambda,\theta}(\beta))$ plays the role of an asymptotic multiplicity. It is replaced by one when $\F=\C$ because this is the only field for which the irreducible decompositions described from propositions \ref{tensrepA} to \ref{tensrepD} are multiplicity free. Let us now describe the measure $\nu_\lambda$ in some particular cases.

\begin{prop} \label{pertu} Let $k$ be an integer smaller that $\tilde{n}$ and  $\lambda\in \mathbb{R}^{\tilde{n}}$ be equal to $(\lambda_1, \cdots ,\lambda_k,0, \cdots,0)$. When $\mathbb{F}=\mathbb{R}$, if $n$ is even and $k=\tilde{n}$, we suppose that  $\lambda_1> \cdots >\lambda_{k-1}>\vert \lambda_k\vert$. For every other case we suppose that $\lambda_1> \cdots >\lambda_k>0$.

 Then the measure $\nu_{\lambda}$ is equal to $\tilde{\nu}_{\lambda}\otimes \delta_0^{\otimes (n-(k+1)\wedge n)} $,  where $\tilde{\nu}_{\lambda}$ has a density $L_\lambda$ with respect to the Lebesgue measure on $\mathbb{R}^{(k+1)\wedge \tilde n}$ defined by
 \begin{itemize}
\item when $\mathbb{F}=\mathbb{C} $,
 \begin{align*}
 L_\lambda(\beta)=\frac{d_n(\beta)}{d_n(\lambda)}1_{\{\beta\succeq \lambda\}}\, e^{-\sum_{i=1}^{(k+1)\wedge n }(\beta_i-\lambda_i)},\end{align*}
 
\item  when $\mathbb{F}=\mathbb{H} $,
 \begin{align*}
 L_\lambda(\beta)=2^n \frac{d_n(\beta)}{d_n(\lambda)} \Big[\int_{\mathbb{R}_+^{k}} 1_{\{\lambda,\beta\succeq z\}}e^{-2\sum_{i=1}^{k}(\lambda_i+\beta_i-2z_i)-\beta_{k+1}1_{\{k<n\}}}\,dz \Big] 1_{\{\beta_{(k+1)\wedge n}\ge 0\}}.\end{align*}

\item  when $\mathbb{F}=\mathbb{R}$,  $n=2r$, $k=r,r-1$,
\begin{align*}
 L_\lambda(\beta)=\frac{1}{2} \frac{d_n(\beta)}{d_n(\lambda)} \Big[\int_{\mathbb{R}_{+}^{r-1}} 1_{\{\lambda,\beta\succeq z, \vert \lambda_r\vert,\vert \beta_r\vert \le z_{r-1}\} } e^{-\sum_{i=1}^{r-1}(\lambda_i+\beta_i-2z_i)-\vert \lambda_r-\beta_{r}\vert  }\, dz \Big],\end{align*}

\item  when  $ \mathbb{F}=\mathbb{R}$, $n=2r$, $k\le r-2$,
\begin{align*}
 L_\lambda(\beta)=\frac{1}{2} \frac{d_n(\beta)}{d_n(\lambda)}\Big[\int_{\mathbb{R}_{+}^{k}} 1_{\{\lambda,\beta\succeq z \}} e^{-\sum_{i=1}^{k}(\lambda_i+\beta_i-2z_i)-\lambda_{k+1} }\, dz \Big] 1_{\{\beta_{k+1}\ge 0\}}.
\end{align*}
\end{itemize}
\end{prop}
\proof  Using the same notations as in the proof of theorem \ref{theopertuLUE} we have that $\nu_\lambda=\int_{\R_+} \nu_{\lambda,\theta} f_\Theta (\theta)\, d\theta$. Thus the proposition follows immediately from lemma \ref{convolution} and the fact that $f_\Theta(\theta)=d_n(\theta)e^{-\theta}$ when $\F=\C$, $f_\Theta(\theta)=2^nd_n(\theta)e^{-2\theta}$ when $\F=\H$, and $f_\Theta(\theta)=\frac{1}{2}d_n(\theta)e^{-\theta}$ when $\F=\R$.  \qed
\begin{rem} \label{BCalea} We  observe in lemma \ref{convolution} that the measures $\nu_{\lambda,\theta}$ are the same, in the cases $\mathbb{F}=\mathbb{R}$, $n=2r+1$ and $\mathbb{F}=\mathbb{H}$, $n=r$ (see \ref{relquatreal} for explanations). Moreover, for that two cases, the functions $d_n$ are the same, up to a constant. Thus  the measures $\nu_\lambda$ defined at theorem \ref{theopertuLUE} are the same, up to the constant $c$. That's why we didn't write both cases in the previous proposition.
\end{rem}

\section{Generalised Laguerre ensembles}
In this section we compute the law of the radial part of a matrix in LUE$_{n,k}(\mathbb{F})$ by considering successive rank one perturbations, i.e.\@ the random walk $(S_k)_{k\ge 0}$  on $\mathcal{P}_n(\F)$ defined by $S_k=\sum_{i=1}^{k}M_i\Omega_n^1 M_i$, where the $M_i$'s are independent standard Gaussian variables in $\mathcal{M}_n(\F)$. We compute the law of the radial part $R_k$ of $S_k$ by induction. The following proposition concerns  the chain $(R_k)_{k\ge 0}$.  In Figures \ref{perturbationA},  \ref{perturbationC} and  \ref{perturbationD} the black discs represent successive states of this chain for $\mathbb{F}=\mathbb{C},\mathbb{H}$ and $\mathbb{R}$. The white discs are intermediate points which indicate the interlacing conditions satisfied by the chain.

\begin{prop} \label{Markovpertu}
The process  $(R_k)_{k\ge 0}$ is a Markov chain whose  transition probability $P(\lambda,\, )$ is equal to $\nu_\lambda$. When $R_0=0$,  $R_k$ has $k\wedge \tilde{n}$ nonzero components.
\end{prop} 
\proof As in  lemma \ref{invmeasure}, we see that $(S_k,R_k, \cdots,R_1)$ has the same law as $$(U\Omega_n(R_k)U^*,R_k, \cdots,R_1),$$ when $U\in U_n(\mathbb{F})$ is a Haar distributed random variable  independent of $(R_k, \cdots,R_1)$. Then for every bounded measurable function $f:\mathcal{P}_n(\mathbb{F})\to\mathbb{R}$, \begin{align*} \mathbb{E}(f(S_{k+1})\vert \sigma\{ R_k, \cdots,R_1\})&=\mathbb{E}(f(U\Omega_n(R_k)U^* + M_{k+1}\Omega_n^1M_{k+1}^*)\vert \sigma\{R_k,\cdots,R_1\})\\ &=\mathbb{E}(f(U\Omega_n(R_k)U^* + M_{k+1}\Omega_n^1M_{k+1}^*)\vert\sigma\{ R_k\}).\end{align*} Thus $(R_k)_{k\ge 0}$ is Markovian. The transition probability is given by theorem \ref{theopertuLUE} and it is clear, for instance from lemma \ref{convolution}, that the last point  is true.
 \qed\medskip

\begin{figure}[h!]
\begin{pspicture}(-9,2.5)(4,-0.25)
\psline{<-}(1,0)(-5.3,0)\psline[linecolor=lightgray]{<-}(1,0.5)(-5.3,0.5)\psline[linecolor=lightgray]{<-}(1,1)(-5.3,1)\psline[linecolor=lightgray]{<-}(1,1.5)(-5.3,1.5)\psline[linecolor=lightgray]{<-}(1,2)(-5.3,2)\psline{->}(-5,-0.2)(-5,2.5)
\psdots(-5,0)\psarc(-5,0){0.1}{0}{360} \psarc(-5,0){0.15}{0}{360} 
\psdots(-5,0.5)\psarc(-5,0.5){0.1}{0}{360}\psdots(-3.6,0.5)
\psdots(-5,1)\psdots(-4.4,1)\psdots(-2.6,1)
\psdots(-4.6,1.5)(-3,1.5)(-1.2,1.5) 
\psdots(-3.4,2)(-2.2,2)(-0,2) 
\put(-5.15,-0.55){$0$}\put(-5.8,0){$R_0$}\put(-5.8,0.5){$R_1$}\put(-5.8,1){$R_2$}\put(-5.8,1.5){$R_3$}\put(-5.8,2){$R_4$}
\end{pspicture}
  \caption{Rank one perturbations on $\mathcal{P}_3(\mathbb{C})$.}
  \label{perturbationA}
\end{figure}

\begin{figure}[h!]
\begin{pspicture}(-9.5,3.5)(5,-0.25)
\psline{<-}(-0,0)(-5.3,0)\psline[linecolor=lightgray]{<-}(-0,0.5)(-5.3,0.5)\psline[linecolor=lightgray]{<-}(-0,1)(-5.3,1)\psline[linecolor=lightgray]{<-}(-0,1.5)(-5.3,1.5)\psline[linecolor=lightgray]{<-}(-0,2)(-5.3,2)\psline[linecolor=lightgray]{<-}(-0,2.5)(-5.3,2.5)\psline[linecolor=lightgray]{<-}(-0,3)(-5.3,3)\psline{->}(-5,-0.2)(-5,3.5)
\psdots(-5,0)\psarc(-5,0){0.1}{0}{360} 
\psdots[dotstyle=o](-5,0.5)\psarc[linewidth=0.2pt](-5,0.5){0.1}{0}{360}
\psdots(-5,1)\psdots(-2.5,1)
\psdots[dotstyle=o](-5,1.5)(-3.2,1.5) 
\psdots(-2,2)(-3.7,2)
\psdots[dotstyle=o](-2.6,2.5)(-4.7,2.5)
\psdots(-1,3)(-4.2,3)
\put(-5.15,-0,55){$0$}\put(-5.8,0){$R_0$}\put(-5.8,1){$R_1$}\put(-5.75,2){$R_2$}\put(-5.8,3){$R_3$}
\end{pspicture}
  \caption{Rank one perturbations on $\mathcal{P}_2(\mathbb{H})$.}
  \label{perturbationC}
\end{figure}

\begin{figure}[h!]
\begin{pspicture}(-10,3.5)(4,-0.25)
\psline{<-}(-0,0)(-6.3,0)\psline[linecolor=lightgray]{<-}(-0,0.5)(-6.3,0.5)\psline[linecolor=lightgray]{<-}(-0,1)(-6.3,1)\psline[linecolor=lightgray]{<-}(-0,1.5)(-6.3,1.5)\psline[linecolor=lightgray]{<-}(-0,2)(-6.3,2)\psline[linecolor=lightgray]{<-}(-0,2.5)(-6.3,2.5)\psline[linecolor=lightgray]{<-}(-0,3)(-6.3,3)\psline{->}(-5,-0.2)(-5,3.5)
\psdots(-5,0)\psarc(-5,0){0.1}{0}{360} 
\psdots[dotstyle=o](-5,0.5) 
\psdots(-5,1)\psdots(-2.2,1)
\psdots[dotstyle=o](-3.7,1.5) 
\psdots(-2.7,2)(-5.5,2)
\psdots[dotstyle=o](-3.5,2.5)
\psdots(-1.5,3)(-4.2,3)
\put(-5.15,-0,55){$0$}\put(-6.8,0){$R_0$}\put(-6.8,1){$R_1$}\put(-6.8,2){$R_2$}\put(-6.8,3){$R_3$}
\end{pspicture}
  \caption{Rank one perturbations on $\mathcal{P}_4(\mathbb{R})$.}
  \label{perturbationD}
\end{figure}
We have now gathered all the ingredients needed to get the law of the eigenvalues of the  matrices from the $LUE_{n,k}(\mathbb{F})$.  For $\lambda\in\mathbb{R}^n$, the Vandermonde determinant is  $$\Delta_n(\lambda)=\prod_{1\le i<j\le n}(\lambda_i-\lambda_j).$$ Recall that $\tilde n=n$ when $\F=\C,\H$, $\tilde n=[n/2]$ when $\F=\R$,  $c=1$ when $\mathbb{F}=\mathbb{R},\mathbb{C}$, $c=2$ when $\mathbb{F}=\mathbb{H}$ and $\Omega_k$ is given by (\ref{omegan}).    
\begin{theo}  \label{eigenvaluesLUE} Let $M$ be a standard Gaussian variable in $\mathcal{M}_{n,k}(\mathbb{F})$. Then the positive eigenvalues of $M\Omega_kM^*$ have a density $f_{n,k}$ with respect to the Lebesgue measure on $\mathbb{R}^{\tilde{n}\wedge \tilde{k}}$ and there exists a constant $C>0$ such that  for $\lambda\in \R^{\tilde n\wedge \tilde k}$, 
\begin{align} 
 f_{n,k}(\lambda)=C\, d_n(\lambda)\Delta_{\tilde{n}\wedge \tilde{k}}(\lambda)  \prod_{i=1}^{\tilde{n}\wedge \tilde{k}}  \lambda_i^{(\tilde{k}-\tilde{n}) \vee 0}e^{-c \, \lambda_i} \label{densityLUE}.
\end{align}
\end{theo}
\proof Let us first prove it by induction on $k$ when $\mathbb{F}=\mathbb{C},\mathbb{H}$. The random matrix $M\Omega_kM^*$ has the same law as the random variable  $S_k$ introduced at the beginning of the section. The property is true for $k=1$ by  lemma \ref{chisimple}. Suppose that it is true for $k\ge 1$.  Let $\gamma=(\gamma_1, \cdots,\gamma_{k\wedge n},0, \cdots,0)$ be a vector of $\mathbb{R}^n$ such that $\gamma_1> \cdots >\gamma_{k\wedge n}> 0$ and $M$ be a standard Gaussian variable in $ \mathcal{M}_{n,1}(\mathbb{F})$. Proposition \ref{pertu} ensures that the strictly positive eigenvalues of $\Omega_n(\gamma)+M\Omega_n^1M^*$ have a density $L_\gamma$ with respect to the Lebesgue measure on $\mathbb{R}^{n\wedge (k+1)}$, which proves the first point and implies that, for $\lambda\in \mathbb{R}^{n\wedge (k+1)}$, \begin{align} \label{casplusun} f_{n,k+1} (\lambda)=\int_{\R_+^{n\wedge k} }f_{n,k}(\gamma)L_{\gamma}(\lambda)\, d\gamma.\end{align} Let us now distinguish the complex and quaternionic cases. When $\mathbb{F}=\mathbb{C}$, identity (\ref{casplusun}) and the induction hypothesis imply that there exists a constant $C_1$ such that   \begin{align} \label{casplusunC} f_{n,k+1} (\lambda)=C_1\,d_n(\lambda)\, e^{-\sum_i \lambda_i }  \int_{\mathbb{R}_+^{k\wedge n}}\Delta_{k\wedge n}(\gamma)   \prod_{i=1}^{{n}\wedge{k}} \gamma_i^{({k}-{n}) \vee 0} 1_{\{\lambda\succeq \gamma \}} \,d\gamma.\end{align}
When $k<n$, the integral above is an homogeneous polynomial of degree $\frac{1}{2}k(k+1)$,  equal to zero when $\lambda_i=\lambda_j$, $i\ne j$, so it is proportional to $\Delta_{k+1}(\lambda)$. This  proves the property for $k+1\le n$. The   positive eigenvalues of $MM^*$ being the same as those of $M^*M$, we get the proposition for $k \ge  n$ as well. It implies that for some $c_k>0$, \begin{align}\label{eqck} \int_{\mathbb{R}_+^{k\wedge n}}\Delta_{k\wedge n}(\gamma)   \prod_{i=1}^{{n}\wedge{k}} \gamma_i^{({k}-{n}) \vee 0}  1_{\{\lambda\succeq \gamma \}} \,d\gamma=c_k\, \Delta_{(k+1)\wedge n}(\lambda) \prod_{i=1}^{{n}\wedge (k+1)} \lambda_i^{({k+1}-{n}) \vee 0}.\end{align}  When $\mathbb{F}=\mathbb{H}$, we get that $f_{n,k+1}(\lambda)$ is proportional  to 
\begin{align*}
d_n(\lambda)\, e^{-2\sum_i \lambda_i}  \int_{\mathbb{R}_+^{k\wedge n}}  1_{\{\lambda\succeq z\}}\big[ \int_{\mathbb{R}_+^{k\wedge n}} 1_{\{\gamma\succeq z\}}\, e^{-4\sum_i( \gamma_i-z_i)}\, \Delta_{k\wedge n}(\gamma)\prod_{i=1}^{k\wedge n}\gamma_i^{(k-n)\vee 0}\, d\gamma \big]\, dz.
\end{align*}
The generalised Cauchy Binet identity implies that
\begin{align*}
 \int_{\mathbb{R}_+^{k\wedge n}} \Delta_{k\wedge n}(\gamma)&\, 1_{\gamma\succeq z}\, e^{-4\sum_{i=1}^{k\wedge n}( \gamma_i-z_i)} \prod_{i=1}^{k\wedge n}\gamma_i^{(k-n)\vee 0}\,d\gamma_i \\ &=  \frac{1}{(k\wedge n)!}\int_{\mathbb{R}_+^{k\wedge n}} \det(\gamma_i^{j-1+(k-n)\vee 0}) \det(1_{\{\gamma_i> z_j\}}\, e^{-4(\gamma_i-z_j)})\, d\gamma\\
 &= \det\big(\int_{\mathbb{R}_+}\gamma^{j-1+(k-n)\vee 0}1_{\{\gamma>z_i\}}\, e^{-4(\gamma-z_i)}\, d\gamma\big)\\
  &=C_2 \,  \det(z_i^{j-1+(k-n)\vee 0})=C_2 \, \Delta_{k\wedge n}(z)\prod_{i=1}^{k\wedge n}z_i^{(k-n)\vee 0},
\end{align*} 
where $C_2$ is a constant. Using (\ref{eqck}), this proves the property for $k+1$.    

Let us now prove the proposition when $\mathbb{F}=\mathbb{R}$. By Remark \ref{BCalea}  the odd  real case  is the same as the quaternionic case replacing $n$, $k$ and $c=2$ by $\tilde{n}$, $\tilde{k}$ and $c=1$. Thus, the property is true for the real odd case. If $n$ is even,  it is easier to use what we know about the odd case rather than proposition \ref{pertu}  to get the result. Let us  consider the random matrix $$N=\begin{pmatrix}  M   \\
X
\end{pmatrix},$$ $X$ being a standard Gaussian variable in $\mathcal{M}_{1,k}(\mathbb{R})$,  independent of $M$. Then, the density of the strictly positive eigenvalues of $N\Omega_kN^*$ is $f_{n+1,k}$. This random matrix has a law invariant for the adjoint action of $U_{n+1}(\mathbb{R})$ and  its main minor of order $n$ is $M\Omega_kM^*$. Thus, using lemma  \ref{restrictiongeneral}, we get that for $\lambda\in \R_+^{\tilde n\wedge \tilde k}$, $f_{n,k}(\lambda)$ is proportional to $$\int_{\mathbb{R}_+^{\tilde n\wedge \tilde k}} \frac{d_{n}(\lambda)}{d_{n+1}(\gamma)}f_{n+1,k}(\gamma)1_{\{\gamma\succeq \lambda\}}\, d\gamma. $$
The integer $n+1$ is odd, so we can replace $f_{n+1,k}$ in the previous identity by the formula (\ref{densityLUE}). An easy computation achieves the proof.  \qed\medskip

Let us notice that this theorem shows that the eigenvalues of a random matrix from the LUE($\F$) are distributed as some biorthogonal Laguerre ensembles studied by Borodin in \cite{Borodinseul}. Moreover it allows us to compute the density of the random matrix itself. Let $\epsilon$ be equal to $1$ if $n\notin 2\mathbb{N}$ and $0$ otherwise.
\begin{theo} \label{matrixLUE} When $k\ge n$  the distribution of  a matrix of the LUE$_{n,k}(\mathbb{F})$ has a density $l(H)$ with respect to the Lebesgue measure $dH$ on $\mathcal{P}_n(\F)$ proportional to 
$$\prod_{i=1}^{n}  \lambda_{i}^{k-n} e^{-\lambda_i}1_{\mathbb{R}_+}(\lambda_i), \hspace{1cm}Ê\textrm{ for } \mathbb{F}=\mathbb{C},$$
$$
\frac{1}{\prod_{1\le i<j\le n}(\lambda_i+\lambda_j)} \prod_{i=1}^{n}  \lambda_{i}^{k-n-1} e^{-2\lambda_i}, \hspace{1cm} \textrm{ for } \mathbb{F}=\mathbb{H},$$
$$\frac{1}{\prod_{1\le i<j\le \tilde{n}}(\lambda_i+\lambda_j)} \prod_{i=1}^{\tilde{n}}  \lambda_{i}^{\tilde{k}-\tilde{n}-\epsilon} e^{-\lambda_i}, \hspace{1cm} \textrm{ for } \mathbb{F}=\mathbb{R},$$
where $\lambda$ is the vector of eigenvalues of $H$ when $\F=\C$, of positive eigenvalues of $H$ when $\F=\R,\H$.
\end{theo}\proof The function $d_n$ being proportional to the product of roots on the interior of the Weyl chamber, Weyl integral's formula (see Helgason \cite{Helgason}, Thm. I.5.17)
 says that there exists a constant $C>0$ such that for every invariant measurable function $f:\mathcal{P}_n(\F)\to \R^+$, we have 
$$\int_{\mathcal{P}_n(\F)}f(H)\, dH = C\, \int_{C_n} d_n(\lambda)^2 f(\lambda)\, d\lambda, $$ 
where $dH$ and $d\lambda$ are  the Lebesgue measure  on $\mathcal{P}_n(\F)$ and $\CC_n$. Thus theorem \ref{eigenvaluesLUE}  implies that  the density of  a random matrix of the LUE$_{n,k}(\mathbb{F})$ is proportional to 
$$H\in \mathcal{P}_n(\mathbb{F}) \mapsto \frac{\Delta_{\tilde{n}}(\lambda)}{d_n(\lambda)}\prod_{i=1}^{\tilde{n}}  \lambda_{i}^{\tilde{k}-\tilde{n}} e^{-c\lambda_i}1_{\mathbb{R}_+}(\lambda_i).$$
We achieve the proof by replacing $d_n(\lambda)$ by its value. \qed\medskip

For $\lambda$ in  $\mathcal{C}_n$, let us consider a random matrix $\Omega_n(\lambda)+\sum_{i=1,\cdots,\tilde{k}}M_{i}\Omega_n(\alpha_i)M_i^*$, where the $M_i$'s are  independent standard Gaussian variables in $\mathcal{M}_n(\F)$ and the $\alpha_i$'s are some real numbers,  or equivalently a random matrix $\Omega_n(\lambda)+M\Omega_k(\alpha)M^*$, where $M$ is a standard Gaussian variable in $\mathcal{M}_{n,k}(\mathbb{F})$. When $\lambda=0$ and  $\mathbb{F}=\mathbb{C}$, it has a generalised Wishart distribution. One easy way to compute the law of its eigenvalues is to use the Harish Chandra formula (see for instance Wang \cite{Wang}). But this method doesn't work for the other fields. Our method, which consists in computing the law of the eigenvalues by induction, provide a way to compute the law of the radial part of   $\Omega_n(\lambda)+M\Omega_k(\alpha)M^*$ for any field $\mathbb{F}$. Nevertheless, computations are not always easy for such a general matrix. 
Actually,  in the case when  $\lambda=0$ and $k\le n$, computations are much simpler using lemma \ref{lemGT} rather than this approach. To do this, we need the following theorem which goes back to Borel \cite{Borel} (see  Olshanski \cite{Olshanski2}, Pickrell \cite{Pickrell}).

\begin{theo} \label{HaartoGauss} Let $(U_N)_{N\ge 1}$ be a sequence of random variable such that $U_N$ is  Haar distributed in $U_N(\mathbb{F})$. Then the main minor of order $n$ of ${\sqrt{N}}U_N$ converges in distribution to a standard Gaussian variable  in $\mathcal{M}_n(\mathbb{F})$, when $N$ goes to $+\infty$. 
\end{theo}

 \begin{prop} \label{Wishartgeneral}   Let $k$ be an integer smaller than $n$, $M$ be a standard Gaussian variable in  $\mathcal{M}_{n,k}(\mathbb{F})$ and $\alpha\in \mathbb{R}^{\tilde{k}}$ such that $\alpha_1>...>\alpha_{\tilde{k}}>0$. Then there is a constant $C$ such that the positive eigenvalues of $M\Omega_k(\alpha)M^*$  have a density $g_{n,k}$ with respect to the Lebesgue measure on $\mathbb{R}_+^{\tilde{k}}$ defined by $$g_{n,k}(\lambda)= C\, \frac{d_n(\lambda)}{d_n(\alpha)\prod_{i=1}^{\tilde{k}}\alpha_i}\det(e^{-c\frac{\lambda_i}{\alpha_j}})_{1\le i,j\le \tilde{k}}.$$ 
 In particular, when $\tilde k=1 $ and $\alpha_1=1$, this density is proportional to the function $\theta\in \mathbb{R} \mapsto d_n(\theta)e^{-c\theta}1_{\mathbb{R}_+}(\theta)$.  
 \end{prop}
 \proof    Let $N\in \mathbb{N}$ be an integer greater than $n$. We consider  a Haar distributed random variable $ U_N \in U_N(\mathbb{F})$, the random matrix $M_N=U_N\Omega_N(\alpha)U_N^*$ and its main minor of order $n$ denoted by $\pi_n(M_N)$.  Using lemma \ref{restrictiongeneral} we obtain that the density of  the $\tilde k$ strictly positive eigenvalues of $\pi_{N-1}(M_N)$ is proportional to $$ \frac{d_{N-1}(\lambda)}{d_N(\alpha)} \det\big(\frac{(\alpha_i-\lambda_j)^{c-1}}{(c-1)!}1_{\{\alpha_i> \lambda_j\}}\big)_{1\le i,j\le \tilde{k}}.$$ 
 Iterating for the smaller minors and using the Cauchy Binet identity we obtain that  the density of the strictly positive eigenvalues of $\pi_n(M_N)$ is proportional to $$\frac{d_{n}(\lambda)}{d_N(\alpha)} \det \big(\frac{(\alpha_i-\lambda_j)^{c(N-n)-1}}{(c(N-n)-1)!}1_{\alpha_i> \lambda_j} \big)_{1\le i,j\le \tilde{k}}.$$ 
So the distribution  of the strictly positive eigenvalues of $N\pi_n(M_N)$ converges to a distribution with a density proportional to 
$$ \frac{d_n(\lambda)}{d_n(\alpha)\prod_{i=1}^{\tilde{k}}\alpha_i}\det(e^{-c\frac{\lambda_i}{\alpha_j}})_{1\le i,j\le \tilde{k}}.$$ 
Theorem \ref{HaartoGauss} states that $N\pi_n(M_N)$ converges in distribution to $M\Omega_n(\alpha)M^*$, as $N$ goes to infinity, which completes the proof. \qed \medskip

The joint eigenvalues density of the Laguerre unitary ensemble LUE$(\C)$ has been known for a long time \cite{Goodman}. For the invariant ensembles LUE($\F$) with $\F=\H$ or $\F=\R$, it seems to be new: none of them is  associated to the Gaussian ensembles for the symmetry classes recalled in  \ref{subappli}. We have already seen some specificities of these ensembles: for instance the support of a random matrix of the LUE$_{n,k}(\F$) for the other fields than $\C$ is all the set of rank $\tilde k$ matrices of $\PP_n(\F)$ whereas in the complex case, this is the set of positive rank $k$ Hermitian matrices. For instance, we know that if $M=(M_t)_{t\ge 0}$ is a standard Brownian motion in $\mathcal{M}_n(\mathbb{C})$, then $M_tM_t^*$ and the process of its eigenvalues are Markovian. This follows from stochastic matrix calculus (see Bru \cite{Bru}), or more conceptually from the fact that they are radial parts of the Brownian motion in the flat symmetric space associated to $U(n,k)/U(n)\times U(k)$ (see Forrester \cite{ForresterBook}, Roesler \cite{Rosler}). This is not the case in general: if $(M_t)$ is for instance a standard Brownian motion  in $\mathcal{M}_2(\R)$, then neither $M_t\Omega_2 M_t^*$ nor the process of its eigenvalues (this is the same here !) is Markovian. It will be interesting to investigate these invariant ensembles which seem to be deeply different from the usual ones. 

\section{Interlaced determinantal processes} 
 Let $E$ be a Borel subset of $\mathbb{R}^{r}$. A counting measure $\xi$ on $E$ is a measure such that  $\xi(B)$ is an integer for all bounded Borel set $B$ of $E$. Let us consider  a sequence $(T_k)_{k\ge 1}$ of random variables with values in  $E$  and $\Xi=\sum_{k\in \mathbb{N}}\delta_{T_k}$. If  $\Xi$ is almost surely a counting measure on $E$, we say that $\Xi$ is  a  point process on $E$. Let $m$ be a measure on $E$. A function $\rho_{n}$ on $E^{n}$  such that 
$$\mathbb{E}[\prod_{i=1}^{n}\Xi(B_{i})]=\int_{B_{1}\times \cdots\times B_{n}} \rho_{n}(x_{1}, \cdots,x_{n})\, m(dx_{1}) \cdots m(dx_{n}),$$
for every disjoint bounded  Borel sets $B_{1}, \cdots,B_{n}$ in $E$, is called a  $n^{th}$ correlation function.  The measure $m$ is called the reference measure.
\begin{defn}
If there exists a function $K:E\times E \to \mathbb{C}$ such that  for all $n \geq 1$,
$$ \rho_{n}(x_{1}, \cdots,x_{n})=\det(K(x_{i},x_{j}))_{n\times n},$$
for $x_{1}, \cdots,x_{n}\in E$, then one says that the point process is determinantal and $K$ is called the correlation kernel of the process.
\end{defn}

 Let us give two classical examples of determinantal processes. For this we recall a classical way to show that a point process is determinantal and to obtain its correlation kernel (see Borodin  \cite{Borodinseul}). Suppose that $\mu_n$ is a probability measure on $E^n$ having a density $u_n$ with respect to the measure $m^{\otimes n}$ on $E^n$ defined by \begin{align}  \label{detaucarre}  u_n(\lambda_1,\cdots,\lambda_n)=C \det(\psi_i(\lambda_j))_{n\times n}\det(\phi_i(\lambda_j))_{n\times n},\end{align}where $C$ is a positive constant and the functions $\psi_i$'s and $\phi_i$'s are measurable functions such that $\psi_i\phi_j$ is integrable for any $i$, $j$. We denote $A=(A_{ij})_{1\le i,j\le n}$ the matrix defined by  $$A_{ij}=\int_{E}\psi_i(x)\phi_j(x)m(dx).$$
 Then  $A$ is invertible and the proposition 2.2 of \cite{Borodinseul} claims that the image measure of the probability measure $\mu_n$ by the map $(\lambda_1,\cdots,\lambda_n)\mapsto \sum_{i=1}^{n}\delta_{\lambda_i}$ is a determinantal point process with correlation kernel defined by \begin{align} \label{kernel} K(x,y)=\sum_{i,j=1}^{n}\psi_i(x) (A^{-1})_{ij} \phi_j(y), \, x,y\in E.\end{align}
Suppose for example that $E=\R$, $m$ is the Lebesgue measure and \begin{align}\label{vandaucarre} u_n(\lambda)=\Delta_n(\lambda)^2\prod_{i=1}^nw(\lambda_i),\end{align} where $w$ is a positive  integrable function on $\R$ such that $\int x^kw(x)\, dx <+\infty$ for any $k$. If  $(p_i)_{i\ge 0}$ is a sequence of polynomials such that the $p_i$'s have degree $i$ and  satisfy $$\int_{E}p_i(x)p_j(x)w(x)\, dx=\delta_{ij}, \, i,j\in \N,$$ then $\Delta_n(\lambda)$ is proportional to $\det(p_{i-1}(\lambda_j))$ and  the correlation kernel is  \begin{align} \label{kernelortho} K(x,y)=\sum_{i=1}^{n}p_i(x) w(x)^{\frac{1}{2}}p_i(y)w(y)^{\frac{1}{2}}, \, x,y\in \R.\end{align}
This is an usual way to  show that the point processes associated to the eigenvalues of the random matrices from the GUE or the LUE are determinantal. For these cases, the orthogonal polynomials which have to be considered to get a kernel of the form (\ref{kernelortho})  are respectively the Hermite and the Laguerre ones. Let us  now briefly  describe the cases of the GUE($\F$) and the LUE($\F$) when $\F=\R$ or $\F=\H$. 

We let $\epsilon=1$ if $\F=\H$ or $\F=\R$ with $n$ odd, and $\epsilon=0$ otherwise.  Weyl integral's formula (\cite{Helgason}, Thm. I.5.17) implies that there exists a constant $C$ such that the vector of the positive eigenvalues of  a  random matrix $M$ from the GUE($\F)$, $\F=\H,\R$, has a density $f_{gue}$  defined  on $\R_+^{\tilde n}$  by 
\begin{align}\label{densityGUE} f_{gue}(\lambda)=C\,d_n(\lambda)^2\prod_{i=1}^{\tilde n} e^{-\frac{1}{2} \lambda_i^2}1_{\R_+}(\lambda_i).\end{align}
Lemma \ref{rouge} shows that the density $f_{gue}$ has the form (\ref{detaucarre}) with $\tilde n$ instead of $n$ and for instance $\psi_i(x)=\phi_i(x)=x^{2i-2+\epsilon}e^{-\frac{1}{4}x^2}$. Thus the associated point process is determinantal. Since the Hermite polynomials have only monomials of same parity, it shows that the correlation kernel is $$\sum_{i=1}^{\tilde n}h_{2i-2+\epsilon}(x)h_{2i-2+ \epsilon}(y)e^{-\frac{1}{4}(x^2+y^2)}\quad x,y\in \R_+.$$
 Actually this situation corresponds to a classical one. It suffices to make the change of variable $\lambda_i'=\lambda_i^2$ in (\ref{densityGUE}) to get the classical form (\ref{vandaucarre}) with $w(x)=x^{\alpha}e^{-\frac{1}{2} x}$, where $\alpha=\frac{1}{2}$ when $\F=\H$ or $\F=\R$, $n$ is odd and $\alpha=-\frac{1}{2}$ when $\F=\R$, $n$ is even. The orthogonal polynomials to consider are thus the Laguerre ones. 
 
For $\F=\H$ or $\R$, theorem \ref{eigenvaluesLUE} shows that the density of the positive eigenvalues of a random matrix from the LUE$_{n,k}(\F$), for $k\ge n$, has the form (\ref{detaucarre}) with $\tilde n$ instead of $n$ and for instance $\psi_i(x)=x^{2i-2+\epsilon+\tilde  k -\tilde n}$, $\phi_i(x)=x^{i-1}e^{-cx}$. Thus, the associated  point processes are determinantal and their correlation kernels are given by (\ref{kernel}). Nevertheless, it is important to notice that the orthogonal polynomials  method can't be applied here. 

In the following, we will study more generally the determinantal aspect of the interlaced  processes considered in the previous sections. Using the explicit formula that we got, we write their measures as a product of determinant and use the method Johansson \cite{Johansson}  and Borodin et al.\ \cite{Borodin}  to show that a large class of them are determinantal and to compute their correlation Kernels.

\subsection{"Triangular" interlaced processes} 
The first type of interlaced point process that we consider is the one associated to the eigenvalues of the main minors of an invariant random matrix in $\mathcal{P}_n(\mathbb{F})$.  In this case $E=\{1,\cdots,n\} \times  \mathbb{R}$ and the reference measure $m$ is the product of the counting measure on $\{1,\cdots,n\}$ with the Lebesgue measure on $\R$ when $\F=\C$, on $\R_+$ when $\F=\H$ and $\F=\R$.  

\begin{defn} We say that an invariant random matrix $M$ in $\mathcal{P}_n(\mathbb{F})$ belongs to the class $\mathcal{K}$ if the eigenvalues of $M$ for $\mathbb{F}=\mathbb{C}$, and the positive eigenvalues of $M$ for $\mathbb{F}=\mathbb{R}$ or $\mathbb{F}=\mathbb{H}$, have a joint density with respect to the Lebesgue measure on $\mathbb{R}^{\tilde{n}}$ proportional to $$d_n(\lambda)\det(\psi_j(\lambda_i))_{\tilde{n}\times \tilde{n}},$$ where the $\psi_i$'s  are real continuous functions on $\mathbb{R}$, equal to zero on $\mathbb{R}_-$ for $\mathbb{F}=\mathbb{R}$ and $\F=\mathbb{H}$, and such that for all $k\in \mathbb{N}$, the function $x\mapsto x^k\psi_i(x)$ is integrable on $\mathbb{R}$.
\end{defn}
Many invariant ensembles belong to the class $\mathcal{K}$, especially the random matrices from the GUE($\F$) and the LUE($\F$).   
\begin{theo} \label{theodet} Let $M$ be an invariant random matrix in $ \mathcal{P}_n(\mathbb{F})$, which belongs to the class $\mathcal{K}$. Let us consider the random vector $X=X(M)$ 
and the associated point process $\Xi$ on $E$ defined by 
\begin{align*}  \Xi=\sum_{k=1}^{n}\sum_{i=1}^{k}\delta_{(k,X_{i}^{(k)})}  \textrm{when $\mathbb{F}=\mathbb{C},\mathbb{H}$} , \mbox{ and } \Xi=\sum_{k=1}^{n}\sum_{i=1}^{\tilde{k}}\delta_{(k,\vert X_{i}^{(k)}\vert)}  \textrm{when $\mathbb{F}=\mathbb{R}$}.\end{align*}
Then 

$(i)$ The point process $\Xi$ is  determinantal.

$(ii)$ The correlation kernel of $\Xi$ is, for $(r,x), (s,y) \in E$, 
\begin{align*} 
K((r,x),(s,&y))=-\frac{(y-x)^{c(s-r)-1}}{(c(s-r)-1)!}1_{\{s>r,\, y\ge x \}}\\ &+\alpha\sum_{k=1}^{\tilde{n}}\psi_{r-k}^r(x)\int\frac{\partial^{c(n-s)}d_n}{\partial z_{k}^{c(n-s)}}(z_1, \cdots,z_{k-1},y,z_{k+1},\cdots,z_{\tilde{n}})\prod_{ \substack{i=1\\ i\ne k}}^{\tilde{n}}\psi_{i}(z_i)\, dz_i
\end{align*}
where  $\psi_{r-k}^r(x)=
\int_x^{+\infty}\frac{1}{(c(n-r)-1)!}(z-x)^{c(n-r)-1}\psi_{k}(z)\,dz$, if $r<n$,  $\psi_{n-k}^n(x)=\psi_{k}(x)$ and $\alpha^{-1}=\int d_n(z)\prod_{i=1}^{\tilde{n}}\psi_{i}(z_i)\, dz_i$
\end{theo}
 We observe that $(X^{(1)}(M),\cdots,X^{(n)}(M))$ when $M$ is an invariant random matrix  in $\mathcal{P}_n(\mathbb{H})$, has the same  law as $(X^{(3)}(N),X^{(5)}(N),\cdots,X^{(2n+1)}(N))$ when $N$ is an invariant random matrix in  $\mathcal{P}_{2n+1}(\mathbb{R})$, provided that  $X^{(n)}(M)$ has the same law as $X^{(2n+1)}(N)$. So the quaternionic case is deduced from the real odd one in the previous theorem (see  \ref{relquatreal}).  
 
 \begin{cor} \label{corodet} Under the hypothesis of the previous theorem, suppose that  we can write  $d_n(\lambda)=\det(\chi_{i}(\lambda_{j}))_{\tilde{n}\times\tilde{n}}$, where $(\chi_{k})_{k\ge 1}$ is a sequence of real functions on $\mathbb{R}$ such that the $\chi_{i}\psi_{j}$'s are integrable on $\mathbb{R}$ and $\int_\mathbb{R} \chi_{i}(x)\psi_{j}(x)dx=\delta_{ij}$.  Then 
$$
K((r,x),(s,y))=-\frac{(y-x)^{c(s-r)-1}}{(c(s-r)-1)!}1_{\{s>r,\, y\ge x\}} +\sum_{k=1}^{\tilde{n}}\psi_{r-k}^r(x)\frac{d^{c(n-s)}\chi_{k}}{d x^{c(n-s)}}(y).
$$
\end{cor}
 If the radial part of $M$ is  deterministic and equal to $\lambda$ in the interior of the Weyl chamber, the theorem and its corollary remain true up to slight modifications, replacing $\psi_{i}(z)\, dz$ by the Dirac measure $\delta_{\lambda_i }(dz)$ for $\mathbb{F}=\mathbb{C}$ and by $\delta_{\vert \lambda_i\vert}(dz)$ for $\mathbb{F}=\mathbb{R}$, in the  kernel and the counting measure on $\{1,\cdots,n\}$ in the reference measure by the counting measure on $\{1,\cdots,n-1\}$. Let us describe some   applications  before making the proofs of the theorem and its corollary. Recall that we let $\epsilon$ be equal to $1$ if $n\notin 2\mathbb{N}$ and $0$ otherwise.
 
\medskip

\noindent {\bf The Gaussian case:} GUE($\mathbb{F}$). As we have seen a  standard Gaussian variable $M$ in $\mathcal{P}_n(\mathbb{F})$   satisfies the hypothesis of the theorem with $\psi_i(x)=x^{i-1}e^{-\frac{1}{2} x^2}$ when $\F=\mathbb{C}$ and $\psi_i(x)=x^{2i-2+ \epsilon}e^{-\frac{1}{2} x^2}1_{\{x>0\}}$ otherwise. Besides, the hypothesis of the corollary are satisfied if we let   $\chi_i=h_{i-1}$ when $\F=\mathbb{C}$,  $\chi_i=h_{2i-2+\epsilon}$ when $\F=\mathbb{R}$ and $\chi_i=h_{2i-1}$ when $\F=\H$,  where $(h_i)_{i\ge 0}$ is the sequence of normalized Hermite  polynomials for the weight $e^{-\frac{1}{2} x^2}$, such that $h_i$ has degree $i$.

In the case of the GUE($\C$),  the corollary was obtained by Johansson and Nordenstam \cite{JohanssonNordenstam}, and Okounkov and Reshetikhin \cite{Okounkov}. The following proposition, which provides the correlation kernel for the minor process associated to a matrix from the GUE$_\infty(\R)$, has been announced in \cite{Defosseux}.  Forrester and Nordenstam posted a proof on arxiv a few weeks later in \cite{ForresterNordenstam}.  
\begin{prop} \label{GUER} Let $M$ be a standard Gaussian variable in $\mathcal{P}_\infty(\mathbb{R})$. We consider the radial part  $X^{(k)}\in \mathbb{R}^{\tilde{k}}$  of the  main minor of order $k$ of $M$. Then the point process  $\sum_{k=1}^{+\infty}\sum_{i=1}^{\tilde{k}}\delta_{(k, \vert X^{(k)}_{i}\vert)}$ is determinantal on $\mathbb{N}^*\times \mathbb{R}_{+}$ with correlation kernel
\begin{displaymath}
\begin{array}{ll} 
R((r,x),(s,y))=&-\frac{1_{\{r<s\}}}{(s-r-1)!} (y-x)^{s-r-1}1_{\{y\ge x\}}\\&+\sum_{i=1}^{\tilde{r}\wedge \tilde{s}}\frac{((r-2i)!)^{1/2}}{((s-2i)!)^{1/2}}h_{s-2i}(y)h_{r-2i}(x)e^{-\frac{1}{2}x^{2}} \\
&+ \sum_{i= \tilde{r}+1}^{\tilde{s}}\frac{h_{s-2i}(y)}{((s-2i)!\sqrt{\pi})^{1/2}}\int^{+\infty}_{x} \frac{(z-x)^{2i-r-1}}{(2i-r-1)!}e^{-\frac{1}{2}z^{2}}\, dz.
\end{array}
\end{displaymath}
\end{prop}  
\proof  Let $n$ be an odd integer and $M$ be a standard Gaussian variable in  $\mathcal{P}_{n}(\mathbb{R})$. The matrix $M$ belongs to the class $\mathcal{K}$ with $\psi_i=h_{2i-1}$, $i=1,\cdots,\tilde{n}$. The functions $\chi_i=h_{2i-1}$, $i=1,\cdots,\tilde{n}$, satisfy the hypothesis of corollary \ref{corodet}. So the point process  $\sum_{k=1}^{n}\sum_{i=1}^{\tilde{k}}\delta_{(k,\vert X^{(k)}_{i}\vert)}$ is determinantal on $\mathbb{N}^*\times \mathbb{R}_+$ and its correlation kernel $K$ is equal to
\begin{align*}K((r,x),(s,y))=-\frac{1_{\{s<r\}}}{(r-s-1)!}&(y-x)^{r-s-1}1_{\{y\ge x\}} +\sum_{k=1}^{\tilde{n}}h_{2k-1}^{(n-s)}(y)\xi_{k}(r,x)\end{align*}
where $\xi_{k}(r,x)=\int_{x}^{\infty}\frac{(z-x)^{n-r-1}}{(n-r-1)!}h_{2k-1}(z)e^{-\frac{1}{2}z^2}\, dz.$
Let us consider $$H_{n}(x)=(-1)^{n}e^{\frac{1}{2}x^{2}}\frac{d^{n}}{dx^{n}}e^{-\frac{1}{2}x^{2}}.$$ We recall that $h_{n}=\frac{1}{(n!\sqrt{\pi})^{1/2}}H_{n}$ and $h'_{n}=\sqrt{n}h_{n-1}$. Letting $i=\tilde{n}-k+1$, we get  that 
\begin{align*}
\sum_{k=1}^{\tilde{n}}h_{2k-1}^{(n-s)}&(y)\xi_{k}(r,x)
 = \sum_{i=1}^{\tilde{n}} \Big[\frac{(n-2i)!}{(s-2i)!}\Big]^{1/2}h_{s-2i}(y)\xi_{k}(r,x).
\end{align*}
Integrating by part  we get that when $ i\ge \tilde{r}+1$,
\begin{align*}
\xi_{k}(r,x)=\frac{1}{((n-2i)!\sqrt{\pi})^{1/2}} \int_{x}^{\infty}\frac{(z-x)^{2i-r-1}}{(2i-r-1)!}e^{-\frac{1}{2}z^2}\, dz,
\end{align*}
and when $ i\le \tilde{r}$,
\begin{align*}
\xi_{k}(r,x)=    \Big[\frac{(r-2i)!}{(n-2i)!} \Big]^{1/2}h_{r-2i}(x)e^{-\frac{1}{2}x^{2}},
\end{align*}
which proves the proposition.  \qed\medskip

\noindent{\bf The Laguerre case:} LUE($\mathbb{F}$).  A random matrix from the LUE$_{n,k}(\F)$ satisfies for $k\ge n$ the hypothesis of theorem \ref{theodet}  with $\psi_i(x)=x^{i-1+\tilde k-\tilde n} e^{-c x}1_{\R_+}(x)$.  Those of the corollary \ref{corodet} are satisfied only when $\F=\C$ with $\psi_i=\chi_i=L_{i}$, where  $(L_i)_{i\ge 0}$ is the sequence of normalized Laguerre  polynomials for the weight $x^{k-n}e^{-x}$, such that $L_i$ has degree $i$. 

\medskip

Let us now prove theorem  \ref{theodet}. The main point of its proof is the following lemma which is an application of \cite{Borodin}.  For $f,g:\mathbb{R}\times \mathbb{R} \to \mathbb{R}_+$, $h: \mathbb{R} \to \mathbb{R}_+$ and $ x,y\in \mathbb{R} $, we write, when it's meaningful,
\begin{align} \label{defnstar}
\begin{array}{l} 
(f*g)(x,y)=\int_\mathbb{R} f(x,z)g(z,y)\, dz, \quad (f*h)(x)=\int_\mathbb{R} f(x,z)h(z) \, dz,  \end{array}\\
f^{(1)} = f, \quad \quad f^{(r)} = f*f^{(r-1)}, \quad  \mbox{ if }r\ge 1, \quad \quad f^{(r)}=0, \quad \mbox{ if } r\le 0. \notag
\end{align}

\begin{lem} \label{borodin} Let $M$ be an invariant random matrix as in theorem \ref{theodet} and $\Xi$ the associated point process. We suppose that the support of the functions $\psi_i$'s and $\phi_i$'s are included in an interval $]a,b[$. Then the correlation kernel of $\Xi$ is defined by 
\begin{itemize}
\item when $\mathbb{F}=\mathbb{C}$,
\begin{align*} 
K((r,x),(s,y))&=-\phi^{(s-r)}(x,y)+\sum_{k=1}^{n}\psi_{r-k}^r(x)\sum_{l=1}^{s}(A^{-1})_{kl}\phi^{(s-l+1)}(a,y),\end{align*}
\item when $\mathbb{F}=\mathbb{R}$,
\begin{align*} 
K((r,x),(s,y))&=-\phi^{(s-r)}(x,y)+\sum_{k=1}^{\tilde{n}}\psi_{r-k}^r(x)\sum_{l=1}^{\tilde{s}}(B^{-1})_{kl}\phi^{(s-2l+1)}(0,y),
\end{align*}
\end{itemize}
where $\phi(x,y)=1_{[x,\infty)}(y)$, and $A$ and $B$ are invertible matrices defined by $A_{ij}=\phi^{(n-i+1)}*\psi_{j}(a)$, $i,j=1, \cdots,n$, and $B_{ij}=\phi^{(n-2i+1)}*\psi_{j}(0)$, $i,j=1,\cdots,\tilde{n}$.
\end{lem}
 \proof Let us consider $\Lambda=(\Lambda^{(1)},\cdots,\Lambda^{(n)})$, where   $\Lambda^{(r)}=\sigma_r (X^{(r)}(M))$,  when $\mathbb{F}=\mathbb{C}$ and $\Lambda^{(r)}=\sigma_r (\vert X^{(r)}(M)\vert)$,  when $\mathbb{F}=\mathbb{R}$, where the $\sigma_r$'s are independent random permutations of the coordinates, uniformly distributed and independent from $M$. The reason why we introduce these permutations is that we have to work with symmetric densities. Since the random matrix $M$ is invariant,  theorem \ref{theoGT}, lemma \ref{volGT} and identity (\ref{formuleentrelac}) imply that the density of $(\Lambda^{(1)},\cdots, \Lambda^{(n)})$  is proportional to the function $f$  defined by :
 \begin{itemize}
 \item when $\mathbb{F}=\mathbb{C}$,  with the convention that $x_{r}^{(r-1)}=a$,
 $$f(x^{(1)},\cdots,x^{(n)})=\det(\psi_j(x^{(n)}_i))_{n\times n}\prod_{r=1}^{n}\det \big(1_{\{x_j^{(r)}> x_{i}^{(r-1)}\}}\big)_{r\times r},$$ 
  \item  when $\mathbb{F}=\mathbb{R}$,  with the convention that  $x_{r}^{(2r-1)}=0$,
 $$f(x^{(1)},\cdots,x^{(n)})=\det(\psi_j(x^{(n)}_i))_{\tilde n\times \tilde n}\prod_{r=1}^{n}\det \big(1_{\{x_j^{(r)}> x_{i}^{(r-1)}\}}\big)_{\tilde r\times \tilde r}.$$
 \end{itemize}
We consider a sequence $(\tilde \Lambda_N)_N$ of discrete random variables such that $\tilde \Lambda_N$ belongs to $\frac{1}{N}GT_{n,\mathbb{Z}}$ and $\P(\tilde \Lambda_N =(x^{(1)},\cdots,x^{(n)}))$ is proportional to $f(x^{(1)},\cdots,x^{(n)})$. Then lemma 3.4 in \cite{Borodin}, slightly modified  for $\mathbb{F}=\mathbb{R}$ (see \cite{ForresterNordenstam} for details),  implies that the associated point process is determinantal with a correlation kernel $K_N$ obtained from $K$ replacing the Lebesgue measure on $\R$ in identities (\ref{defnstar}) by the counting measure on $\frac{1}{N}\Z$. We get the lemma letting $N$ goes to infinity.  \qed\medskip

{\sc Proof of theorem \ref{theodet}.}  We write the proof for $\mathbb{F}=\mathbb{R}$. We use the lemma \ref{borodin} and its notations. We have, for $r\ge 1$, $$\phi^{(r)}(0,y)=\frac{y^{r-1}}{(r-1)!}1_{\{y\ge 0\}}.$$ Thus, $\phi^{(s-2l+1)}(0,y)=\frac{\partial ^{n-s}}{\partial y ^{n-s}}\phi^{(n-2l+1)}(0,y)$, $l=1, \cdots,\tilde n$, and 
\begin{align*} 
\sum_{l=1}^{s}(B^{-1})_{kl}\phi^{(s-2l+1)}(0,y) = \frac{\partial ^{n-s}}{\partial y ^{n-s}}\sum_{l=1}^{n}(B^{-1})_{kl}\phi^{(n-2l+1)}(0,y). \end{align*}
Let us denote $s_{lk}(B)$ the matrix obtained from $B$ by suppressing the $l^{th}$ line and the $k^{th}$ column. We have $$(B^{-1})_{kl}=\frac{(-1)^{k+l}}{\det(B)}\det(s_{lk}(B))_{\tilde n-1\times \tilde  n-1}.$$
Thus 
\begin{align*} 
\sum_{l=1}^{\tilde n}&(B^{-1})_{kl}\phi^{(n-2l+1)}(0,y)=\sum_{l=1}^{\tilde n}\frac{(-1)^{k+l}}{\det(B)}\det(s_{lk}(B))_{\tilde n-1\times \tilde n-1}\phi^{(n-2l+1)}(0,y) \\
&= \sum_{l=1}^{\tilde n}\frac{(-1)^{k+l}}{\det(B)}\det(\phi^{(n-2i+1)}*\psi_j(0))_{\substack{i\ne l\\ j\ne k}}\, \phi^{(n-2l+1)}(0,y)\\
&= \sum_{l=1}^{\tilde n}\frac{(-1)^{k+l}}{\det(B)}\int_{\mathbb{R}^{n-1}} \det(\phi^{(n-2i+1)}(0,z_j))_{\substack{i\ne l\\ j\ne k}}\phi^{(n-2l+1)}(0,y)\prod_{ \substack{j=1\\ j\ne k}}^{\tilde n}\psi_j(z_j) dz_j \\
&= \frac{1}{\det(B)}\int_{\mathbb{R}^{n-1}} \det(\phi^{(n-2i+1)}(0,z_j))_{\tilde n\times \tilde n}\prod_{ \substack{j=1\\ j\ne k}}^{\tilde n}\psi_j(z_j) dz_j   \textrm{ , letting } z_k=y. \end{align*}
Moreover, if $V_n$ is the function introduced at definition \ref{asymptoticdim}, we have
\begin{align*}   \det(\phi^{(n-2i+1)}(0,z_j))_{\tilde n\times \tilde n}=\det(\frac{z_j^{n-2i}}{(n-2i)!}1_{\{z_i\ge 0\}})_{\tilde n\times\tilde n}= V_n(z)\prod_{i=1}^{\tilde n}\frac{1_{\{z_i\ge 0\}}}{(n-2i)!},\end{align*} which achieves the proof for $\F=\R$. We get the theorem letting $a$ and $b$ go to $-\infty$ and $+\infty$.   The case  $\mathbb{F}=\mathbb{C}$ is quite similar. We deduce the quaternionic case from the real odd one. \qed\medskip

{\sc Proof of corollary \ref{corodet}}. The corollary is deduced from the theorem using the identities $$\int_{\mathbb{R}^{\tilde{n}}}\det(\chi_i(z_j))_{\tilde{n}\times \tilde{n}} \prod_{j=1}^{\tilde{n}}\psi_{j}(z_j)\, dz_j=\det (\int \chi_i(z)\psi_j(z)\,dz)_{\tilde{n}\times \tilde{n}}=1$$ $$\int_{\mathbb{R}^{\tilde{n}-1}}\det(\chi_i(z_j))_{\tilde{n}\times \tilde{n}}\prod_{ \substack{j=1\\ j\ne k}}^{\tilde{n}}\psi_j(z_j)dz_j=\det (a_{ij})_{\tilde{n}\times \tilde{n}}=\chi_k(y),$$ where $a_{ij}=\delta_{ij}$, $j\ne k$ and $a_{ik}=\chi_{i}(y)$, $i=1,\cdots, \tilde{n}$.   \qed

\subsection{"Rectangular" interlaced processes} 
In section 5, considering successive rank one perturbations, we have constructed Markov processes which have a remarkable property: two successive states satisfy some interlacing conditions. Thus we got interlaced random configurations on $\mathbb{N}\times \mathbb{R}$. More precisely, let $(M_k)_{k\ge 1}$ be a sequence of independent standard Gaussian variables in $\mathcal{M}_n(\mathbb{F})$.   For $\lambda$ in the interior of the Weyl chamber, we consider the process $(R^{(k)})_{k\ge 1}$, where $R^{(k)}$ is the radial part of $\Omega_n(\lambda)+\sum_{i=1}^{k} M_i\Omega_n^1M_i^*$, and the associated point process $\Xi_{\lambda}=\sum_{k=1}^{m}\sum_{i=1}^{\tilde{k}}\delta_{k,R^{(k)}_{i}}$. Since interlacing conditions and function $d_n$ can be written as a determinant  for $\mathbb{F}=\mathbb{C}$, $\F=\mathbb{H}$ or $\F=\R$ with $n$ odd, our proposition \ref{Markovpertu} shows that the hypothesis of  proposition 2.13 in \cite{Johansson} are satisfied in these cases and that the point process $\Xi_\lambda$ is determinantal.  In the even real case, we don't know if this remains true.   Thus we have the following proposition.
\begin{prop}  Let $\mathbb{F}=\mathbb{C}$, $\F=\mathbb{H}$ or $\F=\R$ with $n$ odd. Let $(M_i)_{i\ge 1}$ be a sequence of independent standard Gaussian variable in $\PP_n(\F) $ and $\lambda\in \mathbb{R}_+^{\tilde n}$ such that $\lambda_1>\cdots>\lambda_{\tilde n}$. Let us consider  the point process   $ \Xi_ \lambda =\sum_{k=1}^{m}\delta_{k,R_{i}^{(k)}},$
where $R_{i}^{(k)}$ is the $i^{th}$ positive eigenvalue of $\Omega_n(\lambda)+\sum_{i=1}^{k}M_i\Omega_n^1M_i^*$.

Then,  

$(i)$ the point process $\Xi_ \lambda $ is determinantal on $\{1,\cdots,m\}\times \R_+$.

$(ii)$ The correlation kernel of $\Xi_ \lambda $ is 
\begin{align*} 
K_\lambda((r,x),(s,y))&=-\phi^{(s-r)}  +\sum_{i,j=1}^{\tilde n}\phi^{(m-r)}*\psi(x,i) (A^{-1})_{ij}\phi^{(s)}(\lambda_j,y),
\end{align*}
where  $\phi(x,y)=1_{y\ge x}e^{-(y-x)}$, $\psi(x,i)=x^{i-1}$ when $\F=\C$,  $\phi(x,y)=e^{-c(x+y)}(e^{2c(x\wedge y)} -1)$, $\psi(x,i)=x^{2i-1}$ when $\F=\H,\R$  and $A$ is an invertible matrix defined by $A_{ij}=\phi^{(m)}*\psi(\lambda_i,j)$.

\end{prop}
\part{\sc Orbit measures}
\bigskip

\section{Approximation of orbit measures}

 \subsection{Introduction}
 
  Let $K$ be a compact connected Lie group with Lie algebra $\mathfrak{k}$.  We equip $\mathfrak{k}$ with an Ad($K$)-invariant inner product. This allows us to identify $\mathfrak{k}$ and its dual $\mathfrak{k}^*$. The group $K$ acts on $\mathfrak{k}$ by the adjoint action $Ad$ and on $\mathfrak{k}^*$ by duality  by the coadjoint action.   By definition, the coadjoint orbit through $ \lambda \in \mathfrak{k}^*$ is the set 
  $${\mathcal O}(\lambda)=\{Ad(k)\lambda, k \in K\}.$$
The (normalized) orbit measure is the image on ${\mathcal O}(\lambda)$ of the normalized Haar measure $m_K$ on $K$, i.e.\ the distribution of $Ad(U)\lambda $ where $U$ is a random variable with law $m_K$.
  Computations for invariant ensembles of random matrix theory rest on a detailed analysis of either the sum (convolution) of orbit measures on ${\mathcal O}(\lambda)$ and ${\mathcal O}(\mu)$, where $\lambda, \mu \in \mathfrak{k}^*$, or their projection $p$ on the dual Lie algebra of a subgroup $H$. Let us recall two basic facts of Kirillov's orbit method (\cite{Kirillov1}, \cite{Kirillov2}, p.xix). In his famous "User's guide" the third and fifth rules are the following (we denote by $V_\lambda$ the irreducible module associated to $\lambda$): 
  
\medskip
Rule 3:
    If  what you want is to describe the spectrum of $Res^K_H V_\lambda$ then what you have to do is to take the projection $p({\mathcal O}(\lambda))$ and split into $Ad(H)$ orbits.
    \medskip
    
Rule 5:   If  what you want is to describe the spectrum of the tensor product of $V_\lambda\otimes V_\mu$  then what you have to do is to take the arithmetic sum ${\mathcal O}(\lambda)+{\mathcal O}(\mu)$ and split into $Ad(K)$ orbits.

\medskip
 \noindent Our method is to use these two rules, but in the reverse order: we interchange      "what you want" and "what you have to do".
   First we prove a version of a theorem of Heckman which will allow us to give an effective way to compute the measures   on dominant weights defined with the help of the so called \textit{branching rules}. Then we obtain the convolution or the projection of orbit measures using these rules.
\subsection{Characters} Let $K$ be a connected compact Lie group with Lie algebra $\mathfrak{k}$ and complexified Lie algebra $\mathfrak{k_\C}$.  By compactness, without loss of generality, we can suppose that $K$ is contained in a unitary group, and then the adjoint and the coadjoint actions are given by $Ad(k)x=kxk^*, k \in K, x \in  \mathfrak{k} \mbox{ or } \mathfrak{k^*}$. We choose a maximal torus $T$ of $K$ and we denote  by $\mathfrak{t}$ its Lie algebra. 
 We consider  the roots system $R=\{\alpha\in \mathfrak{t}^*: \exists X \in \mathfrak{k_\C}\setminus \{ 0 \},\, \forall H\in \mathfrak{t} ,\, [H,X]=i\alpha (H)X\}$, the coroots $h_\alpha=2\alpha/\langle\alpha,\alpha\rangle$, $\alpha\in R$.  We choose  the set $\Sigma$ of simple roots of $R$. We introduce  the corresponding set  $R^+$ of positive roots and  the (closed) Weyl chamber $\mathcal{C}=\{\lambda\in \mathfrak{t}^*: \langle\lambda,\alpha \rangle \ge 0 \textrm{ for all } \alpha\in \Sigma\}$.
The set of weight is $P=\{\lambda\in \mathfrak{t}^*: \langle h_\alpha,\lambda \rangle \in \mathbb{Z}, \textrm{ for all } \alpha\in R\}$ and the set of dominant weights is $P^+=P\cap \CC$. We denote by
 $W$ the Weyl group.
 
For $\lambda\in P^{+}$, we denote by $V_\lambda$ the irreducible  $\mathfrak{k}$-module with highest weight $\lambda$ and $\dim(\lambda)$ the dimension of $V_\lambda$.  Its  character $\chi_{\lambda}$  is the function  on $\mathfrak{t}$ defined by,
$$\chi_{\lambda}(\zeta)=\sum_{\mu \in P} m({\mu},{\lambda}) e^{i \langle\mu,\zeta \rangle}, \quad \zeta\in \mathfrak{t},$$ where 
$m({\mu},{\lambda})$ is the multiplicity of the weight 
$\mu$ in the $\mathfrak{k}$-module $V_\lambda$. Notice that we use representations of the Lie algebra rather than representations of the compact group.
We denote $\rho=\frac{1}{2}\sum_{\alpha\in
R^{+}}\alpha$, the half sum of positive roots. The dimension of the module $V_\lambda$ is given by $\chi_{\lambda}(0)$. Recall the Weyl dimension formula  (see Knapp \cite{Knapp}, Thm V.5.84): 
\begin{align}\label{formuladim} \chi_{\lambda}(0)=\prod_{\alpha\in
R^{+}}\frac{ \langle\lambda+\rho,\alpha \rangle}{\langle\rho,\alpha \rangle}\end{align} 
and the Weyl character formula for the Lie algebra of a compact  Lie group (see Knapp \cite{Knapp}, Thm. V.5.77):
 \begin{prop}[Weyl character formula] The character $\chi_{\lambda}$ is equal to 
\begin{align*}
\chi_{\lambda}(\zeta)=\frac{\sum_{w\in
W} \det(w)e^{i
\langle w(\lambda+\rho),\zeta \rangle}}{\sum_{w\in
W} \det(w)e^{i
\langle w(\rho),\zeta \rangle}}.
\end{align*} 
\end{prop}
In this formula, the denominator is also equal to the product $\prod_{\alpha\in
R^{+}}(e^{\frac{i}{2} \langle\alpha,\zeta \rangle}-e^{-\frac{i}{2} \langle\alpha,\zeta \rangle})$. When $K=U_n(\mathbb{C})$ and $\lambda$ have integer coordinates, the characters are the classical Schur functions (see for instance \cite{Fulton}).

Let us recall some properties of invariant probability measures on the adjoint orbits of the group $K$.   Let, for $ z\in \mathfrak{t}\oplus i\mathfrak{t}$, $\lambda\in \mathfrak{t}^*$, $$h(z)=\prod_{\alpha\in
R^{+}}\langle\alpha,z\rangle , \, d(\lambda)=\prod_{\alpha\in
R^{+}}\langle \alpha, \lambda \rangle /\langle \alpha,\rho\rangle .$$
The quantity $d(\lambda)$ can be interpreted as the Liouville measure of the adjoint orbit ${\mathcal O}(\lambda)$ or as an asymptotic dimension.
For $\lambda\in \mathfrak{t}^*$, we introduce the function $\Phi_ \lambda$  on $\mathfrak{k}$ such that $\Phi_ \lambda(\zeta)=\Phi_ \lambda(k\zeta k^*)$ for all $\zeta \in \mathfrak{k}, k \in K$, and such that when  $\zeta \in \mathfrak{t},$
$$\Phi_ \lambda(\zeta)=\frac{\sum_{w\in W}
\det(w)e^{i \langle w \lambda,\zeta \rangle}}{h(i\zeta)d(\lambda)}.$$
We recall the Harish Chandra formula (see Helgason \cite{Helgason}, Thm II.5.35). In different contexts it is also known as the Kirillov formula for compact groups or the Iztkinson-Zuber formula. Recall that $m_K$ is the normalized Haar measure on $K$.
\begin{prop}  For $ \lambda \in \mathfrak{t}^{*}$, $\zeta\in \mathfrak{k}$
\begin{align} \label{HarishChandra} \int_{K} e^{i \langle k \lambda k^*,\zeta  \rangle} \, m_K(dk)= \Phi_ \lambda(\zeta).\end{align} \end{prop} 
This shows that $\Phi_ \lambda(\zeta)$ is a continuous function of $(\lambda,\zeta)$ and $\Phi_ \lambda(0)=1$.
\subsection{A version of Heckman's Theorem} 

We consider  a connected compact subgroup $H$ of $K$ with Lie algebra $\mathfrak{h}$. After maybe a conjugation, we can choose a maximal torus $S$ of $H$ included in $T$ (see for instance Knapp \cite{Knapp}). We denote its Lie algebra by $\mathfrak{s}$. The objects previously associated to $K$ are defined in the same way for $H$. In that case, we add an exponent or a subscript $H$ to them. 
For $\lambda\in P^+, \beta\in P^+_H$ we denote by $m^\lambda_H(\beta)$ the multiplicity of the irreducible $\mathfrak{h}-$module with highest weight $\beta$ in the decomposition into irreducible components of the $\mathfrak{k}-$module $V_{\lambda}$ considered as an $\mathfrak{h}-$module. Rules giving the value of the multiplicities $m^\lambda_H$ are called branching rules.  We have the following decomposition
\begin{align}V_{\lambda}=\oplus_{\beta\in P^{+}_{H}}m^{\lambda}_{H}(\beta)V^H_{\beta},\label{branching} \end{align}
where $V_\lambda$ is considered as an $\mathfrak{h}-$module and $V^H_{\beta}$ is an irreducible $\mathfrak{h}-$module with highest weight $\beta $.
This is equivalent to say that  
$m^{\lambda}_{H}$ is the unique function from $P_{H}^{+}$ to $\mathbb{N} $ satisfying the following identity: for all $\zeta \in \mathfrak{s}$,
\begin{align}\chi_{\lambda}(\zeta)=\sum_{\beta\in P^{+}_{H}}m^{\lambda}_{H}(\beta)\chi_{\beta}^{H}(\zeta).
\label{branchingform} \end{align}
For $x\in \mathfrak{k}^*$, let $\pi_H(x)$ be the orthogonal projection of $x$ on $\mathfrak{h}^*$. The intersection between the orbit of an element $x\in \mathfrak{k}^*$ under the coadjoint action of $K$ and the Weyl chamber $\mathcal{C}$ contains a single point that we call the radial part of $x$ and denote by $r(x)$. The same holds for $H$ and we denote by $r_H(x)$ the radial part of $x\in \mathfrak{h}^*$ in the Weyl chamber $\CC_H$ for the coadjoint action of $H$. We choose a sequence $\varepsilon_n >0$ which converges to $0$ as $n \to \infty$. The following theorem is a variant of theorem 6.4 in Heckman  \cite{Heckman}. We give a direct proof.
\begin{theo} \label{theoprinc}  Let $\lambda$ be in the Weyl chamber  $\mathcal{C}$ and $(\lambda_{n})_{n\ge 1}$ be a sequence of elements in
$P^{+}$  such that  
$\varepsilon_n\lambda_{n}$  converges to $\lambda $ as $n$ tends to $+\infty$. Then 

(i)  the sequence $(\mu_{n})_{n\ge 0}$ of probability measures on $\mathcal{C}_H$ defined by 
$$\mu_{n}=\sum_{\beta\in P^+_H}\frac{\dim_{H}(\beta)}{\dim(\lambda_{n})}m^{\lambda_{n}}_H(\beta)\delta_{\varepsilon_n\beta}$$
 converges to a probability measure $\mu$ which satisfies, for $\zeta\in \mathfrak{h}$,
\begin{align} \label{Fourierprojection}
\int_{\mathcal{C}_H}
{\Phi_{\beta}^{H}(\zeta)}\, \mu(d\beta) ={\Phi_ \lambda(\zeta)},
\end{align} 
  
  (ii)  $\mu$ is the law of $r_{H}(\pi_H(U\lambda U^*))$, where $U$ is distributed according to $m_K$.
\end{theo}
\proof 
Let $\zeta\in
\mathfrak{s}$. We have
$$\frac{\chi_{\lambda_{n}}(\varepsilon_{n}
\zeta)}{\chi_{\lambda_{n}}(0)}=\Phi_{\varepsilon_{n}\lambda_n+\varepsilon_{n}\rho}(\zeta)\prod_{\alpha\in R^{+}}\frac{i \langle\alpha,\varepsilon_{n}\zeta \rangle}{e^{\frac{i}{2} \langle\alpha,\varepsilon_{n}\zeta \rangle}-e^{-\frac{i}{2} \langle\alpha,\varepsilon_{n}\zeta \rangle}}.$$
On the other hand, 
\begin{align*}\frac{\chi_{\lambda_{n}}(\varepsilon_{n}
\zeta)}{\chi_{\lambda_{n}}(0)}&=\sum_{\beta\in P^{+}_{H}}\frac{\chi^{H}_{\beta}(\varepsilon_{n}\zeta)}{\chi^{H}_{\beta}(0)}\frac{m^{\lambda_{n}}_{H}(\beta)\chi^{H}_{\beta}(0)}{\chi_{\lambda_{n}}(0)}\\
&= \Big[\prod_{\alpha\in
R_{H}^{+}} \frac{i\langle \alpha,\varepsilon_{n}\zeta\rangle}{e^{\frac{i}{2} \langle\alpha,\varepsilon_{n}\zeta \rangle}-e^{-\frac{i}{2} \langle\alpha,\varepsilon_{n}\zeta \rangle}} \Big]\int_{\mathcal{C}_{H}}\Phi_{\beta+\varepsilon_{n}\rho_{H}}^{H}(\zeta)\, d\mu_{n}(\beta)
\end{align*} 
Therefore
$$\lim_{n \to +\infty}\int_{\mathcal{C}_{H}}\Phi_{\beta+\varepsilon_{n}\rho_{H}}^{H}(\zeta)\, \mu_{n}(d\beta)=\Phi_{\lambda}(\zeta).$$
The support of $\mu_n$ is contained in the convex hull of the orbit of $\varepsilon_n\lambda_n$ by the Weyl group. This implies that all the measures $\mu_n$ are contained in a same compact set. Uniform continuity on compact sets of the function $\Phi$ ensures that 
\begin{align}\label{convmu} \lim_{n \to +\infty}\int_{\mathcal{C}_{H}}\Phi_{\beta}^{H}(\zeta)\, \mu_{n}(d\beta)=\Phi_{\lambda}(\zeta).\end{align}
Let us consider the image $\gamma_{n}$ of the product measure $m_H\otimes
\mu_{n}$ by the function $(u,\beta)\in
H\times \mathcal{C}_{H}\mapsto u \beta u^*\in \mathfrak{h}^{*}$. The previous convergence and Harish-Chandra's formula applied to $H$ give that
$$\lim_{n\to \infty}\int_{\mathfrak{h}^{*}}e^{i \langle x, \zeta \rangle}\, \gamma_{n}(dx)=\Phi_{\lambda}(\zeta).$$ By invariance of the Haar measure on $H$ by  multiplication, this remains true for every $\zeta \in \mathfrak{h}$, which proves that the sequence 
of  measures $(\gamma_{n})_{n\ge 0}$ converges and consequently so does the sequence 
$(\mu_{n})_{n\ge 0}$. We denote  by 
$\mu$ the limit measure. The convergence (\ref{convmu}) shows that it satisfies the following identity, for $\zeta\in \mathfrak{h}$,
\begin{align*}
\int_{\mathcal{C}_H}
{\Phi_{\beta}^{H}(\zeta)}\, \mu(d\beta) ={\Phi_ \lambda(\zeta)},
\end{align*} 
which proves the first point of the theorem.
Applying the Harish-Chandra formula to $K$ and $H$  we get \begin{eqnarray*}
\int_K e^{i \langle u \lambda u^*,\zeta \rangle}\, m_K(du)&=&\int_K e^{i \langle \pi_{H}(u \lambda u^*),\zeta \rangle}\, m_K(du)\\ &=&\int_H
\int_{\mathcal{C}_H}
e^{i \langle u \beta u^*,\zeta \rangle}\, \mu(d \beta) \, m_{H}(du).
\end{eqnarray*} 
which gives the second point of the theorem. \qed\medskip

In the case when $H=T$, the limit measure $\mu$ is equal to ${d(\lambda)^{-1}}D_\lambda$ where $D_\lambda$ is the Duistermaat-Heckman measure associated to $\lambda$. The tensor product of irreducible representations being a particular restriction of representation, the theorem has the following corollary, which is due to Dooley et al.\ \cite{Dooley}. 
\begin{cor} \label{coroprinc}  Let $\lambda$ and $\gamma$ be in $\mathcal{C}$. 
Let $(\lambda_{n})_{n\ge 1}$ and $(\gamma_{n})_{n\ge 1}$ be two sequences of elements in
$P^{+}$   such that  
$\varepsilon_n\lambda_{n}$ and $\varepsilon_n\gamma_{n}$  respectively converge to $\lambda $ and $\gamma$, as $n$ tends to $+\infty$. 
Let us define  the sequence $(\nu_{n})_{n\ge 0}$ of probability measures on $\mathcal{C}$ by 
$$\nu_{n}=\sum_{\beta\in P^+}\frac{\dim(\beta)}{\dim(\lambda_{n})\dim(\gamma_{n})}M_{\lambda_{n},\gamma_n}(\beta)\delta_{\varepsilon_n\beta},$$
where $M_{\lambda_{n},\gamma_n}(\beta)$ is the multiplicity of the highest weight $\beta$ in the decomposition into irreducible components of $V_{\lambda_n}\otimes V_{\gamma_n}$. Then the sequence $(\nu_{n})_{n\ge 0}$ 
 converges to the law of the radial part of
$\lambda+U\gamma U^*$, where $U$ is distributed according to $m_K$.
\end{cor}
\proof  Let $V_{\lambda_{n}}$ and $V_{\gamma_{n}}$ be   irreducible $\mathfrak{k} $-modules with respective highest weight  $\lambda_{n}$ and $\gamma_{n}$. Let us consider the compact group $K\times K$. Then  $V_{\lambda_{n}}\otimes V_{\gamma_{n}}$ is an irreducible  $(\mathfrak{k}\times \mathfrak{k})-$module with highest  weight $(\lambda_{n}, \gamma_{n})$. Applying theorem \ref{theoprinc} to the compact group  $K\times K$ and the subgroup  $H=\{(k,k), k\in K\}$, we get that the associated sequence $(\nu_{n})_{n\ge 1}$ converges, when $n$ goes to $+\infty$, to the law of $r_{H}(\pi_{H}(Ad(W)(\lambda,\gamma)))$, $W$ being distributed according to the normalized Haar measure on   $K \times K$, i.e.\@ $W=(U,V)$, where $U$ and $V$ are independent random variables with distribution $m_K$. The facts that  $\pi_{H}(Ad(W)(\lambda,\gamma))=U\lambda U^*+V \gamma V^*$ and $r_H(U\lambda U^*+V \gamma V^*)= r_H(\lambda +U^*V \gamma V^*U) $ complete the proof  of the corollary.   \qed

\section{Orbit measures and invariant random matrices}
\subsection{} In this section, we apply theorem \ref{theoprinc} and its corollary  to invariant random matrices in $\mathcal{P}_n(\mathbb{F})$. For $\mathbb{F}=\mathbb{C},\mathbb{H},\mathbb{R}$ the group $U_{n}(\mathbb{F})$ defined in section 2  is one of the classical compact groups, namely, the unitary, the symplectic and the special orthogonal group. Its root system is of type $A_{n-1}$ when $\F=\C$, $C_n$ when $\F=\H$, $B_{r}$ when $\F=\R$ with $n=2r+1$, and $D_{r}$ when $\F=\R$ with $n=2r$. The Lie algebra $\mathfrak{U}_n(\F)$ of $U_n(\mathbb{F})$ is equal to $i\mathcal{P}_n(\mathbb{F})$. 

Let us consider the set $\mathfrak{t}_n=\{i \, \Omega_n(x) : x\in \mathbb{R}^{\tilde{n}} \}$. It is the Lie algebra of a maximal torus of $U_n(\mathbb{F})$. We define the linear forms $\epsilon_k: \mathfrak{t}_n\to \mathbb{R}$, by $\epsilon_k(i\,\Omega_n(x))=x_k$, $x\in \mathbb{R}^{\tilde{n}}$, $k=1,\cdots,\tilde{n}$. We equip $\mathfrak{U}_{n}(\mathbb{F})$ with the scalar product $\langle x,y\rangle=Tr(xy^{*})$  for $\mathbb{F}=\mathbb{C}$ and $\langle x,y\rangle=\frac{1}{2}Tr(xy^{*})$ for $\mathbb{F}=\mathbb{H},\mathbb{R}$.  For each group $U_n(\mathbb{F})$, we choose the following set $\Sigma$ of simple roots :

\begin{itemize}

\item when $\mathbb{F}=\mathbb{C}$, $\Sigma=\{\epsilon_i-\epsilon_{i+1}, \, i=1,\cdots ,n-1\},$

\item when $\mathbb{F}=\mathbb{H}$, $\Sigma=\{ 2\epsilon_n, \epsilon_i-\epsilon_{i+1},\, i=1,\cdots ,n-1\},$

\item when $\mathbb{F}=\mathbb{R}$ and $n=2r+1$, $\Sigma=\{\epsilon_r, \epsilon_i-\epsilon_{i+1}, \, i=1,\cdots ,r-1\},$

\item when $\mathbb{F}=\mathbb{R}$ and $n=2r$, $\{\epsilon_{r-1}+\epsilon_{r},\epsilon_i-\epsilon_{i+1}, \, i=1,\cdots ,r-1\}.$
\end{itemize}
If we identify $\mathbb{R}^{\tilde{n}}$ and $\mathfrak{t}_n$ by the map $x\in \mathbb{R}^{\tilde{n}}\mapsto i\Omega_n(x)\in \mathfrak{t}_n$,   and $\mathfrak{t}_n$ with $\mathfrak{t}_n^{*}$ by the scalar product, we get that
$x\in \mathbb{R}^{\tilde{n}}$ is identifiable with $ i\Omega_n(x)\in \mathfrak{t}_n$ or $\sum_{i=1}^{\tilde{n}}x_i\epsilon_i\in \mathfrak{t}_n^*$. Up to these identifications, the Weyl chamber corresponding to the chosen simple roots  is the set $\mathcal{C}_n$ defined in section 2, and the radial part of the matrix $U\Omega_n(x)U^*$ is $x$, considering either the definition of section 2 or the one of section 7. An integral point in $\CC_n$ is an element with  entries in $\Z$. Although we will not use this fact, one may notice that only  integral dominant weights occur in the representation of the group $U_n(\F)$. When $K=U_n(\mathbb{F})$, the corollary \ref{coroprinc} is equivalent to the following theorem.
\begin{theo}\label{coroprincbis} Let $\lambda$ and $\beta$ be two elements in the Weyl chamber $\mathcal{C}_n$ and an associated sequence of measures $(\nu_k)_{k\ge1}$ chosen as in corollary \ref{coroprinc}. Then $(\nu_k)_{k\ge1}$ converges to the law of the radial part of $\Omega_n(\lambda)+U\Omega_n(\beta)U^*$ where $U$ is a Haar distributed random variable in $U_n(\mathbb{F})$. 
\end{theo}

  We consider the subgroup  $H=\{U\in U_{n}(\mathbb{F}) : U_{in}=U_{ni}=\delta_{in}, i=1,\cdots,n\}$ and its Lie algebra $\{M\in \mathfrak{U}_{n}(\mathbb{F}) : M_{in}=M_{ni}=0, i=1, \cdots,n\}$. They are trivially identifiable with $ U_{n-1}(\mathbb{F})$ and $\mathfrak{U}_{n-1}(\mathbb{F})$. The orthogonal projection of a matrix $M$ of $\mathfrak{U}_n(\F)$ on this last subspace is equal, up to some zeros, to the main minor of order $n-1$ of $M$.  Thus, for the group $U_n(\mathbb{F})$ and the subgroup $H$, theorem \ref{theoprinc} gives:

  \begin{theo} \label{theoprincbis} Let $\lambda$ be in the Weyl chamber $\mathcal{C}_n$. Let us consider $M=U\Omega_n(\lambda)U^*$, where $U$ is a Haar distributed random variable in $U_{n}(\mathbb{F})$ and an associated sequence of measures  $(\mu_k)_{k\ge 1}$ on $\mathcal{C}_{n-1}$ as in Theorem \ref{theoprinc}.  Then  $(\mu_k)_{k\ge 1}$ converges to the law of the radial part of the main minor of order $n-1$ of $M$. 
  \end{theo}

\subsection{Relation between quaternionic and real odd case}\label{relquatreal}We have observed in the previous sections  that on the one hand the rank one perturbations are the same for $\F=\R$ and $n=2r+1$ as for $\F=\H$ and $n=r$, and on the other hand that the law of the radial part of the main minor of order $n-1$ of $U\Omega_n(\lambda)U^*$, with $U$ Haar distributed in $U_n(\H)$, is the same as the law of the radial part of the main minor of order $2n-1$ of $V\Omega_{2n+1}(\lambda)V^*$, with $V$ Haar distributed in $U_{2n+1}(\R)$. 
It is not a coincidence: identity (\ref{Fourierprojection}) shows that the convolution of invariant orbit measures or the projection of invariant measure depend only on the Weyl group of the groups and subgroups considered. At the price of some redundancy,  we have chosen to state explicitely our results in both cases for the convenience of the reading.

 \section{Tensor product and restriction multiplicities}
We want to compute the law of the sum or of the minors of  invariant random matrices.    By theorems \ref{coroprincbis} and \ref{theoprincbis}, it suffices to have a precise description of some appropriate tensor product and restriction multiplicities.   In group representation, these computations are a fundamental issue which have been studied for a long time. Recently the discovery of quantum group provided a new understanding of them.

The rank one perturbations that we introduce in section 3 are related to the tensor products $V_\lambda\otimes V_\gamma$, where $\lambda$ and $\gamma$ are dominant weights, $\gamma$ being proportional to $\epsilon_1$.  Using the theory of crystal graphs of Kashiwara, we obtain in section 9.1, explicit description of  these decompositions. Our results are surely not new and they  are contained, or maybe hidden, in more general ones (see for instance Berenstein and Zelevinski\cite{Berenstein}, Nakashima \cite{Nakashima}) but our descriptions present some advantages: they are quite simple and make interlacing conditions arise, which can be described, in the spirit of Fulmek and Krattenthaler \cite{Fulmek} for instance, in term of non intersecting paths.

In section 9.2, we recall the classical restriction multiplicities that we need for the computation of the law of the main minors.
  
\subsection{Tensor product multiplicities and crystal graphs} 
Let us recall some standard notations for crystal graphs (see, e.g,  Kashiwara \cite{Kashiwara}). As in the previous section we consider a compact connected Lie group  $K$ and its  Lie algebra $\mathfrak{k}$. Recall that the crystal graphs of the $\mathfrak{k}-$modules are oriented coloured graphs with colours $i\in I$. An arrow $a\overset{i}\rightarrow b$ means that $\tilde{f}_i(a)=b$ and $\tilde{e}_i(b)=a$ where $\tilde{e}_i$ and $\tilde{f}_i$ are the crystal graph operators. We denote $\Lambda_i$, $i=1,\ldots n$, the dual basis of the coroots. For a $\mathfrak{k}-$module $V$ and its crystal graph $B$, the weight of a vertex $b\in B$ is defined by $wt(b)=\sum_I(\varphi_i(b)-\varepsilon_i(b))\Lambda_i$, where $\varphi_i(b)=\max \{n\ge 0 : \tilde{f}^n_i(b)\in B \}$ and $\varepsilon_i(b)=\max \{n\ge 0 : \tilde{e}^n_i(b)\in B\}$, $i\in I$.  For each dominant weight $\lambda$ we denote by $B(\lambda)$ the crystal graph of  the irreducible $\mathfrak{k}-$module  $V_\lambda$ with highest weight $\lambda$ and by $u_\lambda$ the highest weight vertex. We recall the proposition $4.2$ of \cite{Kashiwara}.

\begin{prop} \label{tensKash} Let $\lambda$ and $\mu$ be two dominant weights and $B(\mu)$ the crystal graph of $V_\mu$. Then 
$$V_\lambda\otimes V_\mu=\oplus V_{\lambda+wt(b)},$$
where the sum ranges over $b\in B(\mu)$ such that $\varepsilon_{i}(b)\le \langle h_i,\lambda\rangle$ (or equivalently $\varepsilon_i(u_\lambda\otimes b)=0$) for every $i\in I$. 
\end{prop} 
We now consider $K=U_n(\F)$ and we describe the tensor products $V_{\lambda}\otimes V_{a\epsilon_1}$ that we are interested in. For this we use the description of the crystal graphs for classical Lie algebras given by Kashiwara and Nakashima in \cite{KashiwaraNakashima}. In the following, we write $x$ indifferently for $(x_1,\cdots,x_{\tilde{n}})\in \mathbb{R}^{\tilde{n}}$ and $\sum_{i=1}^{\tilde{n}} x_i\epsilon_i$. Notice that $\epsilon_1$ is the highest weight of the standard representation.

\subsubsection{Tensor product of representations for the type $A_{n-1}$} This case  is classical and known as Pieri's formula (see Fulton \cite{Fulton}). But it will help the reader to first see the method we use in this simple example. In the type $A_{n-1}$, the simple coroots are $h_i=\epsilon_i-\epsilon_{i+1}$, $1 \leq i \leq n-1$. The crystal graph of $V_{\epsilon_{1}}$ is, see \cite{Kashiwara},
$$B(\epsilon_{1}): 1\overset{1}\rightarrow 2\overset{2}\rightarrow \cdots \overset{n-1}\rightarrow n.$$
Here the weight of $i$ is $\epsilon_i$, $i=1, \cdots,n$. We use the usual order on $\{1, \cdots,n\}$.
Let $m$ be an integer. Theorem $3.4.2$ of \cite{KashiwaraNakashima} claims in particular that $$B(m\epsilon_1)=\{b_m\otimes \cdots\otimes b_1 \in B(\epsilon_1)^{\otimes m} :  b_{k+1}\ge b_k\}.$$ Let $\lambda$ be a dominant weight. Let us describe the decomposition
of the tensor product $V_\lambda\otimes V_{m \epsilon_1}$.
In proposition \ref{tensKash},  the sum ranges over all elements $b_m\otimes \cdots\otimes b_1 \in B(m \epsilon_1)$ such that, for $1 \leq i \leq n$,  $\varepsilon_{i}(u_\lambda\otimes b_m\otimes \cdots\otimes b_1)=0$, which is equivalent to say that $\varepsilon_{i}(b_k)\le \langle h_i,\lambda+wt(b_{k+1})+ \cdots +wt(b_{m})\rangle$ for $1 \leq k \leq m$. 
When $b\in B(\epsilon_1)$, either $b=i+1$ and  $\varepsilon_i(b)=1=-\langle h_i,wt(b)\rangle$, or  $\varepsilon_i(b)=0 \leq \langle h_i,wt(b)\rangle$. Thus we have \begin{align}\varepsilon_{i}(b)\le \langle h_i,\lambda\rangle \Leftrightarrow 0\le   \langle h_i,\lambda+wt(b)\rangle.\label{equivalA} \end{align} 
So, in the considered decomposition, the sum ranges over all elements $b_m\otimes \cdots\otimes b_1 \in B(\epsilon_1)^{\otimes m}$ satisfying the following conditions for every  $k\in\{1, \cdots,m\},i \in \{1,\cdots,n\}$, \begin{align} \left\{ \begin{array}{l} b_{k+1}\ge  b_{k},    \\   0\le \langle h_i,\lambda+wt(b_{m})+ \cdots +wt(b_k)\rangle. \label{dominantA}\end{array}\right. \end{align}
We draw on figure \ref{decompositionA}  the functions \begin{eqnarray*}\label{fonctionmu}k\mapsto \mu_i(k)=\langle \epsilon_i, \lambda + wt(b_{m})+ \cdots +wt(b_{m-k+1})\rangle .\end{eqnarray*} At each $k$, one and only one of the functions $\mu_1,\cdots,\mu_n$ increases by one unit. Moreover the $i^{th}$ curve cannot   increase if the $(i+1)^{th}$  has not because $b_m\otimes \cdots\otimes b_1$ is an element of $B(m\epsilon_1)$. The curves cannot cross each other since  $0 \le \langle h_i,\lambda+wt(b_{m})+\cdots+wt(b_k)\rangle$. Therefore we see that
the map $b_{m}\otimes \cdots\otimes b_1\mapsto \beta\in \mathbb{Z}^n $, with $\beta_i=\langle \epsilon_i,\lambda+wt(b_m)+ \cdots +wt(b_1) \rangle$,   $i=1, \cdots,n$,  is a bijection from $\{b\in B(m \epsilon_1) : b \textrm{ satisfies conditions  (\ref{dominantA})}\}$ to $\{  \beta\in \mathbb{Z}^n: \beta\succeq \lambda, \, \sum_{i}(\beta_i-\lambda_i)=m\}$. So we get the Pieri's formula (notice that the multiplicity are equal to one):
\begin{prop} \label{tensrepA} Let $\lambda,\gamma \in \mathbb{Z}^n$ such that $\lambda_1\ge \cdots\ge \lambda_n$ and  $\gamma=(m,0,\cdots,0)$, $m\in \mathbb{N}$. Then
$$V_{\lambda}\otimes V_{\gamma}=\oplus_{\beta}V_{\beta}$$
where the sum is over the integral dominant weights  such that $ \beta\succeq \lambda$, and $m=\sum_{i=1}^n  (\beta_i-\lambda_i)$. 
\end{prop}
\begin{figure}[h!]  
\begin{pspicture}(2,6)(3,-0.5)
\psline{->}(0,-0.1)(0,5.25)\psline{->}(-0.1,0)(5,0)\psline{->}(4,-0.1)(4,5.25)
\psline(0,0.25)(1.5,1.75)(4,1.75)(5,1.75)
\psline(0,2)(1.5,2)(2.5,3)(4,3)(5,3)
\psline(0,3.5)(2.5,3.5)(4,5)(5,5)
\psline{<->}(0,-0.25)(4,-0.25)
\put(4.2,5.2){$\beta_{1}$} \put(4.2,3.2){$\beta_{2}$} \put(4.2,1.95){$\beta_{3}$}
\put(5.1,5.1){$\mu_1 $} \put(5.1,3.1){$\mu_2 $} \put(5.1,1.85){$\mu_3$}\put(5.2,0){$k$}
\put(-0.45,3.6){$\lambda_{1}$} \put(-0.45,2.1){$\lambda_{2}$} \put(-0.45,0.35){$\lambda_{3}$}
\put(2,-0.5){$m$}
\end{pspicture}
  \caption{Irreducible decomposition of $V_\lambda\otimes V_{m \epsilon_1}$ for the type $A_{2}$}
  \label{decompositionA}
\end{figure} 

\subsubsection{Tensor product of representations for the type $C_n$} 
The  simple coroots are now $h_i=\epsilon_i-\epsilon_{i+1}$, $1 \leqÊi \leqÊn-1$,  $h_n=\epsilon_n$, and the crystal graph of $V_{\epsilon_{1}}$ is
$$B(\epsilon_1): 1\overset{1}\rightarrow  \cdots \overset{n-1}\rightarrow n
\overset{n}\rightarrow \bar{n} \overset{n-1}\rightarrow
\cdots\overset{1}\rightarrow \bar{1}.$$
Here $i$ and $\bar{i}$ have respective weight $\epsilon_i$ and $-\epsilon_i$. We define the order $\le$ on $B(\epsilon_1)$  by $1\le \cdots\le n\le \overline{n}\le \cdots\le \overline{1}$.
By theorem $4.5.1$ of \cite{KashiwaraNakashima}, if $m \in  \N$,
 $$B(m \epsilon_1)=\{b_m\otimes \cdots\otimes b_1 \in B(\epsilon_1)^{\otimes m} : b_{k+1}\ge  b_{k}\}.$$ 
Let $\lambda$ be a dominant weight. As above is it easy to see that equivalence (\ref{equivalA}) holds. Therefore, by proposition \ref{tensKash},
 the sum ranges over all elements $b_m\otimes \cdots\otimes b_1 \in B(\epsilon_1)^{\otimes m}$ satisfying the following conditions for $1\leq k\leq m, 1 \leq i \leq n$, \begin{align} \left\{ \begin{array}{l} b_{k+1}\ge  b_{k},    \\   0\le \langle h_i,\lambda+wt(b_{m})+ \cdots +wt(b_k)\rangle .\label{dominantC}\end{array}\right. \end{align}
The function $b_{m}\otimes \cdots\otimes b_1\mapsto (\beta,c)\in \mathbb{N}^n\times \mathbb{N}^n$, where for $i=1,\cdots,n$ $$\beta_i=\langle \epsilon_i,\lambda+wt(b_m)+\cdots+wt(b_1) \rangle$$ and $$c_i=\min\{\langle  \epsilon_i,\lambda+wt(b_m)+\cdots+wt(b_k)  \rangle, 1 \leqÊk \leq m\},$$  is a bijection from $\{b\in B(m\Lambda_1) : b \textrm{ satisfies conditions  (\ref{dominantC})}\}$ to $\{  (\beta,c)\in \mathbb{N}^n\times \mathbb{N}^n: \lambda\succeq c, \,\beta\succeq c, \,\sum_{i}(\lambda_i-c_i+\beta_i-c_i)=m\}$. Look at Figure \ref{decompositionC} to be convinced of the bijection. The $i^{th}$ curve cannot decrease (resp. increase) if the $(i-1)^{th}$ (resp.$(i+1)^{th})$ has not since $b_m\otimes \cdots\otimes b_1$ is an element of $B(a\epsilon_1)$. Moreover the curves remain nonnegative and cannot cross each other since $0 \le \langle h_i,\lambda+wt(b_{k})+\cdots+wt(b_m)\rangle$. So we get the following proposition. 
\begin{prop} \label{tensrepC} Let $\lambda,\gamma \in \mathbb{N}^n$ be such that $\lambda_1\ge \cdots\ge \lambda_n$, and $\gamma=(m,0,\cdots,0)$, $m\in \mathbb{N}$. Then
$$V_{\lambda}\otimes V_{\gamma}=\oplus_{\beta}M_{\lambda,\gamma}(\beta)  V_{\beta}$$
where the sum is over all  $\beta\in \mathbb{N}^n$ satisfying $\beta_1\ge \cdots\ge \beta_n$ such that there exists  $c=(c_{1},\cdots,c_{n})\in \mathbb{N}^{n}$ which verifies $\lambda\succeq c $, $\beta\succeq c$ and $\sum_{i=1}^{n}(\lambda_i-c_i + \beta_i-c_i)=m$. In addition, the  multiplicity $M_{\lambda,\gamma}(\beta)$  of the irreducible module with highest weight $\beta$ is  the number of $c\in\mathbb{N}^{n}$ satisfying these relations. 
\end{prop}
\begin{figure}[h!]  
\begin{pspicture}(2,4)(3,-0.5)
\psline{->}(0,-0.1)(0,3.75)\psline{->}(-0.1,0)(5,0)\psline{->}(4.25,-0.1)(4.25,3.75)\psline{->}(2.125,-0.1)(2.125,3.75)
\psline(0,1)(1.75,1)(2.5,0.25)(2.75,0.5)(5,0.5)\psline(0,2.25)(0.75,2.25)(1.75,1.25)(2.75,1.25)(3.25,1.75)(5,1.75)\psline(0,3.25)(0.75,2.5)(3.25,2.5)(4.25,3.5)(5,3.5)\psline{<->}(0,-0.25)(4.25,-0.25)
\put(4.4,3.65){$\beta_{1}$} \put(4.4,1.9){$\beta_{2}$} \put(4.4,0.65){$\beta_{3}$}
\put(5.2,3.65){$\mu_1  $} \put(5.2,1.9){$\mu_2 $} \put(5.2,0.65){$\mu_3 $}\put(5.2,0){$k$}
\put(1.8,2.65){$c_{1}$} \put(1.8,1.4){$c_{2}$} \put(1.8,0.4){$c_{3}$}
\put(-0.4,3.25){$\lambda_{1}$} \put(-0.4,2.25){$\lambda_{2}$} \put(-0.4,1){$\lambda_{3}$}
\put(2,-0.5){$m$}
\psline[linestyle=dashed,linecolor=gray](0,2.5)(4.25,2.5)
\psline[linestyle=dashed,linecolor=gray](0,1.25)(4.25,1.25)
\psline[linestyle=dashed,linecolor=gray](0,0.25)(4.25,0.25)
\end{pspicture}
  \caption{Irreducible decomposition of $V_\lambda\otimes V_{m \epsilon_1}$ for the type $C_{3}$}
  \label{decompositionC}
\end{figure}
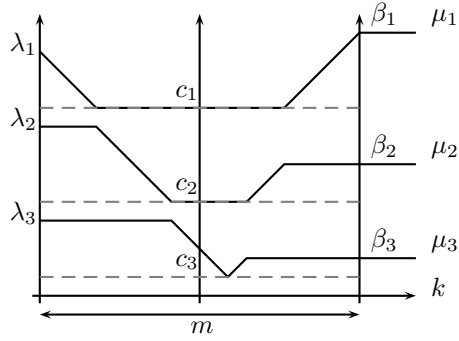 
 We invite the reader to compare this figure with figure \ref{perturbationC}: vectors $R_i$ and $R_{i+1}$ (black discs) satisfy the same interlacing conditions as the highest weights $\lambda$ and $\mu$, and the white discs verify the same interlacing conditions as $c$. 
\medskip

\subsubsection{Tensor product of representations for type $B_r$}
The  coroots of the simple roots are $h_i=\epsilon_i-\epsilon_{i+1}$, $i=1,\cdots,r-1$,  $h_r=2\epsilon_r$, the crystal graph of $V_{\epsilon_{1}}$ is
$$B(\epsilon_{1}): 1\overset{1}\rightarrow  \cdots \overset{r-1}\rightarrow r
\overset{r}\rightarrow 0 \overset{r}\rightarrow \overline{r} \overset{r-1}\rightarrow
\cdots\overset{1}\rightarrow \overline{1},$$
where $i$, $\overline{i}$ and $0$ have respective weight $\epsilon_i$, $-\epsilon_i$ and $0$ for $i=1,\cdots,r$. We define an order on $B(\epsilon_{1}) $ by $1\le \cdots\le r\le 0\le  \overline{r}\le \cdots\le \overline{1}$.
By theorem $5.7.1$ of \cite{KashiwaraNakashima}, $$B(m \epsilon_1)=\{b_m\otimes\cdots\otimes b_1 \in B(\epsilon_1)^{\otimes m} : b_{k+1}\ge  b_{k}, \, b_{k+1}\otimes b_{k} \ne 0\otimes 0\}.$$ Let $\lambda$ be an integral  dominant weight. As for the type $C_n$, in the decomposition
of  $V_\lambda\otimes V_{m \epsilon_1}$
 the sum ranges over the $b_m\otimes \cdots \otimes b_1 \in B(m \epsilon_1)$ such that  $\varepsilon_{i}(b_k)\le \langle h_i,\lambda+wt(b_{k+1})+\cdots+wt(b_m)\rangle$ for $1Ê\leq k \leq m, 1 \leq i \leq r$.
Let $b\in B(\epsilon_1)$. For $i\le r-1$, $\langle h_i,wt(b)\rangle =-1$ if $b=i+1$ or $b=\overline{i}$. Moreover  $\langle h_r,wt(b)\rangle =-2$  if $b=\overline{r}$. In every other cases $\langle h_i,wt(b)\rangle$ is positive. Thus one easily shows that
\begin{align*}    \varepsilon_{i}(b)\le \langle h_i,\lambda\rangle \Leftrightarrow \left\{ \begin{array}{l} \big (b\ne 0 \textrm{ and }  0\le \langle h_i,\lambda+wt(b)\rangle \big)\\
\textrm{  or } \big( b=0 \textrm{ and }  \langle h_r,\lambda\rangle \ge 1\big).\end{array}\right. \end{align*}
So, in the decomposition considered, the sum ranges over all elements $b_m\otimes \cdots\otimes b_1 \in B(\epsilon_1)^{\otimes m}$ satisfying for every  $(k,i)\in\{1,\cdots,m\} \times\{1,\cdots,r\}$
\begin{align} \left\{ \begin{array}{l}  b_{k+1}\ge  b_{k}, \quad b_{k+1}\otimes b_{k}\ne 0\otimes 0    \\  0\le \langle h_i,\lambda+wt(b_{m})+...+wt(b_k)\rangle \\
  1\le \langle h_r,\lambda+wt(b_{m})+\cdots+wt(b_k)\rangle \textrm{ if $b_k=0$}. \end{array}\right. \end{align}
Thus we get the following  proposition.
\begin{prop} \label{tensrepB}  Let $\lambda,\gamma \in \mathbb{N}^{r}$ be such that $\lambda_1\ge \cdots\ge \lambda_r$ and $\gamma=(m,0,\cdots,0)$, $a\in \mathbb{N}$. Then
$$V_{\lambda}\otimes V_{\gamma}=\oplus_{\beta}M_{\lambda,\gamma}(\beta) V_{\beta}$$
where the sum is over all  $\beta\in \mathbb{N}^r$ such that $\beta_1\ge \cdots\ge \beta_r$ such that there exists an integer $s\in \{0,1\}$ and $c\in \mathbb{N}^{r}$ which verifies $\lambda\succeq c $, $\beta\succeq c$ and $\sum_{i=1}^{r}(\lambda_i-c_i+ \beta_i-c_i)+s=m$, $s$ being equal to $0$ if $c_r=0$. In addition, the  multiplicity $M_{\lambda,\gamma}(\beta)$ of the irreducible module with highest weight $\beta $ is the the number of $(c,s)\in\mathbb{N}^{r}\times \{0,1\}$ satisfying these relations. 
\end{prop}
\subsubsection{Tensor product of representations for type $D_r$}
The simple coroots  are $h_i=\epsilon_i-\epsilon_{i+1}$, $i=1,\cdots,r-1$, and  $h_r=\epsilon_r+\epsilon_{r-1}$,  the crystal graph of $V_{\epsilon_{1}}$ is

\begin{pspicture}(0,-1)(0,1.5)
\put(0,0.5){$B(\epsilon_{1}):  \begin{array}{ccccc}
  1\overset{1}\rightarrow  \cdots \overset{r-3}\rightarrow r-2\overset{r-2}\rightarrow & r-1 &   & \overline{r-1} &   \overset{r-2}\rightarrow \overline{r-2} \overset{r-3}\rightarrow\cdots \overset{1} \rightarrow \overline{1} \\
\end{array}.$}
\put(5.3,0.95){$\nearrow$}\put(5.3,1.1){\scriptsize{$r$}}\put(6,1.2){$\overline{r}$}\put(6.5,0.95){$\searrow$}
\put(6.8,1.1){\scriptsize{$r-1$}}
\put(5.3,-0.1){$\searrow $}\put(4.75,-0.15){\scriptsize{$r-1$}}\put(6,-0.25){$r$}\put(6.5,-0.1){$\nearrow $}
\put(6.8,-0.15){\scriptsize{$r$}}
\end{pspicture}
Here $i$ and $\overline{i}$  have respective weight $\epsilon_i$ and $-\epsilon_i$ , $i=1,\cdots,r$. We define a partial order $\le$ on $B(\epsilon_1)$ by $1\le\cdots\le r-1\le \begin{array}{c}
  \overline{r} \\
 r \\
\end{array}\le \overline{r-1}\le \cdots\le \overline{1}$.
For $m\in \N$  theorem $6.7.1$ of \cite{KashiwaraNakashima} states that,  $$B(m \epsilon_1)=\{b_m\otimes\cdots\otimes b_1\in B(\epsilon_1)^{\otimes m} : b_{k+1}\le  b_{k} \}.$$ Let $\lambda$ be a dominant weight such that $\langle \epsilon_r,\lambda\rangle \in \N$. For $b\in B(\epsilon_1)$, the same considerations as for the types $A_{n-1}$ and $C_n$ imply  equivalence (\ref{equivalA}). So that we get proposition \ref{tensrepD}, which is illustrated by figure \ref{decompositionD}. We invite the reader to compare with figure \ref{perturbationD}.
\begin{prop} \label{tensrepD} Let $\lambda,\gamma \in \mathbb{N}^{r}$ be such that $\lambda_1\ge \cdots\ge \vert \lambda_r\vert $, and $\gamma=(m,0,\cdots,0)$, $m\in \mathbb{N}$. Then
$$V_{\lambda}\otimes V_{\gamma}= \oplus_{\beta}M_{\lambda,\gamma}(\beta)  V_{\beta}$$
where the sum is over all  $\beta\in \mathbb{N}^r$ satisfying $\beta_1\ge \cdots\ge \beta_r$ such that there exists  $c\in \mathbb{N}^{r-1}$ which verifiy $\lambda\succeq c $, $\beta\succeq c$, $\max(\vert \lambda_r \vert, \vert \beta_r \vert)\le c_{r-1} $ and $\sum_{k=1}^{r-1}(\lambda_k-c_k+ \beta_k-c_k)+ \vert\lambda_r-\mu_r \vert=m$. In addition, the  multiplicity $M_{\lambda,\gamma}(\beta)$ of the irreducible module with highest weight $\beta $ is the   number of $c\in\mathbb{N}^{r-1}$ satisfying these relations. 
\end{prop}

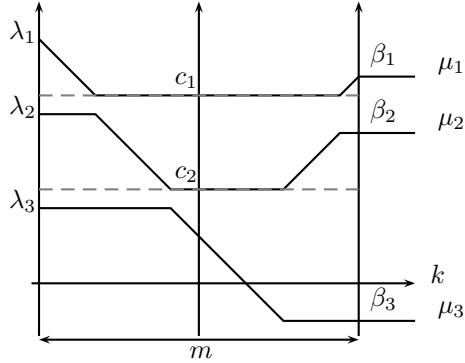
\begin{figure}[h!]
\begin{pspicture}(2.5,4)(3,-1)
\psline{->}(0,-0.7)(0,3.75)\psline{->}(-0.1,0)(5,0)\psline{->}(4.25,-0.7)(4.25,3.75)\psline{->}(2.125,-0.7)(2.125,3.75)
\psline(0,1)(1.75,1)(3.25,-0.5)(5,-0.5)
\psline(0,2.25)(0.75,2.25)(1.75,1.25)(3.25,1.25)(4,2)(5,2)
\psline(0,3.25)(0.75,2.5)(4,2.5)(4.25,2.75)(5,2.75)\psline{<->}(0,-0.75)(4.25,-0.75)
\put(4.4,2.95){$\beta_{1}$} \put(4.4,2.2){$\beta_{2}$} \put(4.4,-0.3){$\beta_{3}$}
\put(5.3,2.85){$\mu_1$} \put(5.3,2.1){$\mu_2$} \put(5.3,-0.4){$\mu_3$}\put(5.2,0){$k$}
\put(1.8,2.65){$c_{1}$}\put(1.8,1.4){$c_{2}$} 
\put(-0.4,3.25){$\lambda_{1}$} \put(-0.4,2.25){$\lambda_{2}$} \put(-0.4,1){$\lambda_{3}$}
\put(2,-1){$m$}
\psline[linestyle=dashed,linecolor=gray](0,2.5)(4.25,2.5)
\psline[linestyle=dashed,linecolor=gray](0,1.25)(4.25,1.25)
\end{pspicture}
  \caption{Irreducible decomposition of $V_\lambda\otimes V_{m\epsilon_1}$ for the type $D_{3}$}
  \label{decompositionD}
\end{figure}

\subsection{Classical restriction multiplicities}
For $\F=\R,\C,\H$, the branching rules when $K=U_n(\mathbb{F})$ and $H=U_{n-1}(\mathbb{F})$,  are well known (see for instance Knapp \cite{Knapp}). Let us recall them. We add a subscript $\mathbb{Z}$ to the Gelfand Tetlin polytopes $GT_n(\lambda)$  to designate the subset of elements with integer entries.
\begin{prop} \label{restriction} Let $\lambda$ be an integral point of $\mathcal{C}_n$. Let $V_\lambda$ be an irreducible module  with highest weight $\lambda$. The irreducible decomposition (\ref{branching}) when $K=U_n(\mathbb{F})$ and $H=U_{n-1}(\mathbb{F})$ is  the following one:
$$V_\lambda =\oplus_{\beta}m^{\lambda}_{U_{n-1}}(\beta)V^{U_{n-1}}_{\beta}, $$
where the sum is over all $\beta $ such that there exists $x\in GT_{n,\mathbb{Z}}(\lambda)$ such that $x^{(n-1)}= \beta $. Moreover for $\mathbb{F}=\mathbb{C},\mathbb{R}$, $m^{\lambda}_{U_{n-1}}(\beta)=1$  and for $\mathbb{F}=\mathbb{H}$,  $m^{\lambda}_{U_{n-1}}(\beta)$ is the number of $c\in \mathbb{N}^n$ for which  there exists  $x\in GT_{n,\mathbb{Z}}(\lambda)$ with $x^{(n-1)}= \beta $ and $x^{(n-\frac{1}{2})}=c$.
\end{prop}

\section{Asymptotic multiplicities and limit measures}
In this section, we  prove lemma \ref{lemGT} and proposition \ref{pertutheta}. 
\subsection{Proof of lemma \ref{lemGT} }\label{prooflemGT}
We have recalled in proposition \ref{restriction} the branching rules in the case when $K=U_{n}(\mathbb{F})$ and $H=U_{n-1}(\mathbb{F})$. Let us consider the chain of subgroups $U_n(\mathbb{F})\supset \cdots\supset U_1(\mathbb{F})$ and the corresponding successive restrictions. If we compare the successive branching rules with the definition of the Gelfand--Tsetlin polytopes $GT_n(\lambda)$  for $\lambda$ an integer point in $\mathcal{C}_n$, we get the famous result that the number of integer points in $GT_n(\lambda)$ is the dimension of the irreducible $\mathfrak{U}_n(\mathbb{F})-$module with highest weight $\lambda$. Actually this is the reason why Gelfand--Tsetlin polytopes have been introduced \cite{Gelfand}. The dimension formula (\ref{formuladim}) implies the following lemma. Let $\epsilon$ be equal to $1$ if $n\notin 2\mathbb{N}$ and $0$ otherwise.
\begin{lem} \label{dimGT}  Let $\lambda$ be an integer point in $\mathcal{C}_n$. The number of points in $GT_{n,\mathbb{Z}}(\lambda)$, denoted $\Card GT_{n,\Z}(\lambda)$, is equal to:\begin{itemize}
\item when $\mathbb{F}=\mathbb{C}$,
\begin{align*} 
\prod_{1\le i<j\le n}\frac{\lambda_i-\lambda_j+j-i}{j-i}, \end{align*}
\item when $\mathbb{F}=\mathbb{H}$,
\begin{align*}
\prod_{1\le i<j\le n}\frac{(\lambda_i-\lambda_j+j-i)(\lambda_i+\lambda_j+2n+2-j-i)}{(j-i)(2n+2-j-i)}\prod_{ i=1}^{ n}\frac{\lambda_i +n+1-i}{n+1-i}, \end{align*}
\item when $\mathbb{F}=\mathbb{R}$,
\begin{align*}
\prod_{1\le i<j\le \tilde{n}}\frac{(\lambda_i-\lambda_j+j-i)(\lambda_i+\lambda_j+2\tilde{n}+\epsilon-j-i)}{(j-i)(2\tilde{n}+\epsilon-j-i)}\prod_{ i=1}^{ \tilde{n}}\Big[\frac{\lambda_i +\tilde{n}+\frac{1}{2}-i}{\tilde{n}+\frac{1}{2}-i}\Big]^\epsilon.\\
\end{align*} 
\end{itemize}
\end{lem}

Before writing the proof of the lemma \ref{lemGT} let us state the following lemma. Recall that $\mu_\lambda$ is the image of the uniform measure on $GT_n(\lambda)$  by the map $x\in GT_n(\lambda)\mapsto x^{(n-1)}$, and  $(\varepsilon_k)_{k\ge 1}$  converges to $0$.
\begin{lem} \label{lemlemGT}
Let $\lambda$ be in the Weyl chamber $\mathcal{C}_n$.  Let us consider a sequence $(\lambda_k)_{k\ge 1}$  of integer points in $\CC_n$ such that  $\varepsilon_k\lambda_k$ converges to $\lambda$, as $k$ goes to infinity, and the associated sequence of measures $(\mu_k)_{k\ge 0}$ defined as in Theorem \ref{theoprinc} for $K=U_n(\mathbb{F})$ and $H=U_{n-1}(\mathbb{F})$. Then $(\mu_k)_{k\ge 1}$ converges to the measure  $\mu_\lambda$.
\end{lem} 
\proof We use the multiplicity $m_{U_{n-1}}^{\lambda_k}(\beta)$ defined in proposition \ref{restriction}. Since the dimension of the irreducible $\mathfrak{U}_n(\F)$-module with highest weight $\gamma$ is given by the number of integer points in $GT_n(\gamma)$, we obtain that $\mu_k$ is equal to  
\begin{align}\label{muk} \sum_{\beta}\frac{\Card  GT_{n-1,\Z}(\beta)}{\Card GT_{n,\Z}(\lambda_k)}m_{U_{n-1}}^{\lambda_k}(\beta)\delta_{\varepsilon_k\beta}.\end{align}
Comparing the definition of Gelfand--Tsetlin polytopes with the branching rules given in proposition \ref{restriction}, we get that the measure $\mu_k$ is the image by the map $x\in GT_n(\varepsilon_k\lambda_k)\mapsto x^{(n-1)}$ of the measure $\frac{1}{\dim(\lambda_k)}\sum_{x \in GT_{n,\mathbb{Z}}(\lambda_k)}\delta_{\varepsilon_k x}$. This last measure converges to the uniform measure on $GT_n(\lambda)$. Thus $\mu_k$ converges to  $\mu_\lambda$.   \qed\medskip

{\sc Proof of lemma \ref{lemGT} :} Let $\lambda\in \mathcal{C}_n$ and $U\in U_n(\mathbb{F})$ a Haar distributed random variable. We choose a sequence of measures $(\mu_k)_{k\ge 1}$  on $\mathcal{C}_n$ as in lemma \ref{lemlemGT} which claims that $(\mu_k)_{k\ge 1}$ converges to the law of the radial part of the main minor of order $n-1$ of $U\Omega_n(\lambda)U^*$. Lemma \ref{lemlemGT} implies that this law is $\mu_\lambda$.    \qed

\subsection{Proof of proposition \ref{pertutheta}} \label{proofpertutheta} The following lemma states the connection between the set $\mathcal{E}(\lambda,\theta)$ defined in section 4 and irreducible decomposition of tensor products of representations studied in section 9. In every case but the real odd one, we denote $\mathcal{E}_{\mathbb{Z}}(\lambda,\theta)$   the subset of $\mathcal{E}(\lambda,\theta)$ whose elements have components in $\mathbb{Z}$. In the case where $ \mathbb{F}=\mathbb{R}$ and $n=2r+1$, we let\begin{align*}\mathcal{E}_\mathbb{Z}(\lambda,\theta)=\{(\beta,z,x,&s)  \in\mathbb{N}^r\times  \mathbb{N}^r\times GT_{n,\mathbb{Z}}\times \{0,1\}: \lambda,\beta\succeq z , \\ & \sum_{i=1}^{r}(\lambda_i+ \beta_i-2z_i)+s=\theta,\, x\in GT_n(\beta), \, s=0 \textrm{ if } z_r=0\}.\end{align*}

\begin{lem} \label{dimE}
Let $\lambda$ and $\gamma=(a,0,...,0)$ be integer points of $\mathcal{C}_n$. Then the number of points in $\mathcal{E}_\mathbb{Z}(\lambda,a)$ is equal to $\dim(\lambda)\dim(\gamma)$.
\end{lem}
\proof  We compare the conditions satisfied by the component $\beta$ of the elements of $\mathcal{E}_\mathbb{Z}(\lambda,a)$ with those satisfied by the weights appearing in the description of the irreducible decomposition of the tensor product $V_\lambda\otimes V_\gamma$ given in propositions \ref{tensrepA} to \ref{tensrepD}. Recalling that $\Card GT_{n,\Z}(\beta)$ is equal to $\dim(V_\beta)$ we get that $\Card\mathcal{E}_\mathbb{Z}(\lambda,a)=\sum_\beta \dim(V_\beta)$, where the sum ranges over the dominant weights $\beta$ (with their multiplicity) appearing in the irreducible decomposition of  $V_\lambda\otimes V_\gamma$.  Thus $\Card\mathcal{E}_\mathbb{Z}(\lambda,a)=\dim(V_\gamma\otimes V_\gamma)=\dim(\lambda)\dim (\gamma)$. \qed

\begin{lem} \label{tensor}
Let $\lambda$ be in the Weyl chamber $\mathcal{C}_n$ and $\theta>0$. Let us consider two sequences $(\lambda_k)_{k\ge 1}$  and $(\gamma_k)_{k\ge 1}$ of integer points in $\CC_n$  such that $\gamma_k$ can be written as $(a_k,0,...,0)$. We suppose that $\varepsilon_k\lambda_k$ converges to $\lambda$ and  $\varepsilon_ka_k$ converges to $\theta$, as $k$ goes to infinity. Then the associated sequence of measures  $(\nu_k)_{k\ge 1}$ given in corollary \ref{coroprinc} for $K=U_n(\mathbb{F})$  converges to the measure $\nu_{\lambda,\theta}$.
\end{lem}
\proof   The measure $\nu_k$ is   the measure $$\sum_{\beta}\frac{\dim(\beta)}{\dim(\lambda_{k})\dim(\gamma_{k})}M_{\lambda_{k},\gamma_k}(\beta)\delta_{\varepsilon_k\beta},$$ where $M_{\lambda_{k},\gamma_k}(\beta)$ is the multiplicity of the highest weight $\beta$ in the irreducible decomposition of $V_{\lambda_k}\otimes V_{\gamma_k}$. The description of this  irreducible decomposition  given in section 9, from proposition \ref{tensrepA} to \ref{tensrepD}, and the fact that $\dim(\beta)=\Card GT_{n,\Z}(\beta)$ show that $\nu_k$ is the image by the projection on the component $\beta$ of the  probability $$\frac{1}{\dim(\lambda_k)\dim(\gamma_k)}\sum_{x \in \mathcal{E}_\mathbb{Z}(\lambda_k,a_k)}\delta_{\varepsilon_k x},$$ 
which proves the proposition. \qed \medskip

{\sc Proof of proposition \ref{pertutheta} : }Let $\lambda\in \mathcal{C}_n$, $\theta\in \mathbb{R}_+^*$ and  choose a sequence of measures $(\nu_k)_{k\ge 1}$ on $\mathcal{C}_n$ as in lemma \ref{tensor}. Then $(\nu_k)_{k\ge 1}$ converges to the law of the radial part of $\Omega_n(\lambda)+U\Omega_n(\theta)U^*$ where $U$ has a Haar distribution. Lemma \ref{tensor} implies that this law is $\nu_{\lambda,\theta}$.  \qed
 \section{Concluding remarks}

\subsection{Random processes with values in $GT_n$} Let $M=(M_t)_{t\ge 0}$ be a standard Brownian motion in $\mathcal{P}_n(\mathbb{C})$. Then the minor process $X(M)=\big(X(M_t)\big)_{t\ge 0}$  is generally not a Markov process. For instance for $\mathbb{F}=\mathbb{C}$, the only cases when $X(M)$ is a Markov process are for $n=1$ and $n=2$. Actually a Brownian motion in $\mathcal{P}_n(\mathbb{C})$ can be obtained as a limit, in a certain sense, of a quantum random walk (see Biane \cite{Biane} for $\mathcal{P}_2(\mathbb{C})$). The fact that the minor process $X(M)$ is not Markovian has to be related to the fact  that for $n\ge 3$, the ''complete system of observables'' in the space of any representation defined by Zhelobenko in chap. X.67  of \cite{Zhelobenko}, is not stable by the Markovian operator of the quantum random walk.

\subsection{Rank one perturbation on classical complex Lie groups}
Klyachko showed in \cite{Klyachko} that the convolution of  biinvariant measures on the complexification $G$ of  the compact group $K$, is deduced from the convolution of invariant measures on $K$. His result is an hyperbolic version of the so called wrapping map introduced by Dooley and Wildberger \cite{DooleyWildberger}.  Using this we can show that the radial part of a Brownian motion on the symmetric space $G/K$ can be approached by an interlaced process.
\subsection{Rank one perturbation on $\bf{U_n(\mathbb{F})}$}
Let us say a word about some other interesting rank one perturbations having invariance properties that we can find in literature. For instance, Diaconis and Shahshahani \cite{Diaconis1}, \cite{Diaconis2}, followed by Porod \cite{Porod1}, \cite{Porod2} and Rosenthal \cite{Rosenthal},  studied  specific random walks on $U_n(\mathbb{F})$, $\mathbb{F}=\mathbb{R},\mathbb{C}, \mathbb{H}$, whose increments are some random rotations, in order to approximate the Haar measure on $U_n(\mathbb{F})$. The wrapping map introduced in \cite{DooleyWildberger} makes a link between these rank one perturbations and those that we studied in this paper.

\end{document}